\documentclass[11 pt, reqno]{amsart}

\numberwithin{equation}{section} \theoremstyle{plain}
\usepackage{amsmath,amsthm,amssymb,xypic}
\usepackage{hyperref}
\usepackage{color}

\theoremstyle{plain}
    \newtheorem{theorem}{Theorem}[section]
    \newtheorem{lemma}[theorem]{Lemma}
    \newtheorem{corollary}[theorem]{Corollary}
    \newtheorem{proposition}[theorem]{Proposition}

 \theoremstyle{definition}
    \newtheorem{definition}[theorem]{Definition}
    
    \newtheorem{example}[theorem]{Example}

    \newtheorem{remark}[theorem]{Remark}

\theoremstyle{remark}

\numberwithin{equation}{section}

 \DeclareMathOperator{\Tr}{Tr}
 \DeclareMathOperator{\tr}{tr}

\DeclareMathOperator{\spec}{spec}
\DeclareMathOperator{\AS}{AS}

\DeclareMathOperator{\Ad}{Ad}
\DeclareMathOperator{\ad}{ad}

\DeclareMathOperator{\End}{End}

\DeclareMathOperator{\ch}{ch}

\DeclareMathOperator{\mat}{mat}

\DeclareMathOperator{\rank}{rank}

\DeclareMathOperator{\SO}{SO}
\DeclareMathOperator{\SL}{SL}

\DeclareMathOperator{\U}{U}

         \DeclareMathOperator{\supp}{supp}

\DeclareMathOperator{\vol}{vol}
\DeclareMathOperator{\HS}{HS}
\DeclareMathOperator{\Pf}{Pf}
\DeclareMathOperator{\Real}{Re}
\DeclareMathOperator{\Riem}{Riem}

\begin{document}

    \newcommand{\R}{\mathbb{R}}
    \newcommand{\C}{\mathbb{C}}
    \newcommand{\N}{\mathbb{N}}
    \newcommand{\Z}{\mathbb{Z}}
    \newcommand{\Q}{\mathbb{Q}}
    \newcommand{\bT}{\mathbb{T}}
    \newcommand{\bP}{\mathbb{P}}

\newcommand{\kg}{\mathfrak{g}}
\newcommand{\ka}{\mathfrak{a}}
\newcommand{\kb}{\mathfrak{b}}
\newcommand{\kk}{\mathfrak{k}}
\newcommand{\kt}{\mathfrak{t}}
\newcommand{\kp}{\mathfrak{p}}
\newcommand{\kh}{\mathfrak{h}}
\newcommand{\kn}{\mathfrak{n}}
\newcommand{\kso}{\mathfrak{so}}

\newcommand{\cA}{\mathcal{A}}
\newcommand{\cE}{\mathcal{E}}
\newcommand{\calL}{\mathcal{L}}
\newcommand{\calH}{\mathcal{H}}
\newcommand{\cO}{\mathcal{O}}
\newcommand{\cB}{\mathcal{B}}
\newcommand{\cK}{\mathcal{K}}
\newcommand{\cP}{\mathcal{P}}
\newcommand{\cN}{\mathcal{N}}
\newcommand{\calD}{\mathcal{D}}
\newcommand{\cC}{\mathcal{C}}
\newcommand{\calS}{\mathcal{S}}
\newcommand{\cM}{\mathcal{M}}
\newcommand{\cU}{\mathcal{U}}
\newcommand{\cT}{\mathcal{T}}

\newcommand{\Xcompl}{(G/Z) \setminus X}

\newcommand{\Rnz}{\R \setminus \{0\}}

\newcommand{\Bigwedge}{\textstyle{\bigwedge}}

\newcommand{\beq}[1]{\begin{equation} \label{#1}}
\newcommand{\eeq}{\end{equation}}

\newcommand{\Todo}{\textbf{To do}}

\newcommand{\ddt}{\left. \frac{d}{dt}\right|_{t=0}}
\newcommand{\dds}{\left. \frac{d}{ds}\right|_{s=0}}

\newcommand{\mattwo}[4]{
\left( \begin{array}{cc}
#1 & #2 \\ #3 & #4
\end{array}
\right)
}

\title{Equivariant analytic torsion for proper actions}

\author{Peter Hochs
and Hemanth Saratchandran}

\address{Institute for Mathematics, Astrophysics and Particle Physics, Radboud University}
\email{p.hochs@math.ru.nl}

\address{Australian Institute for Machine Learning, Adelaide University}
\email{hemanth.saratchandran@adelaide.edu.au}


\date{\today}

\maketitle

\begin{abstract}
We construct an equivariant version of Ray--Singer analytic torsion for proper, isometric actions by locally compact groups on Riemannian manifolds, with compact quotients. 
We obtain results on convergence, metric independence, vanishing for even-dimensional manifolds, a product formula, and a decomposition of classical Ray--Singer analytic torsion as a product over conjugacy classes of  equivariant torsion on universal covers. We do explicit computations for the circle and the line acting on themselves, and for regular elliptic elements of $\SO_0(3,1)$ acting on $3$-dimensional hyperbolic space. Our constructions and results generalise several earlier constructions of equivariant analytic torsion and their properties, most of which apply to finite or compact groups, or to fundamental groups of compact manifolds acting on their universal covers. 
These earlier versions of equivariant analytic torsion were associated to finite or compact conjugacy classes; we allow noncompact conjugacy classes, under suitable growth conditions.
\end{abstract}

\tableofcontents

\section*{Introduction}

\subsection*{Background}

The construction of torsion invariants has become a powerful way to study the topology of smooth manifolds. The first torsion invariant was defined combinatorially in the 1930s by Reidemeister for 3-dimensional compact smooth manifolds \cite{Reideme}. A key property of the Reidemeister torsion was that it was a homeomorphism invariant, and thus could be used to distinguish two spaces that are homotopy equivalent but not homeomorphic. Using this property of the torsion, Reidemeister was able to classify 3-dimensional lens spaces. A generalisation of the Reidemeister torsion for higher dimensional compact smooth manifolds was introduced by Franz, and used to classify higher dimensional lens spaces \cite{Franz35}.

In the 1970s, Ray and Singer asked whether Reidemeister-Franz combinatorial torsion could be defined analytically. Given a closed, oriented, smooth Riemannian manifold, they defined an analytic torsion invariant using the spectrum of the Laplacian defined on  differential forms twisted by a flat vector bundle. Using an asymptotic analysis of the heat kernel on the diagonal, together with a vanishing assumption on the twisted de Rham cohomology of the manifold, they showed that the analytic torsion is independent of the 
Riemannian metric \cite{RS71}. Furthermore, they conjectured that their analytic torsion was equal to the Reidemeister--Franz torsion. The Ray--Singer conjecture was proved true independently by Cheeger
\cite{cheeger79} and M\"{u}ller \cite{mueller78}. Since then analytic torsion has found several applications in diverse areas such as
the study of the growth of torsion homology of arithmetic groups \cite{BV}, locally symmetric spaces \cite{Koehler93, Koehler97, Matz, mueller20} 
and mathematical physics \cite{Bismut88, Hawking}.

The 1990s saw a great deal of activity in the development of equivariant versions of analytic torsion, in the case of actions by finite and compact groups 
\cite{Bismut95,Deitmar98,Koehler93,Koehler97, Lott94, LR91,Lueck93}. (Two notions of equivariant analytic torsion in this setting were compared in \cite{BG04}.)
Furthermore, 
in \cite{Lott92c,Lott99,Mathai92} analytic torsion invariants are defined in the case of the fundamental group of a closed smooth manifold acting on its universal cover. The key insight in the latter works is that one can use a von Neumann trace and follow the method of Ray--Singer, with some refined $L^2$-analysis, to construct an analytic torsion. In \cite{Lott99}, Lott constructs a torsion invariant for each finite conjugacy class of the fundamental group, these are known as delocalised torsion invariants. Della Vedova and Spreafico \cite{DVS25} constructed a notion of equivariant analytic torsion for noncompact groups, localised at unitary representations.



Our goal in this paper is to give a general construction of equivariant analytic torsion for proper, cocompact actions, that unifies and extends the above mentioned earlier constructions. 
Little has been done by way of defining a torsion invariant for non-compact group actions on non-compact manifolds, apart from fundamental groups of compact manifolds acting on their universal covers.
In \cite{Su13}, Su defined an analytic torsion in
the case of a locally compact group $G$ acting properly and cocompactly on a complete smooth manifold $M$ with $G$-invariant metric. Su's insight is that one can use the  
von Neumann trace from \cite{Wang14}, associated to the action of $G$ on $M$, thereby generalising the constructions of \cite{Lott92c,Mathai92} to this setting.

We define an equivariant analytic torsion invariant in
the case of a unimodular, locally compact group $G$ acting properly and cocompactly on a complete smooth manifold $M$ with $G$-invariant metric and orientation. 
This is associated to conjugacy classes in the group, which are not assumed to be finite or compact. 
Thus our construction is a common  generalisation of
\begin{itemize}
\item Lott's delocalised invariants \cite{Lott99} from fundamental groups of compact manifolds to general locally compact groups, not necessarily acting freely;
\item Su's work \cite{Su13} from the trivial conjugacy class to nontrivial ones;
\item constructions for finite and compact groups \cite{Bismut95,Deitmar98,Koehler93,Koehler97,Lueck93} to noncompact groups.
\end{itemize}
The constructions of $L^2$-analytic torsion in \cite{Lott92c, Mathai92} are special cases of both the first two points above. 
In the case that the group $G$ is trivial, our construction recovers the construction of Ray--Singer. In the case that the group $G$ is finite, our construction produces a finite list of torsion invariants, one for each conjugacy class of the group. Taking the sum of these invariants, we  recover the analytic torsion for finite group actions of Lott--Rothenberg \cite{LR91}.


%

\subsection*{The $g$-trace}

An important ingredient in our construction of equivariant analytic torsion is a trace associated to conjugacy classes in groups acting on manifolds, the \emph{$g$-trace}.

Let $M$ be an oriented Riemannian manifold, on which a unimodular, locally compact group $G$ acts properly and isometrically, preserving the orientation, and such that $M/G$ is compact. Fix a Haar measure $dx$ on $G$, and let $\psi \in C^{\infty}_c(M)$ be a nonnegative function such that for all $m \in M$, 
\[
\int_G \psi(xm)\, dx = 1.
\]
Let $g \in G$, and let $Z<G$ be its centraliser. Suppose that there is a (nonzero) $G$-invariant measure $d(hZ)$ on $G/Z$; this is equivalent to $Z$ being unimodular. 

Let $W \to M$ be a Hermitian, $G$-equivariant vector bundle. 
Let $T$ be a  $G$-equivariant, bounded operator  on $L^2(W)$. If $T \psi$ is trace class (where $\psi$ acts as a multiplication operator) and the integral
\beq{eq def g trace intro}
\int_{G/Z} \Tr(hgh^{-1}T\psi)\, d(hZ)
\eeq
converges, then  \eqref{eq def g trace intro} is the \emph{$g$-trace} of $T$, denoted by $\Tr_g(T)$. In practice, we apply this to operators $T$ with smooth Schwartz kernels, and use an expression for the $g$-trace in terms of these kernels. 

If $g=e$, the $g$-trace is the von Neumann trace used in Atiyah's $L^2$-index theorem \cite{Atiyah76}. This was later used in index theorems by Connes--Moscovici \cite{Connes82} and Wang \cite{Wang14}. More general versions of the $g$-trace were applied in recent results in index theory  \cite{HW2, Wangwang}, including results on manifolds with boundary
\cite{HWW, HWWII, PPST21, XieYu}. For semisimple Lie groups, a higher cyclic cocycle generalising the $g$-trace was constructed in \cite{ST19} and applied in \cite{HST20, PPST21}.

The $g$-trace was also used to construct secondary invariants, such as variations on the $\eta$-invariant \cite{CWXY19, HWW, Lott99,  PPST21}. The earlier versions \cite{Bismut95,Deitmar98,Koehler93,Koehler97, Lott92c, Lott99,LR91, Mathai92,  Su13} of equivariant analytic torsion mentioned above can also be formulated in terms of various special cases of the $g$-trace.


In the case where $M$ is the universal cover of a compact manifold $N$, acted on by the fundamental group $G = \pi_1(N)$, the $g$-trace has been used in many places to decompose the operator trace of an operator on $N$ as a sum of $g$-traces of the lift of this operator to $M$. This is the basis of the Selberg trace formula.

Compared to the von Neumann trace $\Tr_e$, there are two main challenges to overcome in applications of the general $g$-trace:
\begin{enumerate}
\item the $g$-trace is not positive in general;
\item if $G/Z$ is not compact, then convergence issues related to the integral \eqref{eq def g trace intro} must be dealt with.
\end{enumerate}
Due to the second point, volume growth behaviour of $G/Z$ plays a crucial role in most results involving the $g$-trace for noncompact $G/Z$. In results on index theory, one tool that can be used here is off-diagonal, short time decay behaviour of heat kernels to localise estimates in compact subsets of $G/Z$. In the construction of higher invariants, as for analytic torsion, an added complication is that large time behaviour of heat kernels is also important.

For this and other reasons, all previous constructions and results for equivariant analytic torsion that we know of only apply when $G/Z$ is compact. Even in the case where $g=e$, assumptions have to be made on large time heat kernel behaviour: positivity of the Novikov--Shubin invariants, see \eqref{eq alpha e intro} below.



\subsection*{Summary of main results}

 As mentioned before, our goal is to construct an 
equivariant analytic torsion that unifies and extends earlier constructions. 
We study its convergence, and relevant properties such as independence of the Riemannian metric. By computations in examples, we show that the construction leads to new information about manifolds and group actions. 

Let $E \to M$ be a flat, Hermitian, $G$-equivariant vector bundle. Let $\nabla^E$ be a flat, $G$-invariant connection preserving the metric. Let $\Delta_E^p$ be the Laplacian $(\nabla^E)^*\nabla^E+ \nabla^E(\nabla^E)^*$ acting on $p$-forms twisted by $E$. Let $P_p$ be orthogonal projection onto the $L^2$-kernel of $\Delta_E^p$. Whenever the following expressions converge, we define the \emph{equivariant analytic torsion} $T_g(\nabla^E)$ of $\nabla^E$ at $g$ by
\beq{eq def torsion intro}
-2\log T_g(\nabla^E) := \left. \frac{d}{ds}\right|_{s=0} \frac{1}{\Gamma(s)} \int_0^{
1} t^{s-1} \cT_g(t) \, dt
+
\int_1^{
\infty} t^{-1}  \cT_g(t) \, dt,
\eeq
where
\[
\cT_g(t) = \sum_{p=0}^{\dim(M)} (-1)^p p \Tr_g(e^{-t\Delta_E^p} - P_p).
\] 
In the first term on the right hand side of \eqref{eq def torsion intro}, we use meromorphic continuation from $s$ with large real part to $s=0$, again assuming the resulting expression is well-defined.

Because the $g$-trace \eqref{eq def g trace intro} depends on the group $G$, not just the element $g$ in $G$, so does equivariant analytic torsion. We make this dependence explicit in Proposition \ref{prop TG TH}.

If $G$ is the trivial group, so $M$ is compact, then $T_e(\nabla^E)$ is the Ray--Singer analytic torsion \cite{RS71}. More generally, the constructions in \cite{Bismut95,Deitmar98,Koehler93,Koehler97, Lott92c, Lott99,LR91, Mathai92,  Su13}  are special cases for the respective types of group actions and elements $g$.

We prove convergence and meromorphic extension of the first term on the right in \eqref{eq def torsion intro}  under a very general condition: existence of a \v{S}varc--Milnor function (see definition \ref{def SM}). For connected semisimple Lie groups, and finitely generated discrete groups, a $G$-invariant distance function on the group is such a function. If $G/Z$ is compact, then the zero function is a \v{S}varc--Milnor function, so this condition is really about noncompactness of $G/Z$.

Convergence of the second term on the right in \eqref{eq def torsion intro} is a more delicate matter. Even in the case where $M$ is the universal cover of a compact manifold $N$, and $g=e \in G = \pi_1(N)$, it is not known in general if that term converges. A sufficient condition is positivity of the \emph{Novikov--Shubin invariants} of $N$:
\beq{eq alpha e intro}
\alpha^p_e := \sup\bigl\{ \alpha>0; \Tr_e(e^{-t\Delta_E^p} - P_p) = \cO(t^{-\alpha}) \text{ as $t \to \infty$} \bigr\}.
\eeq
In the case of actions by fundamental groups on universal covers, with $E$ the trivial line bundle, these numbers were proved to be independent of the Riemannian metric \cite{ES89, NS86}, the smooth structure \cite{Lott92c}, and were even shown to be homotopy invariants  \cite{GS91}. The numbers $\alpha^p_e$ are positive in all examples we are aware of but a general proof of positivity has not been found. 

We assume that the natural generalisations  $\alpha_g^p$ of the numbers \eqref{eq alpha e intro}, with $\Tr_e$ replaced by $\Tr_g$, are positive. Then the second term in \eqref{eq def torsion intro} converges. We prove positivity of $\alpha_g^p$ in several cases. 
Ongoing work by S.\ Shen, Y.\ Song and X.\ Tang leads to a new way to use Bismut's trace formula to study large time behaviour of $g$-traces of heat operators on manifolds of the form $G/K$, with $G$ a connected, real semisimple Lie group and $K<G$ maximal compact. This may  lead to new positivity results for the numbers $\alpha_g^p$.

In cases where \eqref{eq def torsion intro} converges, the first question is if $T_g(\nabla^E)$ is independent of the Riemannian metric on $M$ and the Hermitian metric on $E$. It does depend on $\nabla^E$, as in the classical case, where $E$ and $\nabla^E$ are determined by a representation of the fundamental group of $M$. 
 Independence of the $G$-invariant Hermitian metric on $E$ preserved by $\nabla^E$ holds in general. Independence of the $G$-invariant Riemannian metric on $M$ is a more subtle question, particularly if $G/Z$ is noncompact. 
This is because 
 the large time behaviour of heat kernels on noncompact manifolds plays a central role, and studying this is a difficult problem.

There are many results on large time behaviour of scalar heat kernels, see for example \cite{Grigoryan09} for an overview of some results. Generalising these results to heat kernels on differential forms is not straightforward, primarily due
to lack of positivity of heat kernels, see \cite{Coulhon20, Coulhon07, Devyver14}.
The fact that it is not known if Novikov--Shubin invariants are always positive is one reflection of the difficulty of studying large time behaviour of heat kernels on differential forms on noncompact manifolds. 

In the case where $G/Z$ is noncompact, we give a relation between large time heat kernel decay and volume growth of $G/Z$ that implies independence of equivariant analytic torsion of the Riemannian metric. We give some special cases where this condition holds including, trivially, the case where $G/Z$ is compact. By making the role of large time heat kernel decay explicit, we hope that  future results on this decay behaviour will imply new results on  metric independence of equivariant analytic torsion.

A concrete formulation is as follows. Let $\kappa_t^p$ be the Schwartz kernel of $e^{-t\Delta_E^p}$. We assume that for all $p$, there is a function $F_1$, and a constant $a>0$,  such that for all $t\geq 1$ and all $m,m' \in M$, 
\[
\|\kappa_t^p(m,m')\| \leq F_1(t) e^{-a d(m,m')^2/t}.
\]
Furthermore, we consider a \v{S}varc--Milnor function $l$ on the conjugacy class $(g)$ of $g$, which may be thought of as a distance function to the identity element. For $t>0$, we set
\[
F_3(t) := \sum_{j=0}^{\infty} \vol\bigl(\{ hZ \in G/Z;  j < l(hgh^{-1}) \leq j+1 \}\bigr) e^{-a'j^2/t},
\]
where $a'>0$ depends on $a$ and the properties of $l$ (see \eqref{eq def F3}). 

We write $\Delta_E$ for the direct sum of the Laplacians $\Delta_E^p$; i.e.\ $\Delta_E := (\nabla^E)^*\nabla^E+ \nabla^E(\nabla^E)^*$ acting on  differential forms of all degrees, twisted by $E$.  By $\ker(\Delta_E)$ we mean the $L^2$-kernel of this operator. 

The most general form of our convergence and metric independence result is the following. We view this as the main result in this paper. 
\begin{theorem}\label{thm metric indep intro}
Suppose that $\ker(\Delta_E)=0$. If $F_1(t) F_3(t) = \cO(t^{-\alpha})$ as $t \to \infty$, for an $\alpha>0$, then \eqref{eq def torsion intro} converges. Then equivariant analytic torsion is constant on any smooth path of $G$-invariant Riemannian metrics along which this condition on $F_1 F_3$ holds.
\end{theorem}
See theorem \ref{metric_indep_cconj}. The condition that $F_1(t) F_3(t)$ decays as a negative power of $t$ means that the volume growth of $G/Z$ should be cancelled by the large time decay of the heat kernel in a suitable way. In view of theorem \ref{thm metric indep intro}, it is an interesting question when the heat kernel decay rate $F_1$ can be estimated independently of the $G$-invariant Riemannian metric, given that $M/G$ is compact.

As concrete applications of theorem \ref{thm metric indep intro}, we obtain convergence and metric independence of equivariant analytic torsion 
\begin{itemize}
\item if $G/Z$ is compact and $\alpha_e^p>0$ for all $p$ (see corollary \ref{cor metric indep GZ cpt});
\item if $M$ is simply connected, $\Delta_E$ is invertible (a condition independent of the metric), and $G/Z$ has polynomial, or slow enough exponential volume growth
 (see corollary \ref{cor metric indep invertible});
 \item for metrics in the same path component of the space of $G$-invariant Riemannian metrics satisfying a positive curvature condition, if $G/Z$ has slow enough polynomial volume growth
  (see corollary \ref{cor metric indep slow growth}).
\end{itemize}
 The first case, for compact $G/Z$, 
includes all earlier results on metric independence of equivariant (or $L^2$-, or delocalised) analytic torsion that we are aware of, such as proposition 2 in \cite{LR91},  theorem 4.1 in \cite{Mathai92}, corollary 9 in \cite{Lott92c}, proposition 3 in \cite{Lott99} and remark 3.2 in \cite{Su13}.
%

A weaker condition than the condition on $F_1F_3$ in theorem \ref{thm metric indep intro}, which is still sufficient for the theorem to hold, is that for all $G$-equivariant vector bundle endomorphisms $A$ (more specifically, for the identity and the two endomorphisms in \eqref{eq general large t conditions}), we have $|\Tr_g(A e^{-t\Delta_E})| = \cO(t^{-\alpha})$ for an $\alpha>0$. This is what is used in the case of compact $G/Z$ above. Recall here that $\ker(\Delta_E)$ is assumed to be zero.

Some of the further results on equivariant analytic torsion involve an equivariant Euler characteristic (see definition \ref{def chi g}). We prove
\begin{itemize}
\item vanishing of equivariant analytic torsion on even-dimensional manifolds, under a condition on the equivariant Euler characteristic (see proposition \ref{prop vanish even});
\item a product formula for equivariant analytic torsion  of Cartesian products, involving the equivariant Euler characteristic (see proposition \ref{prop prod form});
\item  a decomposition of Ray--Singer analytic torsion on a compact manifold in terms of equivariant analytic torsion on its universal cover
(see proposition \ref{prop Tsigma TN}).
\end{itemize}

In the last two sections of this paper, we compute equivariant analytic torsion explicitly for
\begin{itemize}
\item the real line acting on itself by addition (see proposition \ref{prop torsion R});
\item the circle acting on itself by rotation (see proposition \ref{prop torsion line} and lemma \ref{lem circle alpha zero});
\item regular elliptic elements of $\SO_0(3,1)$ acting on $3$-dimensional hyperbolic space (see proposition \ref{prop H3}). 
\end{itemize}

This last computation on hyperbolic 3-space makes use of 
Bismut's orbital integral trace formula \cite{BismutHypo}. As far as we are aware, the only other place where this technique is used in the study of analytic torsion is in
Shen's work \cite{Shen18}.

\subsection*{Future work}

The construction of an equivariant analytic torsion naturally leads to the question of whether there exists a combinatorial counterpart, generalising the 
Reidemeister--Franz torsion to this setting, and whether such an equivariant combinatorial torsion is equal to the equivariant analytic torsion. In the classical setting of Ray--Singer and Reidemeister--Franz this is the content of the 
Cheeger--M\"{u}ller theorem \cite{cheeger79, mueller78}. Such 
a theorem has also been proved in various equivariant settings 
\cite{Bunke99, Roth78, Su07}. In future work, we plan to understand when such a theorem holds in our general equivariant context.

Another area of future research based on the results of this paper is to do with the connection of analytic torsion with dynamical systems. This connection was first expounded by Milnor in \cite{Milnor68}. In the case of an orientable hyperbolic manifold, Fried proved an identity relating the analytic torsion of an acyclic unitarily flat vector bundle and the value at zero of the Ruelle dynamical zeta function of the geodesic flow \cite{Fried86}. He conjectured that this relation should hold true in more general settings \cite{Fried87}. Since then, this conjecture has been proved in several cases, see for example
 \cite{DGRS20,  Deitmar98, Moscovici91, Shen18, Yamaguchi22}, or the survey \cite{Shen21}. 
  In future work, we plan to develop an equivariant version of   the Ruelle zeta function and investigate its relations with 
equivariant analytic torsion.

\subsection*{Acknowledgements}

We thank Sebastian Goette, Christopher Pirie, Thomas Schick and an anonymous referee for helpful comments.
PH is partially supported by the Australian Research Council, through Discovery Project DP200100729, and by NWO, through ENW-M grant OCENW.M.21.176. 
HS was supported by the Australian Research Council, through grant FL170100020.

\subsection*{Statement on the use of AI}

No AI-based tools were used to prepare this manuscript.

\section{Equivariant analytic torsion} \label{sec g Tr}

We start with some notation that will be used throughout this paper. 
Let $M$ be a complete, oriented Riemannian manifold, on which a unimodular locally compact group $G$ acts properly and isometrically, preserving the orientation. Suppose that $M/G$ is compact (we then say that the action is \emph{cocompact}).  We denote the Riemannian density by $dm$, or sometimes by $d\mu(m)$ if the density plays an explicit role.

Fix a Haar measure $dx$ on $G$. 
Let $g \in G$, and let $Z<G$ be its centraliser. Suppose that $G/Z$ has a $G$-invariant measure $d(hZ)$, or equivalently that $Z$ is unimodular. 
Let $\psi \in C^{\infty}_c(M)$ be nonnegative-valued, and  such that for all $m \in M$,
\beq{eq cutoff}
\int_G \psi(xm)\, dx = 1.
\eeq
Such a function exists because the action is proper and cocompact. 

An ingredient in the definition of equivariant analytic torsion is the $g$-trace \cite{HWW, HW2}, defined in terms of (the conjugacy class of) an element $g \in G$. Let $W \to M$ be a Hermitian $G$-vector bundle.
\begin{definition}\label{def g Tr}
Let $T$ be a $G$-equivariant  operator from $\Gamma^{\infty}_c(W)$ to $\Gamma^{\infty}(W)$.
 If $T$ has   a smooth kernel $\kappa$, and  the  integral 
\beq{eq def Trh}
\Tr_{g}(T):=\int_{G/Z} \int_M \psi(hgh^{-1}m) \tr\bigl(hgh^{-1}\kappa(hg^{-1}h^{-1}m,m) \bigr)\, dm\, d(hZ).
\eeq
converges absolutely, then $T$ is \emph{$g$-trace class}, and the value of this integral is the \emph{$g$-trace} of $T$.
\end{definition}
We sometimes write $\Tr_g^G$ for $\Tr_g$ to emphasise the group acting.

\begin{remark}
Due to $G$-invariance of the kernel $\kappa$ in definition \ref{def g Tr} and the trace property of the fibre-wise trace $\tr$, one can use a substitution to replace  $\psi(hgh^{-1}m)$ in \eqref{eq def Trh} by $\psi(m)$. 
Then, 
if the operator  $T\psi$  is trace class and $T$ is $g$-trace class, we find that
\beq{eq Trg Tr}
\Tr_{g}(T) =\int_{G/Z} \Tr(hgh^{-1} T\psi)\, d(hZ).
\eeq
In particular, if $g=e$, then we recover the von Neumann trace $\Tr_{e}(T) = \Tr(\sqrt{\psi}T\sqrt{\psi})$ from \cite{Atiyah76}.

The equality \eqref{eq Trg Tr} allows us to extend the definition of the $g$-trace to operators that do not necessarily have smooth kernels. However, in some situations, not considered in this paper, mainly when $M/G$ is noncompact, it can be useful to apply the $g$-trace also to operators $T$ for which $T\psi$ is not trace class, see \cite{HWW}. For that reason, it is convenient to define the $g$-trace in terms of smooth kernels. This does mean that, if $G$ is trivial and $M$ is compact, then the $e$-trace class property is different from the trace class property. 
\end{remark}

Let $E \to M$ be a Hermitian $G$-vector bundle. 
Let $\nabla^E$ be a Hermitian, $G$-invariant, flat connection on $E$, assuming this exists. We also use $\nabla^E$ to denote the induced operator on differential forms with values in $E$.
Consider the Laplacian $\Delta_E := (\nabla^E)^*\nabla^E + \nabla^E (\nabla^E)^*$ acting on differential forms with values in $E$. We denote its restriction to $p$-forms by $\Delta_E^p$. 

Let $P$ be orthogonal projection onto the kernel of $\Delta_E$, and $P_p$ is restriction to $p$-forms. Here and in the rest of this paper, the kernel of a Laplace-type operator like $\Delta_E$ will always mean its kernel in the space of \emph{square-integrable} $E$-valued differential forms. We write $\ker(\Delta_E)$ for the kernel of $\Delta_E$ on such forms.

We denote the number operator on $\Bigwedge^* T^*M \otimes E$, equal to $p$ on $\Bigwedge^p T^*M \otimes E$, by $F$. 
If $(-1)^F F (e^{-t\Delta_E} - P)$ is $g$-trace class, then we write
\beq{eq curly T}
 \cT_g(t) := \Tr_g((-1)^F F (e^{-t\Delta_E} - P)).
\eeq
(A sufficient, but not necessary condition for $(-1)^F F (e^{-t\Delta_E} - P)$ to be $g$-trace class is that $e^{-t\Delta_E} - P$ is $g$-trace class.)
\begin{definition}\label{def g torsion}
Suppose that
\begin{enumerate}
\item $(-1)^F F (e^{-t\Delta_E} - P)$ is $g$-trace class for all $t>0$; 
\item there is an $\alpha_g >0$ such that $ \cT_g(t)= \cO(t^{-\alpha_g})$ as $t \to \infty$; and
\item the integral
\beq{eq small t int}
 \frac{1}{\Gamma(s)} \int_0^{
1} t^{s-1} \cT_g(t) \, dt
\eeq
converges for $s \in \C$ with large enough real part, and this expression has a meromorphic continuation to $\C$ that is regular at $s = 0$.
\end{enumerate}
Then the \emph{equivariant analytic torsion} $T_g(\nabla^E)$ at $g$ is defined by
\beq{eq def deloc torsion}
-2\log T_g(\nabla^E) := \left. \frac{d}{ds}\right|_{s=0} \frac{1}{\Gamma(s)} \int_0^{
1} t^{s-1} \cT_g(t) \, dt
+
\int_1^{
\infty} t^{-1}  \cT_g(t) \, dt. 
\eeq
If $E = M \times \C$ is the trivial line bundle, with the trivial action by $G$ on its fibres, and $\nabla^E = d$, then the \emph{equivariant analytic torsion} of $M$ at $g$ is
$T_g(M) := T_g(d)$.
\end{definition}

Other versions of this definition are given and used in special cases in
 \cite{Lott92c,  Mathai92, Su13} for $g=e$ and in 
 \cite{Lott99, Su07} for $g \not = e$. In  \cite{Lott99, Su07}, the version of analytic torsion used is called \emph{delocalised}, because its definition involves off-diagonal values of the heat kernel. In the case of a finite group $G$, so $M$ is compact, Lott and Rothenberg \cite{LR91} defined a notion of analytic torsion that equals
 \[
 -2
 \sum_{(g)} \log T_g(\nabla^E),
 \]
where the sum is over the conjugacy classes in $G$. Bunke \cite{Bunke99} extended this notion of analytic torsion to compact groups. Other special cases for finite and compact groups were developed in \cite{Bismut95,Deitmar98,Koehler93,Koehler97,Lueck93}. 

If $G$ is a compact group (such as the trivial group), then $M$ is compact. So $e^{-t \Delta_E^p} - P_p$ is trace class, and $\Tr_e(e^{-t \Delta_E^p} - P_p)$ is its trace. This directly implies that $T_e(\nabla^E)$ is the classical Ray--Singer torsion of definition 1.6 in  \cite{RS71} (up to a minus sign in the exponent). 
We give another, less trivial, link between equivariant delocalised analytic torsion and Ray--Singer torsion in proposition \ref{prop Tsigma TN}.

\begin{remark}
As the notation suggests, equivariant analytic torsion depends on the  connection $\nabla^E$ in general. In the compact, non-equivariant  case, a Hermitian vector bundle with a flat, metric-preserving connection corresponds to a unitary representation of the fundamental group via the Riemann--Hilbert correspondence, and analytic torsion is usually considered as depending on such a representation. See remark \ref{rem connection R} for an example.
\end{remark}

\begin{remark}
The integrals from $0$ to $1$ and from $1$ to $\infty$ in \eqref{eq def deloc torsion} appear in different ways, because of the behaviour of the function $t\mapsto t^{-1}\cT_g(t)$ as $t$ tends towards $0$ and $\infty$. Near $t=0$, the integral \eqref{eq small t int} may only converge for $s$ with large real part. But for such $s$, a version of the second term in \eqref{eq def deloc torsion} with $t^{-1}$ replaced by $t^{s-1}$ may diverge. For an example, see the proof of
lemma \ref{lem torsion ZR 0}, where $\cT_e(t)$ is a constant times $t^{-1/2}$.
\end{remark}

\begin{remark} \label{rem int bdry}
For any smooth function $f$ on $(0, \infty)$, and $0<a<b$, the equality 
\beq{eq der ts Gamma}
\dds \frac{t^s}{\Gamma(s)}=1
\eeq
 for all $t>0$ implies that 
\[
\left. \frac{d}{ds}\right|_{s=0} \frac{1}{\Gamma(s)} \int_a^{b} t^{s-1} f(t) \, dt = \int_a^{b} t^{-1} f(t) \, dt. 
\]
This implies that the point $t=1$ where the integrals are split up in \eqref{eq def deloc torsion} may be replaced by any other positive number, and the resulting equivariant analytic torsion will have the same value.
\end{remark}

In the rest of this paper, we examine when the conditions of definition \ref{def g torsion} are satisfied, and what the properties of equivariant analytic torsion are in that case. We also do some explicit computations. Independence of the Riemannian metric and the Hermitian metric on $E$ are fundamental, and imply that equivariant analytic torsion is a smooth invariant of the group action.

The precise statements of most results in his paper require some further definitions, given throughout the paper. Here is an overview of those results.
\begin{itemize}
\item Proposition \ref{prop conv small t} is the main result on  conditions (1) and (3) in definition \ref{def g torsion}.
\item For condition (2) in definition \ref{def g torsion}, there are lemmas \ref{lem NS hyp} and \ref{lem Tgt n=1} for hyperbolic space, and proposition \ref{prop NS cpt} for conjugacy classes with slow enough volume growth.
\item The main results on metric independence are lemma \ref{lem indep Herm metric} for independence of the Hermitian metric on $E$, and corollary \ref{cor metric indep} for independence of the Riemannian metric for conjugacy classes with slow enough volume growth. We view corollary \ref{cor metric indep} as the most important result in this paper, and its proof is based on several results proved in earlier sections.
\item Triviality for even-dimensional manifolds is proposition \ref{prop vanish even}.
\item A product formula involving a generalisation of the Euler characteristic is given in proposition \ref{prop prod form}.
\item A relation between equivariant analytic torsion and classical Ray--Singer analytic torsion is given in proposition \ref{prop Tsigma TN}.
\item We do explicit computations for the real line acting on itself (proposition \ref{prop torsion R}), the circle acting on itself (proposition \ref{prop torsion line} and lemma \ref{lem circle alpha zero}), and regular elliptic elements of $\SO_0(3,1)$ acting on $3$-dimensional hyperbolic space (proposition \ref{prop H3}). 
\end{itemize}

A simple observation with possibly useful consequences is the following. Suppose that $g$ is contained in a subgroup $H<G$ such that $G/H$ is compact. We will see in proposition \ref{prop TG TH} that the equivariant analytic torsions for the actions by $G$ and $H$ differ by a fixed exponent.
This implies that all results on analytic torsion for the action by $H$ (such as convergence and metric independence) imply corresponding results for the action by $G$, and vice versa. 

For example, suppose that $G$ is a connected, semisimple Lie group, and $K<G$ is maximal compact. Suppose that $g$ lies in a discrete, torsion-free subgroup $H=\Gamma<G$ such that $G/\Gamma$ is compact, so that $g$ is a hyperbolic element. Then $\Gamma$ is the fundamental group of the locally symmetric space $\Gamma \backslash G/K$, and results on $L^2$-analytic torsion \cite{Lott92c, Mathai92} and  delocalised analytic torsion  \cite{Lott99} for $\Gamma \backslash G/K$  imply results for the action by $G$.

\section{Properties of the $g$-trace}

The $g$-trace of definition \ref{def g Tr} is a key ingredient of the definition of equivariant analytic torsion. We prove some basic properties of the $g$-trace that will be used in various places in this paper.

We first recall the trace property of the $g$-trace, lemma 3.2 in \cite{HWW}. It is more subtle than the trace property of the operator trace or the von Neumann trace, because it is not clear when the composition of two operators is $g$-trace class, even if the two operators individually are. 
\begin{lemma}\label{lem trace prop}
Let $S$ and $T$ be two $G$-equivariant operators on $\Gamma^{\infty}(W)$ such that $T$ has a smooth kernel, $S$ has a distributional kernel, and $ST$ and $TS$ are $g$-trace class. Then $\Tr_g(ST) =\Tr_g(TS)$.
\end{lemma}

\subsection{Cocompact subgroups}

Let $H<G$ be a unimodular closed subgroup containing $g$, such that $G/H$ is compact. We fix a Haar measure on $H$, and a compatible measure $d(xH)$ on $G/H$.
We now write $Z_G$ for the centraliser of $g$ in $G$, and $Z_H$ for the centraliser of $g$ in $H$.
We had assumed that $G/Z_G$ has a $G$-invariant measure $d(hZ_G)$; we also assume that $H/Z_H$ has an $H$-invariant measure $d(hZ_H)$, this exists if and only if $Z_H$ is unimodular. We fix measures on $Z_G$ and $Z_H$ compatible with these measures on quotient spaces.
\begin{lemma}\label{lem int GH}
The space $Z_G/Z_H$ is compact, and 
for all $f \in C_c(G)$,
\beq{eq int GH}
\vol \left( \frac{Z_G}{Z_H}\right)  \int_{G/Z_G} f(hgh^{-1})\, d(hZ_G) 
 = 
\int_{G/H} \int_{H/Z_H} f(xhgh^{-1}x^{-1})\, d(hZ_H)\, d(xH).
\eeq
\end{lemma}
\begin{proof}
The composition of the inclusion of $Z_G$ into $G$ and the quotient map $q\colon G\to G/H$ descends to a homeomorphism $Z_G/Z_H \xrightarrow{\cong} q(Z_G)$. And $Z_G \subset G$ is closed, so $q(Z_G) \subset G/H$ is closed, and hence compact.

By compatibility of the various measures, the left hand side of \eqref{eq int GH} equals
\[
  \int_{G/Z_G}\int_{Z_G/Z_H} f(hzgzh^{-1})\, d(z Z_H)\, d(hZ_G) 
  = \int_{G/Z_H} f(hgh^{-1})\, d(hZ_H),
\]
which equals the right hand side of \eqref{eq int GH}.
\end{proof}

\begin{lemma}\label{lem TrG TrH}
A  $G$-equivariant operator $T$ on $\Gamma^{\infty}(W)$ with a smooth kernel is $g$-trace class for the action by $G$ if and only if it is $g$-trace class for the action by $H$. Then
\[
\vol \left( \frac{Z_G}{Z_H}\right) 
\Tr^G_g(T) = \Tr^H_g(T).
\]
\end{lemma}
\begin{proof}
Consider the function $f$ on $G$ defined by
\[
f(x) = \int_M \psi(xm)\tr(x\kappa(x^{-1}m,m))\, dm.
\]
By $G$-invariance of $\kappa$, the trace property of $\tr$, and a substitution $m' = x^{-1}m$, we find that for all $x \in G$ and $h \in H$,
\beq{eq conj f}
f(xhgh^{-1}x^{-1}) = \int_M \psi(xhg h^{-1}m') \tr(hgh^{-1}\kappa(hg^{-1}h^{-1}m',m'))\, dm'.
\eeq
The function $\psi^H$ on $M$, defined by
\[
\psi^H(m) = \int_{G/H} \psi(xm)\, d(xH), 
\]
satisfies \eqref{eq cutoff} with $G$ replaced by $H$. Here a Borel section $G/H \to G$ is used implicitly, and $\psi^H$ is only measurable and bounded, with bounded support, but not necessarily smooth.
Lemma \ref{lem int GH} and \eqref{eq conj f} imply that
\begin{multline*}
\vol \left( \frac{Z_G}{Z_H}\right) 
\Tr^G_h(T) \\
=
\int_{H/Z_H}
 \int_M \psi^H(hg h^{-1}m') \tr(hgh^{-1}\kappa(hg^{-1}h^{-1}m',m'))\, dm'
\, d(hZ_H)
= \Tr^H_g(T).
\end{multline*}
Versions of these computations with absolute values in the integrands show that $T$ is $g$-trace class for the action by $G$ if and only if it is $g$-trace class for the action by $H$
\end{proof}

Lemma \ref{lem TrG TrH} implies a relation between equivariant analytic torsion for the actions by $G$ and $H$.
\begin{proposition}\label{prop TG TH}
The conditions of definition \ref{def g torsion} hold for the action by $G$ on $M$ if and only if they hold for the action by $H$ on $M$. Then
\[
T_g^G(\nabla^E)^{\vol(Z_G/Z_H)} = T_g^H(\nabla^E),
\]
where we use superscripts $G$ and $H$ to indicate the group acting.
\end{proposition}
Proposition \ref{prop TG TH} implies any results on $T_g^H(\nabla^E)$ imply corresponding results for $T_g^G(\nabla^E)$ and vice versa.

\subsection{The $g$-trace and the $e$-trace}

The $e$-trace (called the von Neumann trace) has more useful properties, such as positivity, than the $g$-trace in general. If $G/Z$ is compact, then the $g$-trace can be estimated in terms of the $e$-trace, and upper bounds for the $e$-trace imply corresponding upper bounds for the $g$-trace.

The arguments in this subsection are generalisations of arguments in section 4 of \cite{Atiyah76} and section 3 of \cite{Lott99} to possibly non-free actions by possibly non-discrete groups.

As in section \ref{sec g Tr}, let $W \to M$ be a Hermitian $G$-vector bundle. We are thinking of $W = \Bigwedge^p T^*M \otimes E$.
A $G$-equivariant, bounded linear operator $T$ on $L^2(W)$ will be called \emph{strongly $e$-trace class} if $\sqrt{\psi} T \sqrt{\psi}$ is trace class, and \emph{$e$-Hilbert--Schmidt} if $T \sqrt{\psi}$ is Hilbert--Schmidt. Fix a  $G$-equivariant, bounded linear operator $T$ on $L^2(W)$ with a smooth kernel. We denote the Hilbert--Schmidt norm by $\|\cdot\|_{\HS}$. 
\begin{lemma}\label{lem Tre HSe}
If $T$ is $e$-Hilbert--Schmidt, then so is $T^*$, and 
\[
\|T \sqrt{\psi} \|_{\HS} = \|T^* \sqrt{\psi} \|_{\HS}. 
\]
\end{lemma}
\begin{proof}
%
If $T$ is $e$-Hilbert--Schmidt, then by lemma \ref{lem trace prop}, 
\[
\|T\sqrt{\psi}  \|_{\HS} = \Tr(\sqrt{\psi} T^*T \sqrt{\psi} ) = \Tr_e(T^*T) = \Tr_e(TT^*)  = \|T^*\sqrt{\psi}  \|_{\HS}.
\]
\end{proof}


\begin{lemma}\label{lem compos eHS}
If $T$ is $e$-Hilbert--Schmidt, and $A \in \cB(L^2(W))^G$,
 then $AT$ and $TA$ are $e$-Hilbert--Schmidt.
\end{lemma}
\begin{proof}
This is immediate for $AT$. 
Applying
lemma \ref{lem Tre HSe} implies that $T^*$ is $e$-Hilbert--Schmidt if $T$ is, so it follows $A^*T^*$ is also $e$-Hilbert--Schmidt. Another application of lemma \ref{lem Tre HSe} then implies that $TA$ is $e$-Hilbert--Schmidt.
\end{proof}

\begin{lemma}\label{lem Tre AT}
Suppose that $T = S^*S$ for some $S \in \cB(L^2(W))^G$ with a smooth kernel, and that $T$ is  $e$-trace-class. If $A \in \cB(L^2(W))^G$, then $TA$ and $AT$ are strongly $e$--trace class.  Furthermore, 
\[
\begin{split}
|\Tr_e(TA)| & \leq \|A\| \Tr_e(T);\\
|\Tr_e(AT)| & \leq \|A\| \Tr_e(T).
\end{split}
\]
\end{lemma}
\begin{proof}
If $T$ is $e$-trace class, then the operator $S \sqrt{\psi}$ has a square-integrable smooth kernel. So $S$ is $e$-Hilbert--Schmidt.
Hence so is $SA$, by lemma \ref{lem compos eHS}. 
Consider a Hilbert basis $\{e_j\}_{j=1}^{\infty}$ of $L^2(E)$. Then
\beq{eq Tre AT}
\sum_{j=1}^{\infty} |(\sqrt{\psi}  TA \sqrt{\psi}e_j, e_j )_{L^2}| = 
\sum_{j=1}^{\infty} |( SA \sqrt{\psi}e_j, S\sqrt{\psi} e_j )_{L^2}| \leq \|SA\sqrt{\psi} \|_{\HS} \|S\sqrt{\psi} \|_{\HS}.
\eeq
Now lemma \ref{lem Tre HSe} implies that
\[
\|SA\sqrt{\psi} \|_{\HS} = \|A^*S^*\sqrt{\psi} \|_{\HS} \leq \|A^*\| \|S^*\sqrt{\psi} \|_{\HS} = \|A\| \|S\sqrt{\psi} \|_{\HS}. 
\]
So the right hand side of \eqref{eq Tre AT} is at most equal to 
\[
\|A\| \|S\sqrt{\psi} \|_{\HS}^2 = \|A\| \Tr_e(T).
\]

The argument for $AT$ is entirely analogous.
\end{proof}

\begin{proposition} \label{prop GZ cpt}
Suppose that $G/Z$ is compact. Suppose that $T = S^*S$ for some $S \in \cB(L^2(W))^G$ with a smooth kernel, and that $T$ is strongly $e$-trace-class. If $A \in \cB(L^2(W))^G$, then
\[
\begin{split}
|\Tr_g(TA)| & \leq \vol({G/Z}) \|A\| \Tr_e(T);\\
|\Tr_g(AT)| & \leq \vol({G/Z}) \|A\| \Tr_e(T).
\end{split}
\]
In particular, the integrals defining the left hand sides converge.
\end{proposition}
\begin{proof}
We prove the claim for $TA$, the argument for $AT$ is analogous.

Let $\kappa$ be the Schwartz kernel of  $TA$. 
For all $x \in G$, 
\[
\begin{split}
\int_M \psi(xm) \tr(x\kappa(x^{-1}m,m))\, dm &= \int_M \psi(xm) \tr(\kappa(m,xm)x)\, dm\\
&=\int_M \psi(m) \tr(\kappa(x^{-1}m,m)x)\, dm \\
&= \int_M \psi(m) \tr(x\kappa(x^{-1}m,m))\, dm\\
&= \Tr(x TA\psi  )\\
&= \Tr_e(xTA).
\end{split}
\]
For the first three equalities, we used $G$-equivariance of $T$, $G$-invariance of $dm$, and the trace property of $\tr$. It follows that
\[
|\Tr_g(TA)| \leq 
 \int_{{G/Z}} |\Tr_e(hgh^{-1}TA)|\, d(hZ).
\]
By lemma \ref{lem Tre AT} and unitarity of the action by $G$ on $L^2(W)$, the integrand on the right hand side is bounded above by $\|A\| \Tr_e(T)$. 
\end{proof}


\subsection{Derivatives of the $g$-trace}

For a Riemannian metric $\lambda$ on $M$, let $\mu_{\lambda}$ be its Riemannian density.
Fix an operator $T$ on a Hermitian $G$-vector bundle over $M$, 
and assume $T$ has a smooth kernel. Let
$\kappa_T^{\lambda}$ denote its kernel with respect to $\lambda$, so that
\begin{equation} \label{eq T kappa lambda}
(Ts)(m) = \int_M\kappa_T^{\lambda}(m,m')s(m')d\mu_{\lambda}(m')
\end{equation}
for $s \in \Gamma_c^{\infty}(E)$.

The $g$-trace is then given by 
\begin{equation}\label{$g$-trace_metric-dep}
\Tr_g(T,\lambda) = \int_{G/Z}\int_M\psi(hgh^{-1}m)\tr(hgh^{-1}\kappa_T^{\lambda}
(hgh^{-1}m, m))d\mu_{\lambda}(m)d(hZ).
\end{equation}
Note that we have added the dependence of $\lambda$ in the notation of the $g$-trace from the observation that the inner integral in the formula for the $g$-trace is taken with respect to the density $\mu_{\lambda}$.

Both expressions \eqref{eq T kappa lambda} and \eqref{$g$-trace_metric-dep} really depend on $\kappa^{\lambda}_T d\mu_{\lambda}$, which does not depend on the Riemannian metric even though $\kappa^{\lambda}_T $ and $d\mu_{\lambda}$ individually do. This implies the following basic fact.
\begin{lemma}\label{$g$-trace_der1}
The above integral \eqref{$g$-trace_metric-dep}, when it converges, does not depend on the
Riemannian metric $\lambda$.
\end{lemma}
This also follows from the expression \eqref{eq Trg Tr} for the $g$-trace, if $T\psi$ is trace class.
%
%
%
%


We say that a family  $(T_t)_{t \in \R}$ of operators from $\Gamma^{\infty}_c(W)$ to $\Gamma^{\infty}(W)$ with smooth kernels $(\kappa_t)_{t \in \R}$ is \emph{smooth} if the map $(m,m',t) \mapsto \kappa_t(m,m')$ is smooth. 
Then we define the operator $\frac{d}{dt}T_t\colon \Gamma_c^{\infty}(W) \to \Gamma^{\infty}(W)$ by
\begin{equation*}
\bigg{(}\bigg{(}\frac{d}{dt}T_t\bigg{)}s\bigg{)}(m) = 
\frac{d}{dt}\big{(}T_ts(m)\big{)} \qquad \forall s \in \Gamma^{\infty}_c(E) \text{ and } m\in M.
\end{equation*}
\begin{lemma}\label{$g$-trace_der2}
Let $(\lambda_t)_{t \in \R}$ be a smooth family of Riemannian metrics on $M$ and let $(T_t)_{t \in \R}$ be a smooth
family of $g$-trace class operators from $\Gamma^{\infty}_c(W)$ to $\Gamma^{\infty}(W)$. If $\frac{d}{dt}T_t$ is $g$-trace class, then the function $t \mapsto \Tr_g(T_t,\lambda_t) $ is differentiable, and
\begin{equation}\label{eq der trace}
\frac{d}{dt}\Tr_g(T_t,\lambda_t) = \Tr_g\bigg{(}\frac{d}{dt}T_t, \lambda_t\bigg{)}.
\end{equation}
\end{lemma}
\begin{proof}
For a fixed parameter $t_0 \in \R$ we have
\begin{align*}
\frac{d}{dt}\Tr_g(T_t,\lambda_t) \bigg{\vert}_{t = t_0} &= 
\frac{d}{dt}\Tr_g(T_t,\lambda_{t_0})\bigg{\vert}_{t = t_0} 
+ \frac{d}{dt}\Tr_g(T_{t_0},\lambda_t)\bigg{\vert}_{t = t_0} \\
&=
\frac{d}{dt}\Tr_g(T_t,\lambda_{t_0}) \bigg{\vert}_{t = t_0},
\end{align*}
the second term in the first equality being zero by lemma \ref{$g$-trace_der1}.

For any Riemannian metric $\lambda$, we have
\begin{equation}\label{kernel_der}
\kappa_{\frac{dT_t}{dt}}^{\lambda}(m, m') = \frac{d}{dt}\kappa_{T_t}^{\lambda}(m, m').
\end{equation}
Indeed, let  $s \in \Gamma^{\infty}_c(W)$ and $m \in M$, because the support of $s$ is compact we can swap the derivative and integral to find that
\begin{align*}
\int_M\kappa_{\frac{dT_t}{dt}}^{\lambda}(m, m')s(m')d\mu_{\lambda} (m')
&=
\frac{d}{dt}\int_M\kappa_{T_t}^{\lambda}(m,m')s(m')d\mu_{\lambda}(m') \\
&=
\int_M\frac{d}{dt}\bigg{(}\kappa_{T_t}^{\lambda}(m,m')s(m')\bigg{)}d\mu_{\lambda}(m') \\
&=
\int_M\bigg{(}\frac{d}{dt}\kappa_{T_t}^{\lambda}(m,m')\bigg{)}s(m')d\mu_{\lambda}(m'),
\end{align*}
and equation \eqref{kernel_der} follows.

We then compute
\begin{align*}
\frac{d}{dt}\Tr_g\bigg{(}T_t,\lambda_{t_0}\bigg{)}\bigg{\vert}_{t=t_0} &=
\frac{d}{dt}\bigg{\vert}_{t=t_0}
\int_{G/Z}\int_M\psi(hgh^{-1}m)\tr(hgh^{-1}\kappa_{T_t}^{\lambda_{t_0}}(hg^{-1}h^{-1}m,m))
d\mu_{\lambda_{t_0}}(m)d(hZ) \\
&=
\int_{G/Z}\int_M \psi(hgh^{-1}m)\tr\bigg{(}
hgh^{-1}\frac{d}{dt}\kappa_{T_{t_0}}^{\lambda_{t_0}}(hg^{-1}h^{-1}m,m)
\bigg{)}d\mu_{\lambda_{t_0}}(m)d(hZ) \\
&=
\Tr_g\bigg{(}
\frac{dT_t}{dt}\bigg{\vert}_{t=t_0}, \lambda_{t_0}
\bigg{)},
\end{align*}
which proves \eqref{eq der trace}.

\end{proof}

\section{Convergence}\label{sec conv}

\subsection{\v{S}varc--Milnor functions}

In our results on the first and third conditions in definition \ref{def g torsion}, we assume that a \emph{\v{S}varc--Milnor function} for the action exists. Such functions are also used in section \ref{sec metric indep} to obtain results on metric independence of equivariant analytic torsion, and
in subsection \ref{sec prod form}, to obtain a product formula for equivariant analytic torsion.

Recall that we assume that there is a nonzero, $G$-invariant, measure $d(hZ)$ on $G/Z$.
\begin{definition}\label{def SM}
A \emph{\v{S}varc--Milnor function} for the action by $G$ on $M$, with respect to $g \in G$, is a  function $l \colon (g) \to [0, \infty)$, where $(g)$ is the conjugacy class of $g$,  with the following properties:
\begin{enumerate}
\item For all $c>0$
\beq{eq conv Gaussian}
\int_{G/Z} e^{-cl(hgh^{-1})^2}\, d(hZ)
\eeq
converges.
\item The set
\beq{eq conj cpt}
X_r := 
\{  hZ \in G/Z; l(hgh^{-1}) \leq r\}
\eeq
is compact for all $r\geq 0$.
\item For all compact subsets $Y\subset M$, 
there are $b_1, b_2>0$ such that for all 
$m \in Y$  and $h \in G$,
\beq{eq SM}
d(hgh^{-1} m,m) \geq b_1 l(hgh^{-1}) - b_2.
\eeq
\end{enumerate}
\end{definition}

\begin{example}\label{ex SM GZ cpt}
If $G/Z$ is compact, then the zero function is a \v{S}varc--Milnor function.
\end{example}

\begin{example}\label{ex SM discr}
If $G = \Gamma$ is discrete and finitely generated, then the word length metric for a finite generating set
is a \v{S}varc--Milnor function.
Indeed, in such cases the counting measure is a $\Gamma$-invariant measure on $\Gamma/Z$. Now \eqref{eq conv Gaussian} is a sum over the conjugacy class of $g$, and convergence follows from the fact that the number of elements $\gamma \in \Gamma$ with $l(\gamma) \leq r$ grows at most exponentially in $r$. Similarly, the set \eqref{eq conj cpt} is finite and the \v{S}varc--Milnor lemma (see lemma 2 in \cite{Milnor}, or \cite{BDM, Svarc55}) implies an estimate of the form \eqref{eq SM}.
\end{example}


\begin{example}\label{ex SM conn}
Suppose that $G$ is a connected, real semisimple Lie group, and that $g \in G$ is a semisimple element. Then $G/Z$ has a $G$-invariant measure. Let $l$ be the distance function to the identity element in $G$, with respect to a left-invariant Riemannian metric. Then \eqref{eq conv Gaussian} converges by theorem 6 in \cite{HCDSII}; see also proposition 4.2 in \cite{HW2}. The conjugacy class of $g$ is closed in $G$, see theorem 1.2(1) in \cite{An08}, so \eqref{eq conj cpt} is compact. An estimate of the type \eqref{eq SM} holds by a version of the \v{S}varc--Milnor lemma for almost connected Lie groups, lemma \ref{lem SM connected} below. See also Lemmas 5.13 and 5.21 in \cite{PPST21}.
\end{example}

\begin{remark}
In examples \ref{ex SM GZ cpt}--\ref{ex SM conn}, there is a single \v{S}varc--Milnor function on $G$ that works for general $g \in G$, assumed to be semisimple in example \ref{ex SM conn}. The distance to the identity element with respect to a left-invariant distance function seems to be a natural candidate for a \v{S}varc--Milnor function in general.
\end{remark}

\begin{remark}
Consider the map $c_g \colon G/Z \to (g)$, given by $c_g(xZ) = xgx^{-1}$, for $x \in G$. This map is continuous with respect to the subspace topology on $(g)$.
So condition (2) in Definition \ref{def SM} implies that $l$ is proper. Conversely, if $c_g$ is a homeomorphism, then properness of $l$ implies condition (2) in Definition \ref{def SM}.

If $G$ is discrete, then $c_g$ is  a homeomorphism because it is a bijection. If $G$ is a connected Lie group, and $g$ is semisimple, then $(g) \subset G$ is an embedded submanifold (see Theorem 1.1(i) in \cite{An08}). The map $c_g$ is a smooth bijection, and by a computation of its derivative, a local diffeomorphism. So it is a diffeomorphism, and hence a homeomorphism.
\end{remark}

\begin{lemma}\label{lem quasi-isom}
Let $H$ be a group acting cocompactly on a manifold $N$. Then any two $H$-invariant Riemannian metrics on $N$ are quasi-isometric.
\end{lemma}
\begin{proof}
Let $Y\subset N$ be a compact subset such that $H\cdot Y = N$. Let $\lambda,\lambda'$ be two $H$-invariant Riemannian metrics on $N$. Let $SN \subset TN$ be the unit sphere bundle with respect to $\lambda$. Let $C$ be the maximum of the function
\[
v\mapsto \lambda'(v,v)
\]
on the compact set $SN|_{Y}$. Then by $H$-invariance of $\lambda$ and $\lambda'$, we have $\lambda'(v,v)\leq C \lambda(v,v)$ for all $v \in TN$. This implies that the Riemannian distance with respect to $\lambda'$ is at most $\sqrt{C}$ times the Riemannian distance with respect to $\lambda$.
\end{proof}

The following lemma is a version of the \v{S}varc--Milnor lemma for connected Lie groups (see also lemma 5.13 in \cite{PPST21}).
\begin{lemma}\label{lem SM connected}
Let $H$ be a connected Lie group. Let $N$ be a Riemannian manifold, on which $H$ acts properly, isometrically and cocompactly. Fix a point $x_0 \in N$. Then the map 
\beq{eq action H}
h\mapsto h\cdot x_0
\eeq
 is a quasi-isometry from $H$ to $N$, with respect to any left invariant Riemannian metric on $H$.
\end{lemma}
\begin{proof}
Let $K<H$ be maximal compact. 
Let $(\relbar, \relbar)_{\kh}$ be the $\Ad(K)$-invariant inner product on the Lie algebra $\kh$ of $H$ corresponding to the left invariant Riemannian metric on $H$. By Abels' slice theorem \cite{Abels}, there is a $K$-invariant, compact submanifold $Y\subset N$ such that the action map $H\times_K Y \to N$ is a $H$-equivariant diffeomorphism. Then
\beq{eq TN}
TN\cong H\times_K(TY \oplus \kp),
\eeq
where $\kp \subset \kh$ is the orthogonal complement of $\kk$. 
Let $\lambda_Y$ be a $K$-invariant Riemannian metric on $Y$. Consider the $H$-invariant Riemannian metric $\lambda$ on $N$ determined by
\[
\lambda([e, v+Z], [e, v'+Z']) := \lambda_Y(v,v')+(Z, Z')_{\kh},
\]
for $y \in Y$, $v,v' \in T_yY$ and $Z,Z' \in \kp$. By lemma \ref{lem quasi-isom}, this metric is quasi-isometric to any given $H$-invariant Riemannian metric on $N$, so it is enough to prove the claim for $\lambda$.

With respect to $\lambda$, the inclusion map $H \cdot x_0 \hookrightarrow N$ is an isometric embedding, and the  map \eqref{eq action H} descends to  an isometry $H/H_{x_0} \to H \cdot x_0$. Here we use connectedness of $H$; otherwise the Riemannian distance between connected components of $H/H_{x_0}$ is infinite. Let $d_N$ be the Riemannian distance for $\lambda$, and let $d_H$ and $d_{H/H_{x_0}}$ be the Riemannian distances on $H$ and  $H/H_{x_0}$, respectively, defined by $(\relbar, \relbar)_{\kh}$. Then for all $h,h' \in H$,
\[
d_N(hx_0, h'x_0) = d_{H/H_{x_0}} (hH_{x_0}, h'H_{x_0}).
\]
As $H_{x_0}$ is compact, the quotient map $H \to H/H_{x_0}$ is a quasi-isometry. We then find that there are $b_1, b_2>0$ such that for all $h, h' \in H$
\[
{b_1} d_H(h,h')-b_2 \leq
d_N(hx_0, h'x_0) \leq \frac{1}{b_1} d_H(h,h')+b_2.
\]
Furthermore, by compactness of $Y$, every point in $N$ lies within the diameter of $Y$ from the image of \eqref{eq action H}.
\end{proof}

\begin{lemma} \label{lem SM product}
Suppose that for $j=1,2$, $G_j$ is a unimodular, locally compact group 
 acting properly, isometrically and cocompactly on an oriented Riemannian manifold $M_j$, preserving the orientation. Let $g_j \in G_j$ be an element with centraliser $Z_j<G_j$, such that $G_j/Z_j$ has a $G_j$-invariant measure. 
If $l_j$ is a \v{S}varc--Milnor function for the action by $G_j$ on $M_j$ with respect to $g_j$, then
\[
l( h_1 g_1 h_1^{-1}, h_2 g_2 h_2^{-1}) := l_1(h_1 g_1 h_1^{-1})+ l_2(h_2 g_2 h_2^{-1})
\]
for $h_j \in G_j$, defines a \v{S}varc--Milnor function for the action by $G_1 \times G_2$ on $M_1 \times M_2$ with respect to $(g_1, g_2)$.
\end{lemma}
\begin{proof}
This follows from the definitions, and the fact that the Riemannian distance $d_{M_1 \times M_2}$ on $M_1 \times M_2$ is related to the Riemannian distances $d_{M_j}$ on the manifolds $M_j$ by
\[
d_{M_1 \times M_2}(  (m_1, m_2), (m_1', m_2')) = \sqrt{d_{M_1}(m_1, m_1')^2 + d_{M_2}(m_2, m_2')^2 } \geq \frac{1}{\sqrt{2}} \bigl(d_{M_1}(m_1, m_1') + d_{M_2}(m_2, m_2') \bigr),
\]
for all $m_j, m_j' \in M_j$.
\end{proof}

\subsection{Small $t$ asymptotics}\label{sec small t}

In this subsection, we assume that there is a {\v{S}varc--Milnor function} $l$ for the action by $G$ on $M$, with respect to $g$, see definition \ref{def SM}.

We consider a family of $g$-trace class operators on a Hermitian $G$-vector bundle on $M$, such as $\Bigwedge^p T^*M \otimes E$, parametrised by $t>0$, with smooth kernels 
$\kappa_t$.
%
%
We start with some lemmas in cases where the kernels $\kappa_t$ have one of the following properties shared by heat kernels.
\begin{enumerate}
\item 
 There are $a_1, a_2, a_3>0$ such that for all $m,m' \in M$ and $t \in (0,1]$,
\beq{eq decay kappa}
\|\kappa_t(m,m')\| \leq a_1 t^{-a_2} e^{-a_3 d(m,m')^2/t},
\eeq
where $d$ is the Riemannian distance on $M$.
\item For some $r \in \R$, there is an asymptotic expansion on the diagonal $\Delta(M) \subset M \times M$ of the form
\beq{eq as exp kappa}
\kappa_t|_{\Delta(M)} \sim t^{-r}\sum_{j = 0}^{\infty} \kappa_{(j)} t^j,
\eeq
as $t \downarrow 0$, 
with respect to the $C^k$-norms on compact subsets of $\Delta(M)$.
\end{enumerate}
\begin{example}\label{ex kappa t heat op}
The heat operator  of a Laplace-type operator such as $\Delta_E^p$ has the above properties (1) and (2).
See proposition 4.2(1) in \cite{CGRS14} (based on \cite{Greiner71, Kordyukov91}) for the first condition.
Here one uses the fact that $M$ has bounded geometry, because the Riemannian metric is invariant under $G$ and $M/G$ is compact. This result  in  \cite{CGRS14} is stated for squares of Dirac operators, but the proof in fact applies to uniformly elliptic differential operators on vector bundles of bounded geometry. Hence it also applies to our operators $\Delta_E^p$.
 
The second condition follows from standard heat kernel asymptotics, where one can take $r=\dim(M)/2$. Most constructions apply to compact manifolds, see \cite{BGV, Roe98}, and hold with respect to global $C^k$-norms. The constructions imply that the analogous asymptotic expansion holds on noncompact, complete manifolds, with respect to $C^k$-norms on compact sets. See proposition 4.2(2) in \cite{CGRS14} or theorem 6.3 in \cite{Wang14} for explicit arguments. Those arguments apply on the diagonal, but off the diagonal, a Gaussian factor vanishes to all orders in $t$.

The composition of the heat operator of a Laplacian and a $G$-equivariant differential operator of order $d$ also has the properties (1) and (2), see proposition 4.2 in  \cite{CGRS14}.
\end{example}

In some places, particularly for estimates for large $t$, we will consider more general bounds than \eqref{eq decay kappa}, of the form
\beq{eq decay kappa gen}
\|\kappa_t(m,m')\| \leq F_1(t) e^{-a_3 d(m,m')^2/t},
\eeq
for a function $F_1\colon (0, \infty) \to (0, \infty)$.

%

Let $b_1$, $b_2>0$ be such that \eqref{eq SM} holds for all $m$ in the (compact) support of $\psi$. 
The following estimate is a direct consequence of the assumptions, and will be used in various places.
\begin{lemma}\label{lem Trg Xcompl}
There is a compact subset $X \subset G/Z$ such that for all $t>0$ for which \eqref{eq decay kappa gen} holds, 
\begin{multline} \label{eq Trg Xcompl}
\int_{\Xcompl}\int_M \psi(hgh^{-1}m) |\tr\bigl(hgh^{-1}\kappa_t(hg^{-1}h^{-1}m,m) \bigr)|\, dm\, d(hZ)
\\
\leq  \int_M \psi(m)\, dm \, F_1(t) \int_{\Xcompl} e^{-a_3 b_1^2 l(hgh^{-1})^2 /2t}d(hZ).
\end{multline}
In particular, the left hand side converges.
\end{lemma}
\begin{proof}
Using $G$-invariance of $\kappa_t$ and the trace property of $\tr$, we see that if $m' = hgh^{-1}m \in M$, then
\[
 \tr\bigl(hgh^{-1}\kappa_t(hg^{-1}h^{-1}m,m) \bigr) =  \tr\bigl(hgh^{-1}\kappa_t(hg^{-1}h^{-1}m',m') \bigr).
\]
Hence by a substitution $m' = hgh^{-1}m \in M$, the left hand side of \eqref{eq Trg Xcompl} equals
\[
\int_{\Xcompl}  \int_M \psi(m') |\tr\bigl(hgh^{-1}\kappa_t(hg^{-1}h^{-1}m',m') \bigr) | \, dm'\, d(hZ),
\]
for any $X \subset G/Z$ with measurable complement such that the integral converges.
It  is immediate from this equality and the assumptions on $l$ and $\kappa_t$, that the left hand side of \eqref{eq Trg Xcompl}  is at most equal to 
\[
\int_M \psi(m)\, dm \,   F_1(t) \int_{\Xcompl} e^{-a_3 ( b_1 l(hgh^{-1}) - b_2)^2 /t}d(hZ).
\]
We now take 
\[
X:= \{hZ \in G/Z; l(hgh^{-1}) \leq 4 b_2/b_1\},
\]
which is compact by assumption on $l$. Then for all $hZ \in \Xcompl$,
\[
  (b_1 l(hgh^{-1}) - b_2)^2 \geq b_1^2 l(hgh^{-1})^2/2.
\]
\end{proof}

\begin{lemma}\label{lem kappa t g tr cl}
The operator with kernel $\kappa_t$ is $g$-trace class for all $t>0$ such that \eqref{eq decay kappa gen} holds.
\end{lemma}
\begin{proof}
The integrals on the right hand side of \eqref{eq Trg Xcompl} converge because $\psi$ has compact support and \eqref{eq conv Gaussian} converges by assumption. As $X$ is compact, this implies that the integral \eqref{eq def Trh} converges absolutely.
\end{proof}

\begin{lemma}\label{lem small t large l}
Suppose that \eqref{eq decay kappa} holds for all $t \in (0,1]$. Then
there is a compact subset $X \subset G/Z$ such that for all $a>0$, the expression
\beq{eq bounded 01}
 t^{-a} \int_{\Xcompl}  \int_M \psi(hgh^{-1}m) \tr\bigl(hgh^{-1}\kappa_t(hg^{-1}h^{-1}m,m) \bigr)\, dm\, d(hZ)
\eeq
is bounded in $t \in (0,1]$. In particular, the integral
\[
\int_0^1 t^{-a} \int_{\Xcompl}  \int_M \psi(hgh^{-1}m) \tr\bigl(hgh^{-1}\kappa_t(hg^{-1}h^{-1}m,m) \bigr)\, dm\, d(hZ)\, dt
\]
converges absolutely.
\end{lemma}
\begin{proof}
Let $X \subset G/Z$ be as in the proof of lemma \ref{lem Trg Xcompl}. Then $l(hgh^{-1}) > 4b_2/b_1>0$  for all $hZ \in G/Z \setminus X$. So by lemma \ref{lem Trg Xcompl}, with $F_1(t) = a_1 t^{-a_2}$, the absolute value of \eqref{eq bounded 01} is at most equal to a constant times
\[
t^{-a-a_2} e^{-4a_3b_2^2/t}  \int_{\Xcompl} e^{-a_3 b_1^2 (l(hgh^{-1})^2 - 8b_2^2/b_1^2) /2t}d(hZ).
\]
Here the factor $t^{-a-a_2} e^{-4a_3b_2^2/t}$ is bounded in $t \in (0,1]$, and
\[
l(hgh^{-1})^2 - 8b_2^2/b_1^2 \geq l(hgh^{-1})^2/2,
\]
so 
\[
 \int_{\Xcompl} e^{-a_3 b_1^2 (l(hgh^{-1})^2 - 8b_2^2/b_1^2) /2t}d(hZ) \leq 
  \int_{\Xcompl} e^{-a_3 b_1^2 l(hgh^{-1})^2 /4}d(hZ).
\]
This integral converges because \eqref{eq conv Gaussian} does by assumption.
\end{proof}

We adapt a well-known off-diagonal estimate for heat kernels in terms of their on-diagonal values. 
\begin{lemma}\label{lem off diag}
Suppose that a family of kernels $(\kappa_t)_{t>0}$ has the semigroup property, and that the operators defined by  these kernels are self-adjoint. Then 
for all $m \in M$, $x \in G$ and $t>0$,
\[
|\tr(x \kappa_t(x^{-1}m,m))| \leq
 \tr(\kappa_t(m,m)).
\]
\end{lemma}
\begin{proof}
By the semigroup property and self-adjointness,
\beq{eq heat off diag}
\tr (x \kappa_t(x^{-1}m,m)) = \int_M \tr \bigl( x \kappa_{t/2}(x^{-1}m,m') \kappa_{t/2}(m,m')^* \bigr)\, dm'.
\eeq
By the Cauchy--Schwartz inequality, the absolute value of the right hand side of \eqref{eq heat off diag} is less than or equal to
\[
\left(
 \int_M \tr \bigl( x \kappa_{t/2}(x^{-1}m,m') (x \kappa_{t/2}(x^{-1}m,m'))^* \bigr)\, dm'\right)^{1/2}
 \left(
 \int_M \tr \bigl(  \kappa_{t/2}(m,m') ( \kappa_{t/2}(m,m'))^* \bigr)\, dm'\right)^{1/2}.
\]
By versions of  \eqref{eq heat off diag} with $x=e$, and $G$-invariance of $\kappa_{t}$, this expression equals $ \tr( \kappa_{t}(m,m))$.
\end{proof}

\begin{lemma} \label{lem small t small l}
Suppose that $\kappa_t$ has an asymptotic expansion of the form \eqref{eq as exp kappa}. 
For all compact subsets $X \subset G/Z$, the expression
\beq{eq small t small l}
\frac{1}{\Gamma(s)}
\int_0^1 t^{s-1} \int_X  \int_M \psi(hgh^{-1}m) \tr\bigl(hgh^{-1}\kappa_t(hg^{-1}h^{-1}m,m) \bigr)\, dm\, d(hZ)\, dt
\eeq
converges for $s$ with large enough real part. This expression has a meromorphic continuation to $\C$ that is regular at $s=0$.
\end{lemma}
\begin{proof}
This  is a modest generalisation of results from
\cite{MP49, Seeley67}; see also 
section 9.6 of \cite{BGV} and 
lemma 3.1 in \cite{Mathai92}.


By lemma  \ref{lem off diag} and compactness of $X$, it is enough to prove the claim for 
\beq{eq small t small l e}
\frac{1}{\Gamma(s)}
\int_0^1 t^{s-1}   \int_M \psi(m) \tr\bigl(\kappa_t(m,m) \bigr)\, dm\,  dt.
\eeq
%
%
By compactness of $\supp(\psi)$, the asymptotic expansion \eqref{eq as exp kappa} yields an asymptotic expansion
\[
  \int_M \psi(m) \tr\bigl(\kappa_t(m,m) \bigr)\, dm \sim t^{-r}\sum_{j=0}^{\infty} a_j t^j.
\]
So for all $n \in \N$, \eqref{eq small t small l e} equals
\beq{eq small l 2}
 \frac{1}{\Gamma(s)}
\sum_{j=0}^{n} a_j 
\int_0^1 t^{j-r + s-1} \, dt + \frac{1}{\Gamma(s)} \int_0^1 R_n(t, s)\, dt,
\eeq
for a function $R_n$ that is $\cO(t^{s+n - r})$ as $t \downarrow 0$. Therefore, because 
$\Gamma^{-1}$ is holomorphic,  the second term is a holomorphic function of $s$ on the set of $s \in \C$ such that $\Real(s)>-n-1+r$. The first term in \eqref{eq small l 2} equals
\[
\frac{1}{\Gamma(s)}
\sum_{j=0}^{n} \frac{a_j}{j-r + s}. 
\]
It follows that we can extend \eqref{eq small t small l}  meromorphically to all $s \in \C$, by taking $n > -\Real(s)-1-r$. 
As $\Gamma(s)^{-1} = \cO(s)$ as $s \to 0$, this meromorphic extension is regular at $s=0$.
\end{proof}

We reach our main conclusion on conditions (1) and (3) in definition \ref{def g torsion}.
\begin{proposition}\label{prop conv small t}
Suppose that there is a {\v{S}varc--Milnor function} for the action by $G$ on $M$, with respect to $g$.
If  $P$ is $g$-trace class, then conditions (1) and (3) in definition \ref{def g torsion} hold.
\end{proposition}
\begin{proof}
Applying lemma \ref{lem kappa t g tr cl} to example \ref{ex kappa t heat op}, we find that the heat operator $e^{-t \Delta_E^p}$ is $g$-trace class for all $p$. As $P$ is $g$-trace class by assumption, condition (1) in definition \ref{def g torsion} holds.

Lemmas \ref{lem small t large l} and \ref{lem small t small l} applied to example \ref{ex kappa t heat op} imply that condition (3) in definition \ref{def g torsion} holds.
\end{proof}

\begin{remark}
In cases where $0$ is isolated in the spectrum of $\Delta_E$, the projection $P$ was shown to be $g$-trace class in several situations. See proposition 6 in \cite{Lott99},  lemma 7.7(2)  in \cite{PPST21b} and lemma 6.6 in \cite{HWW}. We are mainly interested in the case $P=0$, which is trivially $g$-trace class. 
\end{remark}

\subsection{Large $t$ asymptotics: Delocalised Novikov--Shubin invariants}\label{sec large t conv}

It is not known in general if the second condition in definition \ref{def g torsion} holds. A stronger condition is that the \emph{delocalised Novikov--Shubin numbers}
\beq{eq def NS}
\alpha^p_g := \sup\{ \alpha>0; \Tr_g(e^{-t\Delta_E^p} - P_p) = \cO(t^{-\alpha}) \text{ as $t \to \infty$}\}
\eeq
are positive.

In the case where $E = M \times \C$, $\nabla^E = d$, and 
$G$ is the fundamental group of a compact manifold acting on the universal cover $M$, 
the numbers $\alpha_e^p$, for $g=e$, are called Novikov--Shubin invariants. 
Novikov and Shubin showed that these are independent of the Riemannian metric \cite{ES89, NS86}. This was generalised to topological manifolds in corollary 33 in  \cite{Lott92c}. 
Gromov and Shubin \cite{GS91} sharpened the result by Novikov and Shubin by showing that the Novikov--Shubin invariants are homotopy invariants; see also Theorems 2.55 and 2.58 in \cite{Lueck02}.  

The Novikov--Shubin invariants are known in many cases, see for example section VII of \cite{Lott92c}. See  section 1 of \cite{Mathai92} for results on positivity of these invariants. If $G$ is a connected, non-compact real semisimple Lie group, $K<G$ is maximal compact, and $\Gamma<G$ is a cofinite-volume, torsion-free discrete subgroup, then the Novikov--Shubin invariants of the locally symmetric space $\Gamma \setminus G/K$ are infinite if $\rank(G) = \rank(K)$, and otherwise either infinite or equal to $(\rank(G)-\rank(K))/2$, see Section 11 of \cite{LM00}.

It is an interesting question if or when the numbers $\alpha^p_g$ are independent of the metric in general, as they are in the special case of the Novikov--Shubin invariants. The following is an immediate consequence of existing results on Novikov--Shubin invariants. It applies for example if $G$ is a connected, semisimple Lie group, and $M=G/K$ for a maximal compact $K<G$.
\begin{lemma}\label{lem alpha e invar}
Suppose that $M$ is simply connected, and that there is a discrete subgroup $\Gamma<G$ that acts freely on $M$, and such that $G/\Gamma$ is compact. Then $\alpha_e^p$ is a homotopy invariant of $M/\Gamma$ for all $p$.
\end{lemma}
\begin{proof}
By lemma \ref{lem TrG TrH}, the number $\alpha_e^p$, defined for the action by $G$, equals  the corresponding number $(\alpha_e^p)_{\Gamma}$ defined for the action by $\Gamma$. The latter action is the action by $\Gamma = \pi_1(M/\Gamma)$ on the universal cover $M$ of $M/\Gamma$, so $(\alpha_e^p)_{\Gamma}$ is a homotopy invariant of $M/\Gamma$ by  \cite{GS91}.
\end{proof}
 
In view of the second condition in definition \ref{def g torsion}, it is also interesting when the numbers $\alpha_g^p$ are positive. 
  
If $M$ is compact, then 
\[
\Tr_g(e^{-\Delta_E^p} - P_p) = \sum_{\lambda \in \spec(\Delta_E^p) \setminus \{0\}} \tr(g|_{\ker(\Delta_E^p - \lambda)}) e^{-t \lambda}.
\]
It follows from Weyl's law that this decays exponentially in $t$, so $\alpha_g^p=\infty$.
Hence $\alpha_g^p$ is trivially independent of the Riemannian metric, because compactness is a topological property. They are nonnegative  in any case when $G/Z$ is compact, by the following estimate.
%
\begin{proposition} \label{prop heat large t bdd}
Let $\kappa_t$ be the Schwartz kernel of $e^{-t\Delta_E^p}$. Then there are $C,a_3>0$ such that for all $t\geq 1$, and all $m,m' \in M$, 
\[
\|\kappa_t(m,m')\| \leq C e^{-a_3 d(m,m')^2/t}.
\]
\end{proposition}
This is well-known for $p=0$; we give a proof in the general case in Appendix \ref{app heat decay large t}.



The following fact follows from proposition \ref{prop GZ cpt}. The condition that 
$\alpha_e^p>0$ does not depend on the metric in the setting of lemma \ref{lem alpha e invar}.
\begin{lemma}
If $G/Z$ is compact and  $\alpha_e^p>0$, then $\alpha_g^p>0$.
\end{lemma} 
  
The following case is relevant to hyperbolic geometry.
\begin{lemma} \label{lem NS hyp}
Let $G = \SO_0(n,1)$, and consider its maximal compact subgroup $K = \SO(n)$. Let $g \in G$ be a hyperbolic element. Consider the action by $G$ on $M = G/K$. Suppose that   $E = M \times \C$ and $\nabla^E = d$. 
Then $\alpha_g^p \geq 1/2$ for all $p$.
\end{lemma}
\begin{proof}
For $g=e$, 
this was deduced from theorem 2 in \cite{Fried86} in theorem 6.2 in \cite{Mathai92}. The same method applies to nontrivial hyperbolic elements. 
By theorem 2 in \cite{Fried86}, there are constants $C_{p,g}>0$ such that for all $t>0$, 
\[
\Tr_g(e^{-t \Delta^p_E}) = \left( C_{p, g}e^{-tc_p^2}+  C_{p-1, g}e^{-tc_{p-1}^2}\right) t^{-1/2} e^{-l(g)^2/4t} .
\]
Here $c_p = |(n-1)/2 - p|$, and $l(g)>0$. 

In the current setting, the kernel of $\Delta_E$ is trivial, so the projection $P$ plays no role. This was pointed out in section 5 of \cite{Mathai92}. The group $G$ does not have discrete series representations by Harish-Chandra's equal rank criterion, and the kernel of $\Delta_E$ consists of discrete series representations by theorem 6.1 in \cite{Connes82}. We find that
$|\Tr_g(e^{-t \Delta^p_E} - P_p)| = |  \Tr_g(e^{-t \Delta^p_E})|$. 
So  $\alpha_g^{(n-1)/2}=\alpha_g^{(n+1)/2} = 1/2$  if $n$ is odd, and  $\alpha_g^p= \infty$ if $p \not= (n\pm1)/2$. 
\end{proof}

A version of lemma \ref{lem NS hyp} for $n=1$ and elliptic $g$ follows from the explicit computation in lemma \ref{lem Tgt n=1} below. This implies that $\cT_g(t) = \cO(t^{-1/2})$ in that case. For a similar estimate for individual delocalised Novikov--Shibon numbers, the argument would have to be refined, but that is not necessary for convergence of equivariant analytic torsion.

Positivity of $\alpha_g^p$ depends on the volume growth rate of the sets \eqref{eq conj cpt} and the long-time decay behaviour of heat kernels. We make this more explicit, in a way that is also used in the metric independence proof in section \ref{sec metric indep}.

 Suppose we are given a \v{S}varc--Milnor function $l$, and a family of smooth kernels 
 $\kappa_t$ satisfying \eqref{eq decay kappa gen}. Let $b_1$ and $b_2$ be as in \eqref{eq SM} for $Y = \supp(\psi)$, and $F_1$ and $a_3$ as in \eqref{eq decay kappa gen}.  We write $a:= a_3 b_1^2/2$.
 For $r,t>0$, set
 \begin{align}
 F_2(r)&:= \vol(X_{r+1}\setminus X_r); \label{eq def F2}\\
 F_3(t)&:= \sum_{j=0}^{\infty} F_2(j) e^{-aj^2/t}, \label{eq def F3}
 \end{align}
 assuming the latter sum converges. It follows from the definitions that for all $t>0$, 
 \beq{eq Gaussian F3}
 \int_{G/Z} e^{-a l(hgh^{-1})^2/t} d(hZ) \leq F_3(t).
 \eeq
 The behaviour of $F_3$ as $t\to \infty$ will be important, we give bounds in two relevant cases.
 \begin{lemma}\label{lem F2 F3}
 Let $b \in \R$. 
 \begin{itemize}
 \item[(a)] If $F_2(r) = \cO(r^b)$ as $r \to \infty$, then $F_3(t) = \cO(t^{\frac{b+1}{2}})$ as $t \to \infty$.
  \item[(b)] If $F_2(r) = \cO(e^{br})$ as $r \to \infty$, then $F_3(t) = \cO(t^{1/2} e^{b^2 t/4a})$ as $t \to \infty$.
 \end{itemize}
 \end{lemma}
 \begin{proof}
In case (a), note that 
\[
\sum_{j=0}^{\left \lceil \sqrt{\frac{bt}{2a}} \right\rceil} j^b e^{-aj^2/t}  = \cO(t^{\frac{b+1}{2}}).
\]
The summands are decreasing in $j$ for $j \geq \sqrt{\frac{bt}{2a}}$, so
\[
\sum_{j=\left \lceil \sqrt{\frac{bt}{2a}} \right\rceil+1}^{\infty}  j^b e^{-aj^2/t}\leq \int_{0}^{\infty} x^b e^{-ax^2/t}\, dx = {t}^{\frac{b+1}{2}}\int_{0}^{\infty} y^b e^{-ay^2}\, dy.
\]
In case (b), 
\[
F_3(t) = \cO\left(  e^{b^2 t/4a} \sum_{j=0}^{\infty} e^{-\frac{a}{t} \left(j-\frac{bt}{2a} \right)^2} \right),
\]
and
\[
\sum_{j=0}^{\infty} e^{-\frac{a}{t} \left(j-\frac{bt}{2a} \right)^2} \leq 2\int_{-1}^{\infty} 
 e^{-\frac{a}{t} \left(x-\frac{bt}{2a} \right)^2} \, dx \leq 
2\sqrt{\frac{t}{a}}\int_{\R} e^{-y^2}\, dy
= 2\sqrt{\frac{\pi t}{a}}.
\]
 \end{proof}

\begin{proposition}\label{prop NS cpt}
Suppose that the Schwartz kernel $\kappa_t$ of $e^{-t\Delta_E^p}$ satisfies \eqref{eq decay kappa gen} for all $t\geq 1$. 
Suppose that $\ker(\Delta_E) = \{0\}$, and that $F_1(t)F_3(t) = \cO(t^{-\alpha})$, for 
$\alpha>0$. 
Then $\alpha_g^p \geq \alpha$. 
\end{proposition}
\begin{proof}
Let $X \subset G/Z$ be as in lemma \ref{lem Trg Xcompl}. Then for all $t \geq 1$, 
\begin{multline*}
|\Tr_g(e^{-t\Delta_E^p})| 
\leq  \int_M \psi(m)\, dm \, F_1(t) \int_{\Xcompl} e^{-a_3 b_1^2 l(hgh^{-1})^2 /2t}d(hZ)\\
+
F_1(t) \left|
\int_X \int_M \chi(hgh^{-1}m) e^{-a_3 d(hg^{-1}h^{-1}m,m)^2/t}\, dm\, d(hZ)
 \right|.
\end{multline*}
By \eqref{eq Gaussian F3}, the first term is $\cO(F_1(t)F_3(t))$, and the second term is 
$\cO(F_1(t))$ by compactness of $X$.
\end{proof}

By proposition \ref{prop NS cpt} and our metric independence results, theorem \ref{metric_indep_cconj} and corollary \ref{cor metric indep}, the behaviour of  $F_1(t)F_3(t)$ as $t \to \infty$ is relevant. This behaviour specifies how fast heat kernels (and their derivatives) should decay as $t \to \infty$ for these results to hold, where the required decay rate, described by $F_1$, is determined by the volume growth rate of the sets \eqref{eq conj cpt}, through the function $F_3$.
  
On compact manifolds, the function $F_1$ is bounded below in general if $\kappa_t$ is the heat kernel, but decays exponentially if $\kappa_t$ is the kernel of $e^{-t\Delta_E}-P$. The latter is enough for our purposes, see remark \ref{rem term P}.
  
Determining the exact form of the decay rate $F_1$ on noncompact manifolds is a subtle problem in general, closely related to the open question if $\alpha_e^p$ is always positive. Many results have been obtained for scalar heat kernels, based for example on isoperimetric or Faber--Krahn inequalities; see \cite{Grigoryan09}. Generalising such estimates to heat kernels on differential forms is a nontrivial problem. Some results were obtained in \cite{Coulhon20, Coulhon07, Devyver14}. In these results, it is assumed that the negative part of a curvature operator (which is the Ricci curvature in the case of one-forms), decays fast enough so that it lies in a suitable $L^p$-space. In our context, this curvature operator is $G$-invariant, so its negative part decays at infinity if and only if it is zero.
We apply such results to prove corollary \ref{cor metric indep slow growth}.

\section{Metric independence}\label{sec metric indep}

In this section, we prove that if the conjugacy class of $g$ has slow enough volume growth, 
and the $E$-valued $L^2$-cohomology of $M$ vanishes, then the equivariant analytic torsion is independent of the $G$-invariant Riemannian metric on $M$, in an appropriate sense. 
It is also independent of  the $G$-invariant Hermitian metric on $E$ that is preserved by $\nabla^E$, by more elementary arguments. 
See theorem \ref{metric_indep_cconj} and corollary \ref{cor metric indep}. The  proof is based on an asymptotic 
analysis strategy like the one  in \cite{LR91}. Additional ingredients are the use of separate arguments on a compact subset of $G/Z$ and on its complement, and long time estimates for $g$-traces of certain operators related to heat operators.

\subsection{Independence of the metric on $E$ and of rescaling by constants}

We start with two relatively straightforward remarks, lemma \ref{lem indep Herm metric} and remark \ref{rem rescaling}. The first is about
independence of equivariant analytic torsion of the $G$-invariant Hermitian metric on $E$ preserved by $\nabla^E$. 
The Hodge star operator $\star$ associated to a  Riemannian metric extends to a vector bundle homomorphism
\beq{eq Hodge star E}
\star \otimes 1_E \colon \Bigwedge^p T^*M \otimes E \to  \Bigwedge^{\dim(M) - p} T^*M \otimes E
\eeq
for all $p$; we also denote this by $\star$.  
\begin{lemma}\label{lem indep Herm metric}
Suppose that the conditions of definition \ref{def g torsion} hold for a given $G$-invariant Hermitian metric on $E$ that is preserved by $\nabla^E$. Then these conditions hold  for all $G$-invariant Hermitian metrics on $E$ which are preserved by $\nabla^E$, and the resulting equivariant analytic torsion is the same for all such metrics.
\end{lemma}
\begin{proof}
Fix a Hermitian metric $(\relbar, \relbar)_E$ on $E$. 
Define the operation $\wedge^E\colon \bigl(\Bigwedge^* T^*M \otimes E\bigr) \otimes \bigl(\Bigwedge^* T^*M \otimes E\bigr) \to \Bigwedge^*T^*M$ as the sesquilinear extension of 
\[
(\omega \otimes v) \wedge^E (\nu \otimes w) = (v,w)_E \, \omega \wedge \nu,
\]
for $m \in M$, $\omega, \nu \in \Bigwedge^*T^*_mM$ and $v,w \in E_m$. 
 If $\nabla^E$ preserves $(\relbar, \relbar)_E$, then one has the Leibniz-type relation
\beq{eq Leibniz nabla E}
d(\alpha \wedge^E \beta) = (\nabla^E \alpha) \wedge^E \beta + (-1)^p \alpha \wedge^E (\nabla^E \beta),
\eeq
for all $\alpha \in \Omega^p(M; E)$ and $\beta \in \Omega^*(M; E)$. 

Noting that the left hand side of \eqref{eq Leibniz nabla E} has integral zero over $M$,  and using $\star^2 = (-1)^{F\dim(M) + F}$, one finds that
\[
(\nabla^E)^* =  (-1)^{F \dim(M)  + \dim(M) + 1} \star \nabla^E \star
\]
is  the formal adjoint of $\nabla^E$ with respect to the $L^2$-inner product determined by the Riemannian metric and $(\relbar, \relbar)_E$.
The right hand side does not depend on the Hermitian metric on $E$, hence neither does the left hand side. We find that the Laplacian $\Delta_E$ is the same operator for every Hermitian metric on $E$ that is preserved by $\nabla^E$. This implies the lemma.
%
%
%
\end{proof}

For the rest of this section, we focus on  independence of equivariant analytic torsion of the Riemannian metric. 
\begin{remark}\label{rem rescaling}
If equivariant analytic torsion is expected to be independent of the Riemannian metric, 
then at the very least, it should be unchanged if the Riemannian metric on $M$  is multiplied by a positive constant. Then the adjoint connection $(\nabla^E)^*$ is multiplied by a positive constant $c$, and so is the Laplacian $\Delta_E$. If $\ker(\Delta_E) = \{0\}$, then it follows that the function \eqref{eq curly T} for the rescaled metric is $t\mapsto \cT_g(ct)$, where $\cT_g$ is the function  \eqref{eq curly T} for the initial metric (here we also use a special case of proposition \ref{$g$-trace_der1}). We then find
\[
\begin{split}
 \left. \frac{d}{ds}\right|_{s=0} \frac{1}{\Gamma(s)} \int_0^{
1} t^{s-1} \cT_g(ct) \, dt
&=  \left. \frac{d}{ds}\right|_{s=0} \frac{1}{\Gamma(s)} \int_0^{
c} t^{s-1} \cT_g(t) \, dt; 
\\
\int_1^{
\infty} t^{-1}  \cT_g(ct) \, dt &= \int_c^{
\infty} t^{-1}  \cT_g(t) \, dt. 
\end{split}
\]
By remark \ref{rem int bdry}, these equalities imply that equivariant analytic torsion for the rescaled metric equals equivariant analytic torsion for the original metric. 
\end{remark}
%

\subsection{Derivatives of equivariant analytic torsion}\label{sec der torsion}


We start the proof of our metric independence results by defining some operators that will allow us to rewrite the definition
of the equivariant analytic torsion, see equation \eqref{eq def deloc torsion}, 
in a compact form.


We will assume that $(\lambda_{\epsilon})_{\epsilon \in \R}$ is a 1-parameter 
family of $G$-invariant Riemannian metrics on $M$. Furthermore, we assume that
\begin{itemize}
\item there is a \v{S}varc--Milnor function for the action of $G$ with respect to $g$;
 \item the projection $P$ is $g$-trace class (this condition is independent of the $G$-invariant Riemannian metric; see the proof of lemma \ref{form1}).
\end{itemize}
Then conditions (1) and (3) in definition \ref{def g torsion} hold by proposition \ref{prop conv small t}. We assume that condition (2) holds as well, for all $\epsilon$. By  proposition \ref{prop NS cpt}, this is the case if $F_1(t)F_3(t) = \cO(t^{-\alpha})$ for an $\alpha>0$, which we assume in theorem \ref{metric_indep_cconj} and corollary \ref{cor metric indep}. Then $T_g(\nabla^E)$ is well-defined for all $\epsilon$.

We make further assumptions below, see theorem \ref{metric_indep_cconj} for the full set of assumptions.

Using the  number operator $F$ on $\Bigwedge^* T^*M \otimes E$,
we can  express the equivariant analytic torsion, see 
equation \eqref{eq def deloc torsion}, in the following form
\begin{align}\label{eqtor_compact}
-2\log T_g(\nabla^E) &= \frac{d}{ds}\bigg{\vert}_{s=0}\frac{1}{\Gamma(s)}
\int_0^1t^{s-1}\Tr_g\bigg{(}
(-1)^FF\big{(}e^{-t\Delta_E} - P \big{)}
\bigg{)}dt \nonumber \\ 
&\hspace{1cm}+
\int_1^{\infty}t^{-1}
\Tr_g\bigg{(}
(-1)^FF\big{(}e^{-t\Delta_E} - P \big{)}
\bigg{)}dt.
\end{align}

Noting that the Hodge star operator depends on the metric, we let 
$V = \bigg{(}\frac{d}{d\epsilon}\star\bigg{)}\star^{-1}$. To simplify notation, we write $d:= \nabla^E$ and $\delta := (\nabla^E)^*$, where the adjoint is defined with respect to the metric depending on $\epsilon$.
The $d$ operator is defined
independent of the metric, hence $\frac{d}{d\epsilon}d = 0$. However, 
\begin{align}
\frac{d}{d\epsilon}\delta &= \frac{d}{d\epsilon}(\pm\star d\star) = [V, \delta]; \quad \text{so that} \nonumber \\ 
\frac{d}{d\epsilon}\Delta_E &= \{d, [V,\delta]\}, \label{der_ops_2}
\end{align}
where $\{\relbar, \relbar\}$ denotes anticommutation. 

We will need the following commutation formulas, which are straightforward computations: 
\begin{equation}\label{commute_ops}
[F, d] = d  \quad \text{and} \quad  [F,\delta] = -\delta.
\end{equation}

As there is a \v{S}varc--Milnor function for the action by $G$, with respect to $g$,  example \ref{ex kappa t heat op} and lemma \ref{lem kappa t g tr cl} imply that the composition of $e^{-t\Delta_E}$ with a $G$-equivariant differential operator is $g$-trace class. We will use this without mentioning it explicitly in the rest of this section.

%

\begin{lemma} \label{lem anticomm} 
For all $t>0$ and all $\epsilon$, 
\beq{eq anticomm}
\Tr_g\bigg{(}
(-1)^FF\{d, [V,\delta]\}e^{-t\Delta_E}
\bigg{)} =
-\Tr_g\bigl{(}V(-1)^F\Delta_E e^{-t\Delta_E}\bigr{)}.
\eeq
\end{lemma}
\begin{proof}
Expanding the bracket terms in $\Tr_g\bigg{(}
(-1)^FF\{d, [V,\delta]\}e^{-t\Delta_E}
\bigg{)}$,  we obtain
\begin{multline*}
\Tr_g\bigg{(}
(-1)^FF\{d, [V,\delta]\}e^{-t\Delta_E}
\bigg{)} =
\Tr_g\bigg{(}
(-1)^FFd[V,\delta]e^{-t\Delta_E} + (-1)^FF[V,\delta]de^{-t\Delta_E}
\bigg{)} \\
=
\Tr_g\bigg{(}
(-1)^FFdV\delta e^{-t\Delta_E}- (-1)^FFd\delta Ve^{-t\Delta_E} + 
(-1)^FFV\delta d e^{-t\Delta_E} 
- (-1)^FF\delta Vd e^{-t\Delta_E}
\bigg{)}.
\end{multline*}

The operators $d$, $\delta$, $V$ all commute with $\Delta_E$, hence they commute with
$e^{-t\Delta_E}$. The $g$-trace is invariant under cyclic permutations of the operators occurring here, see lemma \ref{lem trace prop}. This implies we can write the above as
\begin{align*}
\Tr_g\bigg{(}
(-1)^FF\{d, [V,\delta]\}e^{-t\Delta_E}
\bigg{)} &=
\Tr_g\bigg{(}
V\Big{(}\delta(-1)^FFd + \delta d(-1)^FF - (-1)^FFd\delta - d(-1)^FF\delta 
\Big{)}e^{-t\Delta_E}
\bigg{)}.
\end{align*}

Using the fact that $(-1)^F$ anticommutes with $d$ and $\delta$, we can simplify  the above equation to 
\begin{equation*}
\Tr_g\bigg{(}
(-1)^FF\{d, [V,\delta]\}e^{-t\Delta_E}
\bigg{)} =
\Tr_g\bigg{(}
V(-1)^F\big{(}\delta [d,F] + [d,F]\delta 
\big{)}e^{-t\Delta_E}
\bigg{)}.
\end{equation*}
By the first equality in \eqref{commute_ops} the above becomes \eqref{eq anticomm}.
\end{proof}

For the rest of this section, let $F_1$ be such that \eqref{eq decay kappa gen} holds for all $\epsilon$, where $\kappa_t$ is the kernel of $e^{-t\Delta_E}$. Let $F_3$ be the function \eqref{eq def F3}, for $a$ such that \eqref{eq Gaussian F3} holds for all 
$\epsilon$. 
\begin{lemma}\label{form1} 
Suppose that $\lim_{t \to \infty} F_1(t)F_3(t)=0$.
Then
\begin{multline}\label{eq form1}
\frac{d}{d\epsilon}\bigg{(}-2\log T_g(\nabla^E)\bigg{)} =
\frac{d}{ds}\bigg{\vert}_{s=0}\frac{1}{\Gamma(s)}
\int_0^1t^s
\Tr_g\bigg{(}
V(-1)^F\Delta_E\big{(}e^{-t\Delta_E}\big{)}
\bigg{)}\, dt \\
+ \Tr_g \bigl(V
(-1)^F   e^{-\Delta_E} \bigr).
\end{multline}
\end{lemma} 
\begin{proof}
By \eqref{eqtor_compact},  
\begin{multline} \label{eq form1 3}
\frac{d}{d\epsilon}\bigg{(}-2\log T_g(\nabla^E)\bigg{)} =
\frac{d}{d\epsilon}
\frac{d}{ds}\bigg{\vert}_{s=0}\frac{1}{\Gamma(s)}
\int_0^1t^{s-1}
\Tr_g\bigl(
(-1)^FF(e^{-t\Delta_E} - 
P)\bigr) 
\, dt \\ 
+ \frac{d}{d\epsilon}
\int_1^{\infty}t^{-1}
\Tr_g \bigl(
(-1)^FF(e^{-t\Delta_E} -
P) \bigr)
\, dt.
\end{multline}

We first consider the first term on the right hand side. The space $\ker(\Delta_E)$ equals the $L^2$-cohomology of $M$ with values in $E$, and these 
cohomology groups depend on the quasi-isometry class of the metric, see proposition 3.1  in \cite{carron}. By lemma \ref{lem quasi-isom}, 
the metrics $\lambda_{\epsilon}$ are all quasi-isometric. Therefore, the projection $P$ does not depend on $\epsilon$. Together with \eqref{der_ops_2}, this implies that, 
\[
\frac{d}{d\epsilon}\bigl( 
(-1)^FF(e^{-t\Delta_E}-P) \bigr) = -t (-1)^FF\{d, [V,\delta]\}
e^{-t\Delta_E}.
\]
The operator on the right is $g$-trace class by lemma \ref{lem kappa t g tr cl} and example \ref{ex kappa t heat op}. 
So by lemma \ref{$g$-trace_der2},
\[
-t \Tr_g\left( (-1)^FF\{d, [V,\delta]\} \right) = \frac{d}{d\epsilon} \Tr_g \bigl( 
(-1)^FF(e^{-t\Delta_E}-P) \bigr).
\]
Combining this with lemmas  \ref{lem small t large l} and \ref{lem small t small l} and example \ref{ex kappa t heat op}, we find that if $s$ has large real part, 
\beq{eq der int eps 1}
-
\int_0^1t^{s}
\Tr_g\bigg{(}
(-1)^FF\{d, [V,\delta]\}
e^{-t\Delta_E}
\bigg{)} dt = 
\int_0^1t^{s-1}
\frac{d}{d\epsilon}\Tr_g\bigl( 
(-1)^FF(e^{-t\Delta_E}-P) \bigr)\, dt.
%
\eeq
Another application of lemma \ref{$g$-trace_der2} shows that $\frac{d}{d\epsilon}\Tr_g\bigl( 
(-1)^FF(e^{-t\Delta_E}-P) \bigr)$ depends differentiably, hence continuously,  on $\varepsilon$. So the derivative with respect to $\varepsilon$ may be taken outside the integral on the right hand side of \eqref{eq der int eps 1}, 
and we find that
\beq{eq der int eps 2}
\frac{d}{d\epsilon}
\frac{d}{ds}\bigg{\vert}_{s=0}\frac{1}{\Gamma(s)}
\int_0^1t^{s-1}
\Tr_g\bigl( 
(-1)^FF(e^{-t\Delta_E}-P) \bigr)\, dt  = 
-\frac{d}{ds}\bigg{\vert}_{s=0}\frac{1}{\Gamma(s)}
\int_0^1t^{s}
\Tr_g\bigg{(}
(-1)^FF\{d, [V,\delta]\}
e^{-t\Delta_E}
\bigg{)} dt.
\eeq
%
%
%
By lemma \ref{lem anticomm}, we conclude that the first term on the right hand side of \eqref{eq form1 3} equals the first term on the right hand side of \eqref{eq form1}.
 
Next, we consider the second term on the right hand side of \eqref{eq form1 3}. By a computation similar to 
the one leading to \eqref{eq der int eps 2}, 
 \[
 \frac{d}{d\epsilon}
\int_1^{\infty}t^{-1}
\Tr_g \bigl(
(-1)^FF(e^{-t\Delta_E} -
P) \bigr)
\, dt =- \int_1^{\infty}
\Tr_g \bigl(
(-1)^F F \{d, [V,\delta]\} (e^{-t\Delta_E}) 
\, dt.
 \]
 By lemma \ref{lem anticomm}, the right hand side equals
 \begin{align}
 \int_1^{\infty}
\Tr_g \bigl(V
(-1)^F  \Delta_E e^{-t\Delta_E} \bigr)\, dt &=  -\int_1^{\infty} \frac{d}{dt}
\Tr_g \bigl(V
(-1)^F   e^{-t\Delta_E} \bigr) \, dt \nonumber \\
&= -\lim_{t \to \infty} \Tr_g \bigl(V
(-1)^F   e^{-t\Delta_E} \bigr) + \Tr_g \bigl(V
(-1)^F   e^{-\Delta_E} \bigr). \label{eq form1 4}
\end{align}

Analogous to the proof of proposition \ref{prop NS cpt}, we  find  that
 \[
\Tr_g \bigl(V
(-1)^F   e^{-t\Delta_E} \bigr) = \cO(F_1(t) F_3(t))
 \]
as $ t\to \infty$. So the first term in \eqref{eq form1 4} is zero. We conclude that the second term on the right hand side of \eqref{eq form1 3} equals the second term on the right hand side of \eqref{eq form1}.
\end{proof}



\begin{lemma}\label{form2} 
Suppose that $\lim_{t \to \infty}F_1(t) F_3(t) = 0$.
 If $\Tr_g(V (-1)^F P)=0$, 
then
\beq{eq form2}
\frac{d}{d\epsilon}\bigg{(}-2\log T_g(\nabla^E)\bigg{)} =
\frac{d}{ds}\bigg{\vert}_{s=0}\frac{s}{\Gamma(s)}
\int_0^1t^{s-1}\Tr_g\big{(}V(-1)^Fe^{-t\Delta_E}\big{)}\, dt.
\eeq
\end{lemma}
\begin{proof}
%
%
%
In the first term on the right hand side of \eqref{eq form1}, we note that
\[
t^s \Tr_g\bigl( V (-1)^F \Delta_E e^{-t\Delta_E} \bigr) = 
-\frac{d}{dt} \left( 
t^s \Tr_g\bigl( V (-1)^F  e^{-t\Delta_E} \bigr)
\right) + 
st^{s-1} \Tr_g\bigl( V (-1)^F  e^{-t\Delta_E} \bigr).
\]
So the first term on the right hand side of \eqref{eq form1} equals
\beq{eq form2 2}
B
+
\frac{d}{ds}\bigg{\vert}_{s=0}\frac{s}{\Gamma(s)}
\int_0^1
t^{s-1} \Tr_g \bigl( V (-1)^F  e^{-t\Delta_E}\bigr)
\, dt,
\eeq
where
\begin{align}
B&:= -
\frac{d}{ds}\bigg{\vert}_{s=0}\frac{1}{\Gamma(s)}
\int_0^1
\frac{d}{dt} \left( 
t^s \Tr_g\bigl( V (-1)^F  e^{-t\Delta_E} \bigr)
\right)
\, dt \nonumber \\
&=-\dds \frac{1}{\Gamma(s)} \bigg{(}
\lim_{t\rightarrow 1^-} - \lim_{t\rightarrow 0^+}
\bigg{)}
t^s\Tr_g\big{(}V(-1)^Fe^{-t\Delta_E}\big{)}. \label{eq B lim}
\end{align}
The second term in \eqref{eq form2 2} is well defined via meromorphic continuation by lemmas \ref{lem small t large l} and \ref{lem small t small l}, and equals the right hand side of \eqref{eq form2}. The first term on the right hand side of \eqref{eq form1}  is well-defined for the same reason, hence so is $B$.


If $\Real(s)$ is large enough, then
\[
\lim_{t\rightarrow 0^+}
t^s\Tr_g\big{(}V(-1)^Fe^{-t\Delta_E}\big{)} = 0.
\]
As the term $B$ is defined via meromorphic continuation of a function of $s$, defined initially for $s$ with large real part, this implies that term in \eqref{eq B lim} involving the limit $t \to 0^+$ is zero. It follows that
\[
B = 
-\dds \frac{1}{\Gamma(s)}
\lim_{t\rightarrow 1^-} 
t^s\Tr_g\big{(}V(-1)^Fe^{-t\Delta_E}\big{)} = -\Tr_g\big{(}V(-1)^Fe^{-\Delta_E}\big{)}.
\]
This cancels the second term on the right hand side of \eqref{eq form1}, so \eqref{eq form2} follows from lemma \ref{form1}.
\end{proof}

\begin{remark} \label{rem term P}
In \eqref{eq form1 4}, one can add $\Tr_g(V (-1)^F P)$ to the first first term, and subtract this from the second term. This can be useful, particularly in the compact case (as in (2.32) in \cite{LR91}), because then $\lim_{t \to \infty}\Tr_g\bigl(V (-1)^F (e^{-t\Delta_E} - P)\bigr) = 0$. See also the last sentence of Subsection \ref{sec van der}. Then the resulting term involving $P$ in \eqref{eq form2} drops out, because 
$\left. \frac{d}{ds}\right|_{s=0}\frac{s}{\Gamma(s)}
\int_0^1t^{s-1} dt=0$.

However, because we will use estimates of the form \eqref{eq decay kappa gen} for the heat kernel, where usually $\lim_{t \to \infty} F_1(t)=0$ on noncompact manifolds, and because we state our metric independence results only in the case $P=0$, we have left \eqref{eq form1 4} as it is.
\end{remark}

\subsection{The constant term in an asymptotic expansion}



We will use a slight refinement of lemma 3.1 in \cite{HWW}.

Let $dz$ be the left Haar measure on $Z$ such that for all $f \in C_c(G)$,
\beq{eq def dz}
\int_G f(x)\, dx = \int_{G/Z} \int_Z f(hz^{-1})\, dz\, d(hZ).
\eeq
Let $\psi_G \in C^{\infty}(G)$ be a nonnegative function with compact support on all orbits of the right action by $Z$, and  such that for all $h \in G$,
\beq{eq cutoff chi G}
\int_Z \psi_G(hz^{-1})\, dz = 1.
\eeq
Fix a closed subset $X \subset G/Z$. Let $\pi\colon G \to G/Z$ be the quotient map. 
Let $\psi_g^X$ be the function on $M$ defined by
\beq{eq def psi g}
\psi_g^X(m) = \int_{\pi^{-1}(X)} \psi_G(h) \psi(hgm)\, dh.
\eeq


For a measurable subset $Y \subset G/Z$, and for a $g$-trace class operator $T$ with smooth kernel $\kappa$, we write
\beq{eq def TrgY}
\Tr_g^Y(T) := \int_{Y} \int_M \psi(hgh^{-1}m) \tr\bigl(hgh^{-1}\kappa(hg^{-1}h^{-1}m,m) \bigr)\, dm\, d(hZ).
\eeq
\begin{lemma}\label{lem Trg GZ cpt X}
For every
 $g$-trace class operator $T$ on $E$, with smooth kernel $\kappa$, 
 \[
 \Tr_g^X(T) = \int_M \psi_g^X(m) \tr(g\kappa(g^{-1}m,m))\, dm.
 \]
 If $X$ is compact, then $\psi_g^X$ has compact support.
\end{lemma}
\begin{proof}
Using $G$-equivariance of $\kappa$ and the trace property of $\tr$, one computes that for all $g \in G$, $z \in Z$ and $m \in M$,
\[
\tr \bigl(hgh^{-1} \kappa(hg^{-1}h^{-1}m,m) \bigr) = \tr \bigl(g \kappa(g^{-1}zh^{-1}m,zh^{-1}m) \bigr). 
\]
With this equality and \eqref{eq cutoff chi G}, we compute
\[
 \Tr_g^X(T) = \int_X \int_Z \int_M \psi_G(hz^{-1})  \psi(hgh^{-1}m)
  \tr \bigl(g \kappa(g^{-1}zh^{-1}m,zh^{-1}m) \bigr)\, dm\, dz\, d(hZ).
\]
Swapping the integration order (using absolute convergence) and substituting $m' = zh^{-1}m$, we find that the right hand side equals
\[
\int_M
\int_X \int_Z  \psi_G(hz^{-1})  \psi(hz^{-1}gm')\, dz\, d(hZ) 
  \tr \bigl(g \kappa(g^{-1}m',m') \bigr)\, dm'.
\]
By \eqref{eq def dz}, 
\[
\int_X \int_Z  \psi_G(hz^{-1})  \psi(hz^{-1}gm')\, dz\, d(hZ) = \psi_g^X(m')
\]
for all $m' \in M$.

If $m \in \supp(\psi_g^X)$, then there is an $h \in G$ such that
\beq{eq h X M}
\begin{split}
hZ &\in X;\\
h &\in \supp(\psi_G);\\
hgm & \in \supp(\psi). 
\end{split}
\eeq
If $X$ is compact, then the fact that $\psi_G$ has compact support on $Z$-orbits implies that the set $Y$ of $h\in G$ satisfying the first two conditions is compact.
 Then compactness of $\supp(\psi)$ and properness of the action implies that the set of $m \in M$ satisfying the third condition in \eqref{eq h X M} for some $h \in Y$ is compact.
%
\end{proof}

\begin{lemma}\label{lem t0 term}
If $X \subset G/Z$ is compact, then $\Tr_g^X\big{(}V(-1)^Fe^{-t\Delta_E}\big{)}$ has an asymptotic expansion in $t$. 
If  the dimension of $M$ is odd, then
there is no $t^0$ term
in this asymptotic expansion.
\end{lemma}
\begin{proof}
The standard asymptotic expansion for the kernel $\kappa_t$ of $e^{-t\Delta_E}$ has the form
\beq{eq as exp kappat heat}
\kappa_t \sim t^{-\dim(M)/2} \sum_{j=0}^{\infty} \kappa_{(j)} t^j.
\eeq
This expansion
holds with respect to $C^k$-norms on compact sets. 
Lemma \ref{lem Trg GZ cpt X} allows one to express $\Tr_g^X\big{(}V(-1)^Fe^{-t\Delta_E}\big{)}$ as an integral over a compact subset of $M \times M$. So if $X$ is compact, then the asymptotic expansion \eqref{eq as exp kappat heat} gives rise to an asymptotic expansion of $\Tr_g^X\big{(}V(-1)^Fe^{-t\Delta_E}\big{)}$ involving the same powers of $t$. If $\dim(M)$ is odd, then none of these powers are integers. In particular, there is no $t^0$ term.
\end{proof}

\subsection{Vanishing of the derivative}\label{sec van der}


We now combine the results from this section to obtain a vanishing result for the derivative of equivariant analytic torsion with respect to $\epsilon$, see theorem \ref{metric_indep_cconj}. This implies our main result on metric independence, see corollary \ref{cor metric indep}.

\begin{lemma}\label{lem integrals f}
Let $f$ be a continuous function on $(0, 1]$ such that $t^{-1} \log(t) f(t)$ is bounded on $(0,1]$.
Then
\[
\dds \frac{s}{\Gamma(s)} \int_0^1 t^{s-1}f(t)\, dt = 0.
\]
\end{lemma}
\begin{proof}
For $s>0$, write $F(s) := \int_0^1 t^{s-1}f(t)\, dt$. As
$t^{-1}  f(t)$ is bounded, the integral
$
\int_0^1 t^{-1}f(t)\, dt 
$
converges,
so 
$F$ extends continuously to $s=0$. 

Furthermore, because $t^{-1} \log(t) f(t)$ is bounded, its integral from $t=0$ to $1$ converges,  
and equals
\[
\int_0^1 \dds t^{s-1}f(t)\, dt.
\]
Hence,
 $F$ is differentiable at $s=0$. The claim now follows from the product rule and the equalities
 \[
 \begin{split}
 \lim_{s\to 0} \frac{s}{\Gamma(s)} &= 0;\\
  \lim_{s\to 0} \frac{d}{ds}  \frac{s}{\Gamma(s)} &= 0.
 \end{split}
 \]
 \end{proof}
 
 \begin{lemma}\label{lem vanish Xcompl}
 There is a compact subset $X \subset G/Z$ such that
 \[
  \frac{d}{ds}\bigg{\vert}_{s=0} \frac{s}{\Gamma(s)}
  \int_0^1t^{s-1}\Tr_g^{\Xcompl}\big{(}V(-1)^Fe^{-t\Delta_E}\big{)}dt 
= 0.
 \]
 \end{lemma}
 \begin{proof}
 Let $X$ be as in lemma \ref{lem small t large l}. 
For $t>0$, set $f(t) := \Tr_g^{\Xcompl}\big{(}V(-1)^Fe^{-t\Delta_E}\big{)}$. 
%
%
%
%
Note that $\log(t) = \cO(t^{-1})$ as $t\downarrow 0$. So example \ref{ex kappa t heat op} and lemma \ref{lem small t large l}, with $a=2$,  imply that $t^{-1}\log(t) f(t)$ is bounded in $t \in (0,1]$. 
Hence the claim follows from lemma \ref{lem integrals f}.
 \end{proof}

\begin{theorem}\label{metric_indep_cconj}
Suppose that the dimension of $M$ is odd, and that $\ker(\Delta_E) = \{0\}$.
Suppose that there is a function $F_1$  such that for all $\epsilon$, the estimate 
\eqref{eq decay kappa gen} holds for the kernel $\kappa_t$ of $ e^{-t\Delta_E}$.
Suppose that there is a \v{S}varc--Milnor function for the action of $G$ with respect to $g$, and let 
 $F_3$ be as in \eqref{eq def F3}, for $a$ such that \eqref{eq Gaussian F3} holds for all $\epsilon$. Suppose that $F_1(t)F_3(t) = \cO(t^{-\alpha})$ as $t \to \infty$, for some $\alpha>0$.
Then $ T_g(\nabla^E)$ converges for all $\epsilon$, and 
\begin{equation*}
\frac{d}{d\epsilon}\bigg{(}-2\log T_g(\nabla^E)\bigg{)} = 
0.
\end{equation*}
\end{theorem}
\begin{proof}
The conditions of definition \ref{def g torsion} hold by propositions \ref{prop conv small t} and \ref{prop NS cpt}. 
Let $X \subset G/Z$ be as in lemma \ref{lem vanish Xcompl}.  
From lemmas  \ref{form2} and \ref{lem vanish Xcompl}, we have
\beq{eq metr indep 1}
\frac{d}{d\epsilon}\bigg{(}-2\log T_g(\nabla^E)\bigg{)} =
\frac{d}{ds}\bigg{\vert}_{s=0}\frac{s}{\Gamma(s)}
\int_0^1t^{s-1}\Tr_g^X\big{(}V(-1)^Fe^{-t\Delta_E}\big{)}\, dt.
\eeq
For all $a \in \R$, 
\[
\frac{d}{ds}\bigg{\vert}_{s=0}\frac{s}{\Gamma(s)}
\int_0^1t^{s-1} t^a\, dt = \left\{
\begin{array}{ll}
1 & \text{if $a=0$};\\
0 & \text{if $a \not=0$}.
\end{array}
 \right.
\]
It follows that the right hand side of \eqref{eq metr indep 1} equals 
 the $t^0$ term in  the asymptotic expansion of 
$\Tr_g^X\big{(}V(-1)^Fe^{-t\Delta_E}\big{)}$. 
 The theorem now follows by lemma \ref{lem t0 term}.
\end{proof}
The condition in theorem \ref{metric_indep_cconj} that $\ker(\Delta_E) = \{0\}$ does not depend on a choice of $G$-invariant Riemannian metric, as pointed out in the proof of lemma \ref{form1}.

Theorem \ref{metric_indep_cconj} has the following immediate consequence. 
\begin{corollary}\label{cor metric indep}
Suppose that the dimension of $M$ is odd, and that $\ker(\Delta_E) = \{0\}$.
Let $\lambda$ and $\lambda'$ be two $G$-invariant Riemannian metrics on $M$. Suppose that there is a smooth path of $G$-invariant Riemannian metrics that connects $\lambda$ and $\lambda'$, and such that there are functions
 $F_1$ and $F_3$ as in theorem \ref{metric_indep_cconj}. Suppose, in particular, that $F_1(t)F_3(t) = \cO(t^{-\alpha})$ as $t \to \infty$,  for some $\alpha>0$. 
Then $T_g(\nabla^E)$ is well-defined for both $\lambda$ and $\lambda'$, and the two values are equal. 
\end{corollary}

It was assumed in theorem \ref{metric_indep_cconj} and corollary \ref{cor metric indep} that $M$ is odd-dimensional. We will see  in proposition \ref{prop vanish even} that equivariant analytic torsion of an even-dimensional manifold is always $1$ if it converges, and hence trivially independent of the metric.

The condition in theorem \ref{metric_indep_cconj} and corollary \ref{cor metric indep} that $F_1(t)F_3(t) = \cO(t^{-\alpha})$ for some $\alpha>0$ means that the long time heat kernel decay should be fast enough relative to the volume growth of the sets \eqref{eq conj cpt}. See lemma \ref{lem F2 F3} for the growth rate of $F_3$ in some cases, and the results mentioned at the end of subsection \ref{sec large t conv} for  results on the decay rate of $F_1$. 

We have formulated this condition in terms of $F_1$ and $F_3$ to separate the effects of the volume growth of $G/Z$ and the long time decay of the heat kernel. This may make the condition
 more explicit and possibly easier to check, but a more general condition that is still sufficient for theorem \ref{metric_indep_cconj} and corollary \ref{cor metric indep} is that for all $\epsilon>0$, 
\beq{eq general large t conditions}
\begin{split}
\lim_{t \to \infty} \Tr_g(V (-1)^F e^{-t\Delta_E}) &= 0; \quad \text{and}\\
\Tr_g\bigl((-1)^F F (e^{-t\Delta_E}-P) \bigr)&= \cO(t^{-\alpha(\epsilon)})
\end{split}
\eeq
for some $\alpha(\epsilon)>0$. Indeed, the condition on $F_1 F_3$ is used in the proofs of proposition \ref{prop NS cpt} and lemmas \ref{form1} and \ref{form2}, and in those places the conditions \eqref{eq general large t conditions} are sufficient.

In the first condition in \eqref{eq general large t conditions}, recall that we assume that $\ker(\Delta_E) = 0$; otherwise a version with $e^{-t\Delta_E}$ replaced by $e^{-t\Delta_E} - P$ is also sufficient by remark \ref{rem term P}.

\subsection{Special cases}


\begin{corollary}\label{cor metric indep GZ cpt}
Suppose that $G/Z$ is compact, and that $\alpha_e^p>0$ for all $G$-invariant Riemannian metrics on $M$. Suppose that there is a \v{S}varc--Milnor function for the action with respect to $g$. Then equivariant analytic torsion converges and is independent of the Riemannian metric.
\end{corollary}
\begin{proof}
The volume growth function $F_2$ in \eqref{eq def F2} is bounded, so by lemma \ref{lem F2 F3} with $b=0$, we now have $F_3(t) = \cO(t^{1/2})$.  
 To apply corollary \ref{cor metric indep} directly, we would need $\kappa_t$ to satisfy \eqref{eq decay kappa gen}  for all $\epsilon$, with $F_1(t) = \cO(t^{-\beta})$ for a $\beta>1/2$, as $t \to \infty$. 
 
 However, the more general conditions \eqref{eq general large t conditions} hold 
  if $\alpha_e^p>0$ for all $p$ and $\epsilon$ via proposition \ref{prop GZ cpt} (this was used in the proof of proposition 3 in \cite{Lott99}). Therefore a version of corollary \ref{cor metric indep} with these weaker, but still sufficient, conditions applies and yields convergence and metric independence.
\end{proof}

In general, it is not known if the condition $\alpha_e^p>0$ in corollary \ref{cor metric indep GZ cpt} depends on the metric. It does not in the setting of lemma \ref{lem alpha e invar}.

\begin{corollary}\label{cor metric indep invertible}
Suppose that $M$ is simply connected. Then invertibility of $\Delta_E$ does not depend on the Riemannian metric. If $\Delta_E$ is invertible, then equivariant analytic torsion converges and is metric independent if the sets $X_r$ have polynomial, or slow enough exponential, volume growth. 
\end{corollary}
\begin{proof}
Suppose that $\Delta_E$ is invertible. This condition
 is independent of the metric if  $M$ is simply connected by corollary 29 in \cite{Lott92c}. Then, on a path of $G$-invariant Riemannian metrics, the heat kernel $\kappa_t$ satisfies \eqref{eq decay kappa gen} with $F_1(t) = \cO(e^{-\Lambda t})$, where $\Lambda$ is the infimum of the bottoms of the spectra  of the invertible Laplacians along the path. If $\Lambda$ is positive,  then lemma \ref{lem F2 F3}(b) implies that the condition $F_1(t)F_3(t) = \cO(t^{-\alpha})$ for some $\alpha>0$ holds if $F_2(r) = \cO(e^{br})$ for $b < 2\sqrt{a\Lambda}$, with $a$ as above \eqref{eq def F2}, or if the sets $X_r$ have polynomial volume growth. 

We have assumed for simplicity that the functions $F_1$ and $F_3$ do not depend on $\epsilon$. More generally, theorem \ref{metric_indep_cconj} and corollary \ref{cor metric indep} still hold if these functions do depend on $\epsilon$, if for all $\epsilon>0$ there is an $\alpha_{\epsilon}>0$ such that $F_1(t)F_3(t) = \cO(t^{-\alpha_{\epsilon}})$. 
This implies that equivariant analytic torsion is independent of the metric if the sets $X_r$ have polynomial volume growth. 
\end{proof}
It is unknown to us if
simple connectivity is necessary in corollary \ref{cor metric indep invertible}.

For another consequence of theorem \ref{metric_indep_cconj}, we use results on long-time heat kernel decay from \cite{Coulhon20}. We introduce some notation and conditions to state this. We consider the case where $E$ is the trivial line bundle, with the trivial action by $G$ on its fibres.

For a $G$-invariant Riemannian metric $\lambda$ on $M$, and for $m \in M$ and $r>0$, we write $V_{\lambda}(m, r)$ for the volume of the ball around $m$ of radius $r$, with respect to $\lambda$.

Given a $G$-invariant metric on $M$, there is a Bochner-type equality
\[
\Delta^p = \nabla^*\nabla + R^p,
\]
with $\nabla$ the connection on $p$-forms induced by the Levi--Civita connection, and $R^p \in \End(\Bigwedge^p T^*M)$ defined in terms of curvature. In particular,  $R^1$ is Ricci curvature. We write $\Riem^+(M)^G$ for the space of $G$-invariant Riemannian metrics on $M$ such that $R^p$ is nonnegative for all $p$.

Recall the definition \eqref{eq def F2} of the volume-growth function $F_2$ of the sets $X_r$.
\begin{corollary} \label{cor metric indep slow growth}
Suppose that $\Riem^+(M)^G$ is not empty. 
Let $\lambda \in \Riem^+(M)^G$, and suppose that there are $m_0 \in M$, $C>0$ and $\kappa>4$ such that for all $r>0$,
\beq{eq vol growth kappa}
V_{\lambda}(m_0, r) \geq Cr^{\kappa}.
\eeq
Suppose that $F_2(r) = \cO(r^b)$ for some $b<\kappa -1$. Suppose that $\ker(\Delta) =\{0\}$, and that there is a \v{S}varc--Milnor function for the action, with respect to $g$. Then
\begin{itemize}
\item[(a)] the conditions of definition \ref{def g torsion} hold for $\lambda$;
\item[(b)] the conditions of definition \ref{def g torsion} in fact hold for all $\lambda' \in  \Riem^+(M)^G$;
\item[(c)] equivariant analytic torsion is constant on smooth path components of $\Riem^+(M)^G$.
\end{itemize}
\end{corollary}
\begin{proof}
Let $p_t$ be the heat kernel for the Hodge-Laplacian $\Delta^0$ on scalar functions.  
By proposition 5.2 in \cite{Grigoryan94} and the comment above it, nonnegativity of the Ricci curvature $R_1$ implies that there is an  $a_1>0$ such that for all $m \in M$ and  $t >0$, 
\beq{eq on-diag scalar}
p_t(m,m) \leq \frac{a_1}{V_{\lambda}(m, \sqrt{t})}. 
\eeq
The same result in \cite{Grigoryan94} also states that the function $V_{\lambda}$ satisfies a volume-doubling condition. In fact, it is shown that an estimate of the form \eqref{eq on-diag scalar} and a volume doubling condition together are equivalent to an isoperimetric inequality that follows from positivity of $R_1$.

A special case of corollary 1.12 in \cite{Coulhon20} states that the volume doubling property of $V_{\lambda}$, the estimate \eqref{eq on-diag scalar}, nonnegativity of $R^p$, \eqref{eq vol growth kappa} and vanishing of $\ker(\Delta)$ imply that \eqref{eq decay kappa gen} holds for all $t>0$, with
\beq{eq F1 FK}
F_1(t) = \cO\left(\frac{1}{V_{\lambda}(m, \sqrt{t})} \right) =\cO(t^{-\kappa/2}).
\eeq
We used \eqref{eq vol growth kappa} for the second estimate. 

By lemma \ref{lem quasi-isom} and proposition \ref{prop conv small t}, conditions (1) and (3) hold for any $G$-invariant Riemannian metric on $M$. Part (a) of lemma \ref{lem F2 F3} and \eqref{eq F1 FK} imply that
\beq{eq F1 F3 FK}
F_1(t)F_3(t) = \cO(t^{(b+1-\kappa)/2}).
\eeq
So the conditions of proposition \ref{prop NS cpt} hold, and hence condition (2) in definition \ref{def g torsion} is satisfied. This is part (a).

Part (b) follows from part (a), because the condition \eqref{eq vol growth kappa}, including the value of the constant $\kappa$,  is independent of the $G$-invariant Riemannian metric.

By part (b) and \eqref{eq F1 F3 FK},  corollary \ref{cor metric indep} applies to any two metrics in the same smooth path component of $ \Riem^+(M)^G$, so (c) follows.
\end{proof}

\section{Further properties}

After the results on convergence and metric independence from sections \ref{sec conv} and \ref{sec metric indep}, we prove some more properties of equivariant analytic torsion. Two of these involve the notion of a $g$-Euler chararacteristic. 

\subsection{The $g$-Euler characteristic}

Let $M^g$ be the fixed-point set of $g$ in $M$. Then $M^g$ is invariant under $Z$, and $M^g/Z$ is compact (see lemma 2.4 in \cite{HW2}). The connected components of $M^g$ are submanifolds of $M$, possibly of different dimensions. We will treat $M^g$ as a submanifold of $M$, but in general all constructions should be applied to its connected components separately.

Let $\psi_g \in C^{\infty}_c(M^g)$ be such that  for all $m \in M^g$,
\[
\int_Z \psi_g(zm)\, dz = 1.
\]
One way to choose this function is \eqref{eq def psi g}, with $X=G/Z$.
Let $R^{M^g}$ be the curvature of the Levi--Civita connection on $M^g$, and $\Pf(-R^{M^g})$ the Pfaffian of $-R^{M^g}$, see definition 1.35 in \cite{BGV}. If  the closure $K$ of the cyclic group  generated by $g$ is compact,  
 let 
$
\ch_g(E|_{\supp(\psi_g)})
$
be the equivariant Chern character of $E|_{\supp(\psi_g)}$ with respect to the action by $K$, evaluated at $g$, see section 3 of \cite{ASIII}.
\begin{definition}\label{def chi g}
If $g$ lies in a compact subgroup of $G$, then 
the \emph{$g$-Euler characteristic} of $E$ is
\beq{eq def chi g}
\chi_g(E) := \frac{1}{(2\pi)^{\dim(M^g)/2}} \int_{M^g} \psi_g \Pf(-R^{M^g})\ch_g(E|_{\supp(\psi_g)}).
\eeq
If $g$ does not lie in a compact subgroup of $G$, then we set $\chi_g(E) := 0$.
\end{definition}

\begin{example} \label{ex chi g cpt}
If $M^g$, or equivalently $Z$, is compact, then we may take $\psi_g \equiv 1$. Then, if $E$ is the trivial line bundle,  $\chi_g(E)$ is the Euler characteristic of $M^g$ by the Gauss--Bonnet--Chern theorem, see theorem 4.7 in \cite{BGV} or theorem 3.7 in \cite{Zhang01}. 

If $Z$ is discrete and acts freely on $M^g$, and $E$ is the trivial line bundle, then 
\[
\chi_g(E) = \frac{1}{(2\pi)^{\dim(M^g/Z)/2}} \int_{M^g/Z}  \Pf(-R^{M^g/Z}),
\]
which equals 
the Euler characteristic of $M^g/Z$.
\end{example}

\begin{example}\label{ex chi g odd zero}
%
If $M$ is odd-dimensional, then so is $M^g$, because the action by $G$ preserves the orientation. The integrand in \eqref{eq def chi g} only has components in even degrees, so $\chi_g(E) = 0$. 
\end{example}

%
%
We will use an equivariant version of the Gauss--Bonnet--Chern theorem for proper actions.
\begin{theorem}\label{thm g GBC}
Suppose that there is a \v{S}varc--Milnor function for the action by $G$ on $M$ with respect to $g$.
Then for all $t>0$, 
\beq{eq g CGB}
 \Tr_g( (-1)^F e^{-t \Delta_E}) = \chi_g(E).
\eeq
\end{theorem}
\begin{proof}
Suppose first that 
$g$ lies in a compact subgroup of $G$. Let  $
 \AS_g(\nabla^E + (\nabla^E)^*) 
 $
 be the Atiyah--Segal--Singer integrand for the Dirac-type operator $\nabla^E + (\nabla^E)^*$, see theorem 6.16 in \cite{BGV}. 
 The equality 
\beq{eq index thms}
\Tr_g( (-1)^F e^{-t \Delta_E}) = \int_{M^g} \psi_g \AS_g(\nabla^E + (\nabla^E)^*),
\eeq
was proved in various special cases in different places: 
\begin{itemize}
\item If $g = e$,  then 
 \eqref{eq index thms} is theorem 6.10 in \cite{Wang14};
\item if $G$ is discrete and finitely generated, then 
 \eqref{eq index thms} 
is theorem 3.23 in \cite{Wangwang};
\item if $G$ is a connected semisimple Lie group, and $g$ is semisimple, then
  \eqref{eq index thms} is proposition 4.11  in \cite{HW2}; see also theorem 5.32 in \cite{PPST21} and Theorem 3.19 in \cite{HP25}.
\end{itemize}
On careful inspection of the proof of  theorem 4.31 in \cite{PPST21}, which is \eqref{eq index thms} for semisimple elements of connected, semisimple Lie groups, one finds that what is really needed to make that proof work is just the existence of a \v{S}varc--Milnor function. More precisely, conditions (3) and (1) in definition \ref{def SM} take the places of  lemmas 5.13 and 5.21 in \cite{PPST21}, respectively. Lemma 2.4 in \cite{HW2} takes the place of lemma 5.22 in \cite{PPST21}. The rest of the argument stays the same.

By classical arguments, 
\[
 \AS_g(\nabla^E + (\nabla^E)^*)  = \frac{1}{(2\pi)^{\dim(M^g)/2}}  \Pf(-R^{M^g})\ch_g(E|_{\supp(\psi_g)}).
\]
See for example page 262 of \cite{Lawson89} and proposition 4.6 in \cite{BGV}.

In the case where $g$ does not lie in a compact subgroup of $G$, the arguments from \cite{HW2, PPST21, Wangwang} generalise in a more direct way to our current setting, to show that $\Tr_g( (-1)^F e^{-t \Delta_E}) = 0$. The expression \eqref{eq def Trh} for this number then localises near  $M^g$, which is empty because the action is proper. Use,  for example,  point (2) in the proof of theorem 4.31 in \cite{PPST21}, with the \v{S}varc--Milnor function used as above, and note that the empty set is a neighbourhood of $M^g$.
\end{proof}

\begin{remark}
In the settings considered in \cite{HW2, Wangwang, Wang14}, it is also shown that the number \eqref{eq g CGB} can be obtained from a $K$-theoretic index via an orbital integral trace on a  suitable convolution algebra of functions on $G$.
\end{remark}

\subsection{Triviality on even-dimensional manifolds}

\begin{proposition}\label{prop vanish even}
Suppose that the dimension of $M$ is even,  that there is a \v{S}varc--Milnor function for the action, and that $P$ is $g$-trace class.
\begin{itemize}
\item[(a)] If $\Tr_g((-1)^FP) = \chi_g(E)$ (in particular, if $P=0$ and $\chi_g(E)=0$), then the conditions of definition \ref{def g torsion} hold, and $T_g(\nabla^E) = 1$. 
 \item[(b)] If $\Tr_g( (-1)^F P) \not= \chi_g(E)$ (in particular, if $P=0$ and $\chi_g(E)\not=0$), then the expression \eqref{eq def deloc torsion} for $T_g(\nabla^E)$ diverges.
 \end{itemize}
\end{proposition}
\begin{proof}
As in \eqref{eq Hodge star E}, let $\star$ denote the Hodge star operator  on $E$-valued forms on $M$. The Hodge star 
operator commutes with the operators $e^{-t\Delta_E}$ and $P$, and therefore commutes 
with $e^{-t\Delta_E} - P$. Let  $\dim(M) = n$, then
\begin{equation}\label{F_comm_star}
F\star + \star F = n\star
\end{equation}
and
\begin{equation}\label{F_comm_star2}
(-1)^F\star = (-1)^n\star (-1)^F.
\end{equation}
Forgetting the functions $t^{s-1}$ and $t$ in equation \eqref{eqtor_compact}, consider
the integrand
\beq{eq Trg star}
\begin{split}
\Tr_g\bigg{(}
(-1)^FF\big{(}e^{-t\Delta_E} - P \big{)}
\bigg{)}
&=
\Tr_g\bigg{(}\star^{-1}
(-1)^FF\big{(}e^{-t\Delta_E} - P \big{)}\star
\bigg{)}\\
&=
\Tr_g\bigg{(}\star^{-1}
(-1)^FF\star\big{(}e^{-t\Delta_E} - P \big{)}
\bigg{)}.
\end{split}
\eeq
For the first equality, we used the trace property of $\Tr_g$. This is lemma \ref{lem trace prop}, but because $\star$ is a vector bundle endomorphism, this now follows directly from the trace property of the fibre-wise trace.
 
We have
\begin{align*}
\star^{-1}(-1)^FF\star &= \star^{-1}(-1)^F\star (n-F) \quad \text{by  \eqref{F_comm_star}} \\
&= 
(-1)^n(-1)^F(n-F) \quad \text{by  \eqref{F_comm_star2}}.
\end{align*} 

Together with \eqref{eq Trg star}, this implies
\[
\Tr_g\bigg{(}
(-1)^FF\big{(}e^{-t\Delta_E} - P \big{)}
\bigg{)} =
(-1)^n\Tr_g\bigg{(}
(-1)^F(n-F)\big{(}e^{-t\Delta_E} - P \big{)}
\bigg{)}.
\]
If $n$ is even, we obtain
\[
\Tr_g\bigg{(}
(-1)^FF\big{(}e^{-t\Delta_E} - P \big{)}
\bigg{)} = \frac{n}{2}
\Tr_g\bigg{(}
(-1)^F\big{(}e^{-t\Delta_E} - P \big{)}
\bigg{)}.
\]
We conclude that $\cT_g(t)$ is constant, because the right hand side is. (This is part of theorem \ref{thm g GBC}.) If $\chi_g(E) = \Tr_{g}((-1)^F P)$, then this constant is zero by theorem \ref{thm g GBC}, and part (a) follows. Otherwise, this constant is nonzero, and the second integral in \eqref{eq def deloc torsion} diverges as in part (b).
 \end{proof}

\begin{remark}\label{rem chi g P}
The condition $\Tr_g((-1)^FP) = \chi_g(E)$ in part (a) of proposition \ref{prop vanish even} always holds in cases where equivariant analytic torsion was studied before. Indeed,
\begin{itemize}
\item[(1)] if $M$ is compact, then $\Tr_g((-1)^FP) = \chi_g(E)$ is an equivariant version of the Gauss--Bonnet--Chern theorem, see example \ref{ex chi g cpt};
\item[(2)] if $g=e$, then $\Tr_e((-1)^FP) = \Tr_e((-1)^F e^{-t \Delta_E}) = \chi_e(E)$, see corollary 3.12 and theorem 6.10 in \cite{Wang14};
\item[(3)] if $g\not=e$ and the action is free, or $g$ is not contained in a compact subgroup,  then $\chi_g(E) = 0$ because $M^g = \emptyset$. Then the condition $\Tr_g((-1)^FP) = \chi_g(E)$ holds in the most relevant case where $\ker(\Delta_E)=\{0\}$.
\end{itemize}

More generally than the second point, suppose that $G/Z$ is compact. Then by proposition \ref{prop GZ cpt},
\[
|\Tr_g\bigl( (-1)^F (e^{-t\Delta_E} - P) \bigr)| \leq \vol(G/Z) |\Tr_e\bigl( (-1)^F (e^{-t\Delta_E} - P)| = 0,
\]
where we used point (2). Then case (a) of proposition \ref{prop vanish even} applies.

A crucial difference between the cases $g=e$ and $g\not=e$ is that if $g\not=e$, then $\Tr_g( (-1)^F e^{-t\Delta_E})$ may be nonzero even if the kernel of $\Delta_E$ is zero. This difference has consequences  to $K$-theory of group $C^*$-algebras and Dirac induction. For a semisimple Lie group with discrete series, the von Neumann trace only detects $K$-theory classes defined by discrete series representations, and these lie in the kernel of the relevant Dirac operator \cite{Atiyah77}. 
Corollary 4.1 in \cite{HW4} states that $g$-traces for nontrivial group elements determine all information on $K$-theory elements, and in particular can be nonzero even when the kernel of this Dirac operator is trivial. The trace $\tau_g$ in \cite{HW4} acting on $K$-theory is linked to the number $\Tr_g( (-1)^F e^{-t\Delta_E})$ via proposition 3.6 in \cite{HW2}. 
\end{remark}

\subsection{A product formula}\label{sec prod form}


For $j=1,2$, let  $G_j$ be a unimodular, locally compact group 
 acting properly, isometrically and cocompactly on an oriented Riemannian manifold $M_j$, preserving the orientation. Let $g_j \in G_j$ be an element with centraliser $Z_j<G_j$, such that $G_j/Z_j$ has a $G_j$-invariant measure. 
 
 Let $E_j \to M_j$ be a flat, Hermitian $G_j$-vector bundle, and $\nabla^{E_j}$ a $G_j$-invariant, flat connection on $E_j$ preserving the Hermitian metric.
Consider the flat connection 
\[
\nabla^{E_1 \boxtimes E_2} := \nabla^{E_1} \otimes 1_{E_2} + 1_{E_1}\otimes \nabla^{E_2}
\]
on the external tensor product ${E_1 \boxtimes E_2} \to M_1 \times M_2$. For $p_j = 0,\ldots, \dim(M_j)$, let $P^{E_j}$ be projection onto the $L^2$-kernel of $\Delta_{E_j}^{p_j}$. 
We denote the analogous projection for $\Delta_{E_1 \boxtimes E_2}^p$ by $P_{p}^{E_1 \boxtimes E_2}$. 
\begin{proposition}\label{prop prod form}
Suppose that  for $j=1,2$, condition (2) of definition \ref{def g torsion} holds, that 
there is a \v{S}varc--Milnor function for the action, and that $P^{E_j}$ is $g_j$-trace class. Then $T_{(g_1, g_2)}(\nabla^{E_1 \boxtimes E_2})$ is well-defined.
If either 
\begin{itemize}
\item[(a)] for $j=1,2$, the space $G_j/Z_j$ is compact and the Novikov--Shubin numbers $(\alpha_e^{p_j})_{E_j}$ for $\Delta_{E_j}$ are positive, or
\item[(b)] $\Tr_{g_1}(P^{E_1}) = \Tr_{g_2}(P^{E_2}) = 0$,
\end{itemize}
 then 
\beq{eq prod form}
T_{(g_1, g_2)}(\nabla^{E_1 \boxtimes E_2}) = T_{g_1}(\nabla^{E_1})^{\chi_{g_2}(E_2)}T_{g_2}(\nabla^{E_2})^{\chi_{g_1}(E_1)}.
\eeq
\end{proposition}
\begin{proof}
The argument is similar to the proof of proposition 11 in \cite{Lott92c}, where theorem \ref{thm g GBC} takes the place of Atiyah's $L^2$-index theorem. See also theorem 4.3 in \cite{Mathai92}.

We have
\[
\Delta_{E_1\boxtimes E_2}^p = \bigoplus_{p_1+p_2= p} \Delta_{E_1}^{p_1} \otimes 1_{E_2} + 1_{E_1} \otimes \Delta_{E_2}^{p_2}. 
\]
So for all $t>0$, 
\beq{eq heat prod}
e^{-t\Delta_{E_1\boxtimes E_2}^p} = \bigoplus_{p_1+p_2= p} e^{-t\Delta_{E_1}^{p_1}} \otimes  e^{-t\Delta_{E_2}^{p_2}}. 
\eeq
We denote the number operator on $\Bigwedge^* T^*(M_1 \times M_2)$ by $F$, and the number operator on $\Bigwedge^* T^*M_j$ by $F_j$. Then by \eqref{eq heat prod}, 
\beq{eq curly T product}
\begin{split}
\Tr_{(g_1, g_2)} \bigl( (-1)^F F e^{-t\Delta_{E_1 \boxtimes E_2}} \bigr) &= \sum_{p=0}^{\dim(M_1\times M_2)} (-1)^p p \Tr_{(g_1, g_2)}(e^{-t\Delta_{E_1\boxtimes E_2}^p} ) \\
&= \sum_{p_1=0}^{\dim(M_1)}
\sum_{p_2=0}^{\dim(M_2)}
 (-1)^{p_1+p_2} (p_1+p_2) \Tr_{g_1}(e^{-t\Delta_{E_1}^{p_1}} )\Tr_{g_2}(e^{-t\Delta_{E_2}^{p_2}} ) \\
& =\Tr_{g_1} \bigl( (-1)^{F_1} F_1 e^{-t\Delta_{E_1}} \bigr) \Tr_{g_2} \bigl( (-1)^{F_2}  e^{-t\Delta_{E_2}} \bigr)  \\
& \qquad + \Tr_{g_2} \bigl( (-1)^{F_2} F_2 e^{-t\Delta_{E_2}} \bigr) \Tr_{g_1} \bigl( (-1)^{F_1} e^{-t\Delta_{E_1}} \bigr)\\
& =\Tr_{g_1} \bigl( (-1)^{F_1} F_1 e^{-t\Delta_{E_1}} \bigr) \chi_{g_2}(E_2) + \Tr_{g_2} \bigl( (-1)^{F_2} F_2 e^{-t\Delta_{E_2}} \bigr) \chi_{g_1}(E_1),
 \end{split}
\eeq
 by theorem \ref{thm g GBC}.
 
 The conditions of definition \ref{def g torsion} hold for $\nabla^{E_j}$ by proposition \ref{prop conv small t}. By \eqref{eq curly T product}, the second condition of definition \ref{def g torsion} holds for $\nabla^{E_1 \boxtimes E_2}$. By lemma \ref{lem SM product} and proposition \ref{prop conv small t}, conditions (1) and (3) of definition \ref{def g torsion} hold for $\nabla^{E_1 \boxtimes E_2}$ as well, so $ T_{(g_1, g_2)}(\nabla^{E_1 \boxtimes E_2})$ is well-defined. 
 
 In case (a), proposition \ref{prop GZ cpt} implies that
 \[
\left|\Tr_{g_j} \bigl( (-1)^{F_j} F_j (e^{-t\Delta_{E_j}} - P^{E_j}) \bigr) \right| \leq \vol(G_j/Z_j) \dim(M) \left|\Tr_{e} (e^{-t\Delta_{E_j}} - P^{E_j})  \right|.
 \]
 The right hand side goes to zero as $t \to \infty$, because $(\alpha_e^{p_j})_{E_j}>0$ for all $p_j$. Using the analogous estimate on $M_1 \times M_2$, we find that subtracting the limits $t \to \infty$ from both sides of \eqref{eq curly T product} yields
 \begin{multline*}
 \Tr_{(g_1, g_2)} \bigl( (-1)^F F (e^{-t\Delta_{E_1 \boxtimes E_2}} - P^{E_1 \boxtimes E_2} \bigr) \\=\Tr_{g_1} \bigl( (-1)^{F_1} F_1 (e^{-t\Delta_{E_1}} - P^{E_1}) \bigr) \chi_{g_2}(E_2) + \Tr_{g_2} \bigl( (-1)^{F_2} F_2 (e^{-t\Delta_{E_2}} - P^{E_2}) \bigr) \chi_{g_1}(E_1).
 \end{multline*}
 So
\beq{eq prod form log}
-2\log T_{(g_1, g_2)}(\nabla^{E_1 \boxtimes E_2}) = -2 \chi_{g_2}(E_2)  \log T_{g_1}(\nabla^{E_1})
-2 \chi_{g_1}(E_1)  \log T_{g_2}(\nabla^{E_2}),
\eeq
and \eqref{eq prod form} follows.

In case (b), we note that
 the respective Laplacians are nonnegative, so
\beq{eq kernels tensor}
\ker(\Delta_{E_1 \boxtimes E_2}) =
\ker(\Delta_{E_1}) \otimes \ker(\Delta_{E_2}).
\eeq
This implies that
\beq{eq ker prod}
 \Tr_{(g_1, g_2)}(P^{E_1 \boxtimes E_2}_{p})  = 
 \sum_{p_1+p_2= p}
 \Tr_{g_1}(P^{E_1}_{p_1})  \Tr_{g_2}(P^{E_2}_{p_2})  = 0.
\eeq
Now \eqref{eq curly T product} directly implies \eqref{eq prod form log}, and hence \eqref{eq prod form} .
%
%
%
%
%
%
\end{proof}


\begin{remark}
For a nontrivial result, one of the manifolds $M_j$ in proposition \ref{prop prod form} should be odd-dimensional so that its equivariant analytic torsion is may be nontrivial, by proposition \ref{prop vanish even}, and the other should be even-dimensional so that its $g_j$-Euler characteristic is not zero, by example \ref{ex chi g odd zero} (for this it is also necessary that $g_j$ lies in a compact subgroup of $G_j$). So only one factor in \eqref{eq prod form} is different from $1$.

The condition that $\Tr_{g_j}(P^{E_j})=0$ in part (b) of proposition \ref{prop prod form} does not imply that $\chi_{g_j}(E_j)=0$; see remark \ref{rem chi g P}.
\end{remark}

\subsection{A relation with Ray--Singer analytic torsion} \label{sec computing torsion}

In computations of equivariant analytic torsion, it can be useful to combine the two terms in \eqref{eq def deloc torsion} in the following way. 
\begin{lemma}\label{lem Tg sigma}
Suppose the conditions in definition \ref{def g torsion} hold. 
If $\Real(\sigma)>0$, then the expression $T_g(\nabla^E, \sigma)$, defined by
\beq{eq Tg sigma}
-2\log T_g(\nabla^E, \sigma) :=   \left. \frac{d}{ds}\right|_{s=0} \frac{1}{\Gamma(s)} \int_0^{
\infty} t^{s-1} e^{-\sigma t} \cT_g(t) \, dt
\eeq
converges. It defines a meromorphic function of $\sigma$ that is regular at $\sigma = 0$, and
\beq{eq Tg zero}
T_g(\nabla^E, 0)=T_g(\nabla^E).
\eeq
\end{lemma} 
\begin{proof}
The expression
\[
I_1(s, \sigma) := 
 \frac{1}{\Gamma(s)} \int_0^{1} t^{s-1} e^{-\sigma t} \cT_g(t) \, dt
\]
is well-defined for all $s$ with large enough real part, and all $\sigma \in \C$. By lemmas \ref{lem small t large l} and  \ref{lem small t small l},  it has a meromorphic extension to a neighbourhood of $s=0$, which is regular at $s=0$. It follows $ \left. \frac{\partial}{\partial s}\right|_{s=0}I_1(s, \sigma)$ is holomorphic in $\sigma$, and its value at $\sigma=0$ is the first term on the right hand side of \eqref{eq def deloc torsion}.

The expression
\[
I_2(s, \sigma) := \frac{1}{\Gamma(s)} \int_1^{
\infty} t^{s-1} e^{-\sigma t} \cT_g(t) \, dt
\]
converges whenever 
$\Real(\sigma)\geq 0$ and $s<\alpha_g$. So whenever $\Real(\sigma)\geq 0$, the equality $\dds \frac{t^s}{\Gamma(s)}=1$ implies that 
\[
 \left. \frac{\partial}{\partial s}\right|_{s=0}I_2(s, \sigma) =  \int_1^{
\infty} t^{-1} e^{-\sigma t} \cT_g(t) \, dt.
 \]
 Evaluating this at $\sigma=0$ gives the second term on the right hand side of \eqref{eq def deloc torsion}.
\end{proof}

\begin{remark}
Lemma \ref{lem Tg sigma} can be used to generalise the definition of equivariant analytic torsion. Whenever the expression $ T_g(\nabla^E,  \sigma) $ is well-defined if $\Real(\sigma)$ is large enough, and this defines a meromorphic function of $\sigma$ that is regular at $\sigma = 0$, then we can take \eqref{eq Tg zero} as the definition of $T_g(\nabla^E)$.
\end{remark}


Lemma \ref{lem Tg sigma} allows us to relate equivariant analytic torsion to classical Ray--Singer analytic torsion, in the case of fundamental groups acting on universal covers. 
Let $N$ be a compact Riemannian manifold. Let $M$ be its universal cover, acted on by $G = \Gamma = \pi_1(N)$. Let $E_N \to N$ be the flat vector bundle associated to a finite-dimensional unitary representation $\rho\colon \Gamma \to \U(r)$. Let $\nabla^{E_N}$ be the flat connection $\nabla^{E_N} := d \otimes 1$ on $\Gamma^{\infty}(E_N) = (C^{\infty}(M) \otimes \C^r)^{\Gamma}$.
Let $\nabla^E$ be its lift to $E = M \times \C^r$. We write $\Delta_{E_N} := (\nabla^{E_N})^*\nabla^{E_N} + \nabla^{E_N}(\nabla^{E_N})^*$
\begin{proposition}\label{prop Tsigma TN}
Suppose that $\ker(\Delta_E)$ and $\ker(\Delta_{E_N})$ are trivial. 
The expression 
\[
\prod_{(\gamma)} T_{\gamma}(\nabla^E, \sigma)
\]
converges if $\Real(\sigma)$ is large enough, where the product is over the conjugacy classes in $\Gamma$. It has a meromorphic extension to a neighbourhood of $\sigma = 0$. It is regular at $\sigma = 0$, and its value there is the Ray--Singer analytic torsion of $N$ associated to $\rho$.
\end{proposition}
\begin{proof}
Let $Y \subset M$ be a compact fundamental domain for the action by $\Gamma$. 
By the \v{S}varc-Milnor lemma, and compactness of $Y$, there are $b_1, b_2>0$ such that for all $m \in Y$ and $x \in \Gamma$,
\beq{eq SM lemma 1}
d(xm,m) \geq b_1 l(x) - b_2,
\eeq
for a word length function $l$ on $\Gamma$. For any $a,b,c \geq 0$ such that $a\geq b-c$, we have
\[
2(a^2 + c^2) \geq (a+c)^2 \geq b^2, 
\]
so $a^2 \geq b^2/2 - c^2$. Applying this to \eqref{eq SM lemma 1}, we obtain
\beq{eq SM lemma}
d(xm,m)^2 \geq \frac{b_1^2}{2} l(x)^2 - b_2^2.
\eeq

Let $\kappa_t$ be the Schwartz kernel of $(-1)^F F e^{-t\Delta_E}$. 
Let $C$ and $a_3$ be as in proposition \ref{prop heat large t bdd}.
%
%
%
%
 By \eqref{eq decay kappa gen}  and \eqref{eq SM lemma},  we have for all $m \in Y$, $x \in \Gamma$ and $t \geq 1$,
\beq{eq Tsigma 1}
\left| \int_{Y} \tr(x\kappa_t(x^{-1}m,m))\, dm \right| \leq
\vol(Y) C e^{a_3b_2^2} e^{-a_3b_1^2 l(x)^2/2t}.
\eeq
By this inequality and lemma 4.8 in \cite{HWWII},  the sum and integral
\[
\sum_{x \in \Gamma} \int_1^{\infty} t^{-1}e^{-\sigma t} \left| \int_{Y} \tr(x\kappa_t(x^{-1}m,m))\, dm\right| \, dt
\]
converges if $\Real(\sigma)$ is large enough. Here one uses that the finitely generated group $\Gamma$ has at most exponential growth. By the Fubini--Tonelli theorem, this implies that
\begin{multline*}
\sum_{(\gamma)} \int_1^{\infty} t^{-1}e^{-\sigma t} \sum_{x \in (\gamma)} \int_{Y} \tr(x\kappa_t(x^{-1}m,m))\, dm \, dt\\
=
 \int_1^{\infty} t^{-1}e^{-\sigma t} \sum_{(\gamma)}\sum_{x \in (\gamma)} \int_{Y} \tr(x\kappa_t(x^{-1}m,m))\, dm \, dt.
\end{multline*}
Furthermore, these sums and integrals converge absolutely, and therefore \eqref{eq der ts Gamma} implies that
\begin{multline}\label{eq Tgamma 1}
\sum_{(\gamma)} \left. \frac{d}{ds}\right|_{s=0} \frac{1}{\Gamma(s)}  \int_1^{\infty} t^{s-1}e^{-\sigma t} \sum_{x \in (\gamma)} \int_{Y} \tr(x\kappa_t(x^{-1}m,m))\, dm \, dt\\
=
 \left. \frac{d}{ds}\right|_{s=0} \frac{1}{\Gamma(s)} 
 \int_1^{\infty} t^{s-1}e^{-\sigma t} \sum_{(\gamma)}\sum_{x \in (\gamma)} \int_{Y} \tr(x\kappa_t(x^{-1}m,m))\, dm \, dt.
\end{multline}

We now turn to integrals over  $t$ from $0$ to $1$.
Consider the proper function $\varphi\colon \Gamma \to \R$ defined by
\[
\varphi(x) = \frac{b_1^2l(x)^2}{4} - b_2^2.
\]
The set
\[
S:= \{x \in \Gamma; \varphi(x) \leq 1\}
\]
is finite. 
Let $a_1, a_2, a_3>0$ be such that \eqref{eq decay kappa} holds for all $t \in (0,1]$. Then \eqref{eq SM lemma} implies that
for all $x \in \Gamma$, and all $t \in (0,1]$,
\[
\left| \int_{Y} \tr(x\kappa_t(x^{-1}m,m))\, dm \right| \leq \vol(Y)
a_1t^{-a_2}  e^{-a_3 b_1^2 l(x)^2/4}    e^{-a_3 \varphi(x)/t}.
\]
So
\beq{eq int 01}
\begin{split}
\sum_{x \in \Gamma \setminus S} \int_0^1 t^{-1}e^{-\sigma t} 
\left| \int_{Y} \tr(x\kappa_t(x^{-1}m,m))\, dm \right| \, dt &\leq
\vol(Y)
a_1 \sum_{x \in \Gamma \setminus S} e^{-a_3b_1^2 l(x)^2/4}
\int_0^1 t^{-a_2-1}e^{-\sigma t} e^{-a_3 \varphi(x)/t}\, dt\\
&\leq
\vol(Y)
a_1 \Bigl(\sum_{x \in \Gamma \setminus S} e^{-a_3b_1^2 l(x)^2/4}\Bigr)
\Bigl(
\int_0^1 t^{-a_2-1}e^{-\sigma t} e^{-a_3 /t}\, dt\Bigr),
\end{split}
\eeq
because $\varphi(x)>1$ for all $x \not \in S$. The integral over $t$ on the right hand side of \eqref{eq int 01} converges because the integrand is bounded, and 
the sum 
\[
\sum_{x \in \Gamma } e^{-a_3 b_1^2 l(x)^2/4}
\]
converges, because $\Gamma$ has at most exponential growth. We find that the left hand side of \eqref{eq int 01} converges. By the Fubini--Tonelli theorem,
\begin{multline*}
\sum_{(\gamma)}
 \int_0^1 t^{-1}e^{-\sigma t} 
\sum_{x \in (\gamma) \setminus S}
 \int_{Y} \tr(x\kappa_t(x^{-1}m,m))\, dm \, dt  \\
 = 
 \int_0^1 t^{-1}e^{-\sigma t}  \sum_{(\gamma)}
\sum_{x \in (\gamma) \setminus S}
 \int_{Y} \tr(x\kappa_t(x^{-1}m,m))\, dm \, dt.
\end{multline*}
It follows that the integrals and sums converge absolutely, so 
\begin{multline}\label{eq Tgamma 2}
\sum_{(\gamma)}  \left. \frac{d}{ds}\right|_{s=0} \frac{1}{\Gamma(s)} 
 \int_0^1 t^{s-1}e^{-\sigma t} 
\sum_{x \in (\gamma) \setminus S}
 \int_{Y} \tr(x\kappa_t(x^{-1}m,m))\, dm \, dt  \\
 =  \left. \frac{d}{ds}\right|_{s=0} \frac{1}{\Gamma(s)} 
 \int_0^1 t^{s-1}e^{-\sigma t}  \sum_{(\gamma)}
\sum_{x \in (\gamma) \setminus S}
 \int_{Y} \tr(x\kappa_t(x^{-1}m,m))\, dm \, dt.
\end{multline}

Finally, because the set $S$ is finite, we trivially have
\begin{multline}\label{eq Tgamma 3}
\sum_{(\gamma)}  \left. \frac{d}{ds}\right|_{s=0} \frac{1}{\Gamma(s)} 
 \int_0^1 t^{s-1}e^{-\sigma t} 
\sum_{x \in (\gamma) \cap S}
 \int_{Y} \tr(x\kappa_t(x^{-1}m,m))\, dm \, dt  \\
 =  \left. \frac{d}{ds}\right|_{s=0} \frac{1}{\Gamma(s)} 
 \int_0^1 t^{s-1}e^{-\sigma t}  \sum_{(\gamma)}
\sum_{x \in (\gamma) \cap S}
 \int_{Y} \tr(x\kappa_t(x^{-1}m,m))\, dm \, dt.
\end{multline}
note that there are finitely many conjugacy classes  with nonempty intersections with $S$.
The equalities \eqref{eq Tgamma 1}, \eqref{eq Tgamma 2} and  \eqref{eq Tgamma 3} imply that if $\Real(\sigma)$ is large enough,
\begin{multline}\label{eq Tg T}
\sum_{(\gamma)}  \left. \frac{d}{ds}\right|_{s=0} \frac{1}{\Gamma(s)} 
 \int_0^{\infty} t^{s-1}e^{-\sigma t} 
\sum_{x \in (\gamma)}
 \int_{Y} \tr(x\kappa_t(x^{-1}m,m))\, dm \, dt  \\
 =  \left. \frac{d}{ds}\right|_{s=0} \frac{1}{\Gamma(s)} 
 \int_0^{\infty} t^{s-1}e^{-\sigma t}  \sum_{(\gamma)}
\sum_{x \in (\gamma)}
 \int_{Y} \tr(x\kappa_t(x^{-1}m,m))\, dm \, dt.
\end{multline}
Now for every $t$, 
\beq{eq Tr Y}
 \sum_{(\gamma)}
\sum_{x \in (\gamma)} \int_Y
\tr(x\kappa_t(x^{-1}m,m)))\, dm
\eeq
converges absolutely by compactness of $Y$, Gaussian off-diagonal decay estimates for $\kappa_t$ (see example \ref{ex kappa t heat op}), the \v{S}varc-Milnor lemma, and the fact that $\Gamma$ has at most exponential growth. So \eqref{eq Tr Y} equals
\[
\int_Y \tr \left(
 \sum_{(\gamma)}
\sum_{x \in (\gamma)} 
x\kappa_t(x^{-1}m,m)) \right)\, dm = \Tr(-1)^F F  e^{-t\Delta_{E_N}}.
\]
It therefore follows from classical theory (see proposition 9.35 in \cite{BGV}, or \cite{Seeley67}, and definition 1.6 in \cite{RS71}) that the right hand side of \eqref{eq Tg T} extends meromorphically to $\sigma$ in the complex plane, is regular at $\sigma  =0$, and its value there is $-2$ times the logarithm of the Ray--Singer analytic torsion of $\nabla^{E_N}$, up to a sign convention.
\end{proof}





\section{One-dimensional examples}\label{sec one dim ex}

In this section, we compute equivariant analytic torsion for the actions by the real line on itself and by  the circle on itself. We also illustrate metric independence in these cases.

We use lemma \ref{lem Tg sigma} in the computations in this section, and in particular the notation \eqref{eq Tg sigma}.

We will use the following fact from calculus in a few places (see (5.28) in \cite{Shen18}).
\begin{lemma}\label{lem int exp}
For all $a>0, b\geq 0$,
\[
\int_{0}^{\infty} e^{-a/t-bt} t^{-3/2}\, dt = \sqrt{\frac{\pi}{a}}e^{-2\sqrt{ab}}.
\]
\end{lemma}
We will also use the equality
\beq{eq der Gamma quotient}
 \left. \frac{d}{ds} \right|_{s=0}\frac{\Gamma(s-1/2)}{\Gamma(s)} = -2\sqrt{\pi}.
\eeq

\subsection{The real line}\label{sec ex R}

Let $R>0$.
Consider the manifold $M = \R$ with $R^2$ times the Euclidean Riemannian metric. We consider the action by $\R$ on itself by addition.

Let $\alpha \in i\R$.
Consider the connection
\[
\nabla^E := d+ \alpha dx
\]
on the trivial line bundle $E \to M$. 
\begin{proposition}\label{prop torsion R}
For nonzero $g \in \R$,
\[
 T_g(\nabla^E) = \exp\left( \frac{e^{\alpha g}}{2 |g| } \right),
\]
and
\[
T_0(\nabla^E) = 1.
\]
\end{proposition}
In this example,  $\Delta_E$ has trivial $L^2$-kernel. 
We find that, as expected, the equivariant analytic torsion of $\nabla^E$ is independent of $R$. Furthermore, we see that $ T_g(\nabla^E)$ does not depend continuously on $g$.

The adjoint of $\nabla^E$ is given by
\[
(\nabla^E)^* f dx= \frac{1}{R^2}(-\frac{d}{dx}+ \bar \alpha)f,
\]
for $f \in C^{\infty}(\R)$. 
The factor $1/R^2$ comes from the fact that the metric on $\Bigwedge^1 T^*M$ is $1/R^2$ times the (standard) metric on $\Bigwedge^0 T^*M$. Hence the Laplacian $\Delta_{E}^p$ acting on $p$-forms is
\beq{eq Delta R}
\Delta_E^p = \frac{-1}{R^2}  \left(\frac{\partial}{\partial x} - \bar \alpha \right)\left(\frac{\partial }{\partial x} + \alpha\right).
\eeq

Let $\kappa^{\R, \alpha}_t$ be the heat kernel of $\Delta_E^0$ or $\Delta_E^1$. 
\begin{lemma}\label{lem heat alpha}
We have 
\beq{eq kappa alpha}
\kappa^{\R, \alpha}_t(x,y) = \frac{R}{\sqrt{4\pi t}} e^{-R^2(x-y)^2/4t -\alpha(x-y)}.
\eeq
\end{lemma}
\begin{proof}
Let $p=0$ or $p=1$.
Using the Fourier transform, we find that for all $x,y \in \R$, 
\[
\kappa^{\R, \alpha}_t(x,y) = \int_{\R}e^{ i(x-y)\xi} e^{t(i \xi - \bar \alpha)( i \xi + \alpha)/R^2}\, d\xi.
\]
As $\bar\alpha = -\alpha$, we can rewrite this into a standard Gaussian integral by substitution, and obtain \eqref{eq kappa alpha}. 
\end{proof}

Let $g \in \R$.
\begin{lemma}\label{lem torsion ZR}
If $g\not=0$, then for all $\sigma>0$
\beq{eq Th ZR}
\log T_g(\nabla^E, \sigma) = \frac{e^{\alpha g - R|g|\sqrt{\sigma}}}{2 |g| }.
\eeq
\end{lemma}
\begin{proof}
Let $\psi \in C^{\infty}_c(\R)$ be such that $\int_{\R} \psi(x)\, dx =1$. Then
by lemma \ref{lem heat alpha},
\[
\Tr_g(e^{-t \Delta_E^p}) = \int_{\R} \psi(x+g) \frac{R}{\sqrt{4\pi t}} e^{-R^2g^2/4t + \alpha g}\, dx = \frac{R}{\sqrt{4\pi t}} e^{-R^2 g^2/4t + \alpha g}.
\]
Hence, for $t>0$,
\beq{eq cal T ZR}
\cT_g(t) = \frac{-R}{\sqrt{4\pi t}} e^{-R^2 g^2/4t +\alpha g},
\eeq
so for all $s,\sigma>0$, if $g \not=0$, 
\[
\begin{split}
-2\log T_g(\nabla^E, \sigma) &= 
\frac{-R e^{\alpha g}}{\sqrt{4\pi}} 
\left. \frac{d}{ds}\right|_{s=0} \frac{1}{\Gamma(s)} \int_0^{
\infty} t^{s-3/2}e^{-\sigma t} e^{-R^2 g^2/4t }   \, dt\\
&= 
\frac{-Re^{\alpha g}}{\sqrt{4\pi}} 
\int_0^{
\infty} t^{-3/2}e^{-\sigma t} e^{-R^2 g^2/4t }   \, dt.
\end{split}
\]
By lemma \ref{lem int exp}, this equals $-2$ times the right hand side of \eqref{eq Th ZR}.
\end{proof}

We now consider $g=0$.
\begin{lemma}\label{lem torsion ZR 0}
We have
\[
\log T_0(\nabla^E, \sigma)=-R\frac{\sqrt{\sigma}}{2}.
\]
\end{lemma}
\begin{proof}
$\Delta^p_E$ has trivial kernel, so $P_p = 0$. By lemma \ref{lem heat alpha},
\[
\Tr_e(e^{-t\Delta_E^p}) = \int_0^1 \frac{R}{\sqrt{4\pi t}}\, dx =  \frac{R}{\sqrt{4\pi t}},
\]
so
\[
-2\log T_0(\nabla^E, \sigma) =\frac{-R}{\sqrt{4\pi}} \left. \frac{d}{ds}\right|_{s=0} \frac{1}{\Gamma(s)} \int_0^{
\infty} t^{s-3/2}e^{-\sigma t}  \, dt.
\]
By the substitution $t \mapsto \sigma t$, and \eqref{eq der Gamma quotient}, this equals $R\sqrt{\sigma}$.
%
\end{proof}

Proposition \ref{prop torsion R} follows from lemmas \ref{lem Tg sigma}, \ref{lem torsion ZR} and \ref{lem torsion ZR 0}.

\subsection{The circle}\label{sec ex S1}

Let $R>0$. 
We view the circle $M$ of circumference $R$ as $\R/\Z$ with metric $R^2 dx^2$. We consider the action by the circle $G = \R/\Z$ on $M$ by rotations.

We consider the connection $\nabla^E = d+\alpha dx$ on the trivial line bundle  $E \to \R/\Z$ (with the trivial action by $\R/\Z$ on fibres), for $\alpha \in i\R$. Then the kernel of $\Delta_E$ is trivial if and only if 
 $\alpha \not\in 2\pi i \Z$. 
 In this subsection, we suppose that  $\alpha \not\in 2\pi i \Z$, 
so that equivariant analytic torsion is independent of $R$, as we will see explicitly. In subsection \ref{sec ex S1 alpha 0}, we 
comment on the case $\alpha \in 2 \pi i \Z$, and find that analytic torsion does depend on $R$.

\begin{proposition}\label{prop torsion line}
Consider the group element $g=r+\Z \in \R/\Z$, for $r \in \R$. If $r \not \in \Z$, then
\[
T_g(\nabla^E) = \exp \left(\frac{1}{2}  \sum_{n \in \Z}  \frac{{1}}{|n-r|} e^{- \alpha (n-r)} \right).
\]
If $r \in \Z$, i.e.\ $g=e$, then
\beq{eq torsion line e}
T_e(\nabla^E)=
 \left( - (2\sinh(\alpha/2))^2\right)^{-1/2}. 
\eeq
\end{proposition}
The case $g=e$ of Proposition \ref{prop torsion line}, i.e.\ the equality \eqref{eq torsion line e}, is example 1.7 in Shen's chapter in \cite{Shen21}.

 \begin{lemma}\label{lem torsion line sigma}
 If $\Real(\sigma)>0$ and $r \not \in \Z$, then 
 \beq{eq torsion line sigma}
 2\log T_g(\nabla^E, \sigma) =  \sum_{n \in \Z}  \frac{{1}}{|n-r|} e^{-R |n-r|\sqrt{\sigma}- \alpha (n-r)}.
 \eeq	
 \end{lemma}
 \begin{proof}
 By lemma \ref{lem heat alpha}, the heat kernel $\kappa^{\R/\Z, \alpha}_t$ of $\Delta_E^0$ or  $\Delta_E^1$ is given by
 \beq{eq kappa S1}
\kappa^{\R/\Z, \alpha}_t(x+\Z, y+\Z) = \sum_{n \in \Z} 
 \frac{R}{\sqrt{4\pi t}} e^{-R^2(x+n-y)^2/4t -\alpha(x+n-y)}.
 \eeq
 For $r \in \R$ we have, with $g = r+\Z$, 
 \[
 \Tr_g(\kappa^{\R/\Z, \alpha}_t) =  \frac{R}{\sqrt{4\pi t}} 
 \sum_{n \in \Z} 
e^{-R^2(n-r)^2/4t -\alpha(n-r)}.
 \]
 We deduced this from \eqref{eq kappa S1}, but we remark that it is also possible to use the fact that $ \Tr_g(e^{-t \Delta_E}) =  \Tr(g\circ e^{-t \Delta_E})$ because $G$ is compact, together with properties of the eigenfunctions $x+\Z \mapsto e^{2\pi i n x}$, and Poisson's summation formula.
 
We find that
\[
\cT_g(t) = -  \frac{R}{\sqrt{4\pi t}} \sum_{n \in \Z} e^{-R^2(n-r)^2/4t - \alpha (n-r)}.
\]

If $r \not \in \Z$, then $(n-r)^2>0$ for all $n \in \Z$.
Hence lemma \ref{lem int exp} implies that for all $\sigma \geq 0$,
\[
 \int_0^{\infty} t^{-3/2}e^{-\sigma t}e^{-R^2(n-r)^2/4t } \, dt = \frac{\sqrt{4\pi}}{R|n-r|} e^{-R |n-r|\sqrt{\sigma}},
\]
so the right hand side of \eqref{eq torsion line sigma} equals
\[
\sum_{n \in \Z}  \int_0^{\infty} t^{-1} e^{-\sigma t}  \frac{R}{\sqrt{4 \pi t}}e^{-R^2(n-r)^2/4t - \alpha (n-r)} \, dt.
\]
This converges absolutely, and equals
\beq{eq Th circle 1}
 \frac{R}{\sqrt{4\pi}} \int_0^{\infty} t^{-3/2}e^{-\sigma t}\sum_{n \in \Z} e^{-R^2(n-r)^2/4t - \alpha (n-r)} \, dt =  -\int_0^{\infty} t^{-1}e^{-\sigma t}\cT_g(t) \, dt =  2\log T_g(\nabla^E, \sigma).
\eeq
In particular, we find that the $\Gamma$-function regularisation near $t=0$ is not needed.
%
\end{proof}

\begin{lemma}\label{lem torsion line sigma e}
 If $\Real(\sigma)>0$ and $r  \in \Z$, then 
 \beq{eq torsion line sigma e}
 2\log T_e(\nabla^E, \sigma) = -R \sqrt{\sigma}-\log(1-e^{-R\sqrt{\sigma}-\alpha})-\log(1-e^{-R\sqrt{\sigma}+\alpha}).
 \eeq
\end{lemma}
\begin{proof}
If $r\in \Z$, then  \eqref{eq Th circle 1} becomes
\beq{eq Te circle}
2\log T_e(\nabla^E, \sigma) = 
 \frac{R}{\sqrt{4\pi}} \int_0^{\infty} t^{-3/2}e^{-\sigma t}\sum_{n \in \Z} e^{-R^2n^2/4t - \alpha  n} \, dt.
\eeq
%
 The term for $n=0$ equals
 \[
 \frac{R}{\sqrt{4\pi}} \int_0^{\infty} t^{-3/2}e^{-\sigma t} \, dt =  \frac{R}{\sqrt{4\pi}} \sqrt{\sigma}\Gamma(-1/2) = -R \sqrt{\sigma}.
 \]
For $n \not=0$, lemma \ref{lem int exp} implies that
 \[
 \int_0^{\infty} t^{-3/2}e^{-\sigma t}e^{-R^2n^2/4t} \, dt = 
 \frac{\sqrt{4\pi}}{R|n|} e^{-R|n|\sqrt{\sigma}}.
 \]
 Therefore,
 \[
 \sum_{n =1}^{\infty} 
  \frac{R}{\sqrt{4\pi}} \int_0^{\infty} t^{-3/2}e^{-\sigma t}e^{-R^2n^2/4t - \alpha  n} \, dt = \sum_{n=1}^{\infty} \frac{e^{n(-R\sqrt{\sigma}-\alpha})}{n} = -\log(1-e^{-R\sqrt{\sigma}-\alpha}).
 \]
 Hence
  \[
 \sum_{n =-\infty}^{-1} 
  \frac{R}{\sqrt{4\pi}} \int_0^{\infty} t^{-3/2}e^{-\sigma t}e^{-R^2n^2/4t - \alpha  n} \, dt = \sum_{n=1}^{\infty} \frac{e^{n(-R\sqrt{\sigma}+\alpha})}{n} = -\log(1-e^{-R\sqrt{\sigma}+\alpha}).
 \]
 We find that \eqref{eq Te circle} converges absolutely if $\Real(\sigma)>0$, and equals
the right hand side of \eqref{eq torsion line sigma e}.
 \end{proof}
 
 Proposition \ref{prop torsion line} follows from lemmas \ref{lem Tg sigma}, \ref{lem torsion line sigma} and \ref{lem torsion line sigma e}. To deduce \eqref{eq torsion line e} from lemma \ref{lem torsion line sigma e}, we use the equality
 \[
  \left( (1-e^{-\alpha})(1-e^{\alpha})\right)^{-1/2} =   \left( - (2\sinh(\alpha/2))^2\right)^{-1/2}. 
 \]

\begin{remark}\label{rem connection R}
The connection $\nabla^E = d+\alpha dx$ used in this subsection arises naturally in the context of flat vector bundles corresponding to representations of the fundamental group of a manifold. Indeed, consider the trivial line bundle $E' = \R \times \C \to \R$, but with the action by $\R$ given by
\beq{eq isom line bdls}
g \cdot (x,z) = (x+g, e^{g\alpha} z),
\eeq
for $g,x \in \R$ and $z \in \C$. Then there is an isomorphism of $\R$-equivariant vector bundles $E \xrightarrow{\cong} E'$, with $E$ as in subsection \ref{sec ex R}, mapping $(x,z) \in E$ to $(x, e^{\alpha x}z)$. Under this isomorphism, the connection $\nabla^E$ on $E$ corresponds to the trivial connection $d$ on $E'$. 

The quotient of the total space of $E'$ by $\Z$ defines the line bundle on $\R/\Z$ associated to the representation $n \mapsto e^{-n\alpha}$ of $\pi_1(\R/\Z) = \Z$. The natural flat connection $d$ on this bundle corresponds to the connection $\nabla^E$ in the current subsection via the $\R/\Z$-equivariant isomorphism with the trivial line bundle induced by \eqref{eq isom line bdls}.
\end{remark}

\subsection{The circle for $\alpha = 0$}\label{sec ex S1 alpha 0}

We continue with the example of the circle with circumference $R$, but now take $\alpha = 0$.
Then both $\ker(\Delta_0)$ and $\ker(\Delta_1)$ are the one-dimensional trivial representation of $G$.
Therefore equivariant analytic torsion may depend on $R$, because the condition that $\ker (\Delta_E) = \{0\}$ in corollary \ref{cor metric indep}, and in remark \ref{rem rescaling}), is not satisfied. We verify this for $g = e =  0+ \Z$.
\begin{lemma}\label{lem circle alpha zero}
If $\alpha = 0$, then
\[
T_e(\R/\Z) = 1/R.
\]
\end{lemma}
\begin{proof}
If $g$ is the trivial element, then for $p=0,1$, 
\[
\Tr_e(e^{-t\Delta_E^p} - P_p) = \Tr(e^{-t\Delta_E^p} - P_p) = \sum_{n \in \Z \setminus \{0\}} e^{-t(2\pi n/R)^2},
\]
and
\beq{eq torsion S1 alpha zero}
2\log T_e(\R/\Z) = \left. \frac{d}{ds}\right|_{s=0} \frac{1}{\Gamma(s)} \int_0^{
\infty} t^{s-1} \sum_{n \in \Z \setminus \{0\}} e^{-t(2\pi n/R)^2}  \, dt.
\eeq
For every $n \in \Z\setminus \{0\}$, and $s\in \C$, a substitution shows that
\[
\int_0^{
\infty} t^{s-1} e^{-t(2\pi n/R)^2}  \, dt = \left( \frac{R}{2\pi n}\right)^{2s} \Gamma(s).
\]
If $\Real(s)>1/2$, then
\[
 \sum_{n \in \Z \setminus \{0\}}  \left( \frac{R}{2\pi n}\right)^{2s} = 2
  \sum_{n=1}^{\infty}  \left( \frac{R}{2\pi n}\right)^{2s} = 2\left( \frac{R}{2\pi}\right)^{2s} \zeta(2s),
  \]
 where $\zeta$ is the Riemann $\zeta$-function, and the sum converges absolutely. So by the Fubini--Tonelli theorem,  the right hand side of \eqref{eq torsion S1 alpha zero} equals
 \[
2 \left. \frac{d}{ds}\right|_{s=0}\left( \frac{R}{2\pi}\right)^{2s} \zeta(2s).
 \]
Using the equalities $\zeta(0) = -1/2$ and $\zeta'(0) = -\log(2\pi)/2$, we find that this equals $-2 \log(R)$.
\end{proof}
\begin{remark}
In section X of \cite{LR91}, Lott and Rothenberg investigate a possible extension of their equivariant analytic torsion from finite groups to compact, but infinite groups.
On page 474 of \cite{LR91}, they show in a related example on the circle that  such an extension depends on the  radius of the circle. In their situation, the  cohomology groups whose vanishing would imply metric independence of analytic torsion are nonzero, as they are in the example in this subsection. 
\end{remark}

\section{Hyperbolic $3$-space}\label{hyperbolic_comp}

Let $G$ be a connected, real semisimple Lie group, and $K<G$ maximal compact. Equivariant analytic torsion of $G/K$, for a semisimple element $g \in G$, contains information, for example, on locally symmetric spaces $\Gamma\backslash G/K$, for $\Gamma<G$ discrete and cocompact. If $\Gamma$ has torsion, then such a space is an orbifold in general.

By remark 7.9.2 in \cite{BismutHypo}, $T_g(G/K)$ can only be nonzero if $G = \SO_0(p,q)$ for $pq>1$ odd, or $G = \SL(3, \R)$. In this section, we consider 
 $G = \SO_0(2n+1, 1)$, for $n \in \Z_{\geq 0}$, and  its maximal compact subgroup $K = \SO(2n+1)$. Then $G/K$ is $(2n+1)$-dimensional real hyperbolic space. 
 
For nontrivial hyperbolic elements $g \in G$, Lott used theorem 2 in \cite{Fried86} to compute $T_g^{\Gamma}(G/K)$ in proposition 9 in \cite{Lott99}, where $\Gamma<G$ is a torsion-free, discrete, cocompact subgroup of $G$ containing  $g$. Knowing $T_g^{\Gamma}(G/K)$ for all $g \in \Gamma$ allows one to reconstruct the geodesic length spectrum of the hyperbolic manifold $\Gamma \backslash G/K$. 
Lott computed $T_e^{\Gamma}(G/K)$ if $n=1$, in proposition 16 in \cite{Lott92c}. In section 5 of \cite{Mathai92}, it was shown that for general semisimple $G$ and $g=e$, the logarithm of $T_e^{\Gamma}(G/K)$ is proportional to $\vol(\Gamma \backslash G/K)$, with a proportionality constant independent of $\Gamma$.

 In this section, we use Bismut's trace formula \cite{BismutHypo} for orbital integrals of heat operators to compute $T_g^{G}(G/K)$ for regular elliptic elements $g$, if $n=1$. For general $n$, every elliptic element of $G$ is conjugate to an element of the maximal torus $T = \SO(2)^n <K$.
\begin{proposition}\label{prop H3}
Suppose that $n=1$. Let $g \in \SO(2)<K$ be counter-clockwise rotation over an angle $x\not=0$. 
   Then $\alpha_g>0$, and
\beq{eq Tg H3 g}
T_g^{G}(G/K) = \exp \left( \frac{-1}{8\sin(x/2)^2} \right).
\eeq
\end{proposition}
\begin{remark}
The proof of proposition \ref{prop H3} in this section can be generalised to arbitrary $n$ in principle, but the computations and resulting expression for equivariant analytic torsion become much more complicated.
\end{remark}
\begin{remark}
By proposition 16 in \cite{Lott92c} and lemma \ref{lem TrG TrH}, we have for $n=1$,
\[
T_e^G(G/K) = e^{1/6\pi}\not=0.
\]
So by \eqref{eq Tg H3 g}, the function $g\mapsto T_g^{G}(G/K)$ is not continuous, as in the examples in section \ref{sec one dim ex}.
\end{remark}

\begin{remark}
As pointed out in the proof of lemma \ref{lem NS hyp}, 
the kernel of $\Delta_E$ is trivial in the current setting, so the projection $P$ plays no role in the computation of equivariant analytic torsion.
\end{remark}

\subsection{Bismut's formula}

We denote Lie algebras of Lie groups by the corresponding lower-case gothic letters. Let 
\[
\kp = \left\{  \begin{pmatrix}
0 & \cdots & 0 & b_1 \\
\vdots & & \vdots & \vdots \\
0 & \cdots & 0& b_{2n+1} \\
 b_1 & \cdots & b_{2n+1} & 0
\end{pmatrix}; b_1, \ldots, b_{2n+1} \in \R
 \right\} \subset \kg.
\]
Let $\ka \subset \kp$ be the one-dimensional subspace of elements with $b_1 = \cdots = b_{2n} = 0$. Then $\kg = \kk \oplus \kp$ is a Cartan decomposition, $\kh = \kt \oplus \ka \subset \kg$ is a Cartan subalgebra, and $H = Z_G(\kh)<G$ is a Cartan subgroup. From now on, we fix an element $g \in \SO(2)^n<T$ that is regular, i.e.\ $Z_G(g) = H$. 

Let $C^{\kk}$ be the Casimir element of $\kk$. For an operator $A$ on a finite-dimensional vector space $V$, we write $\tr(A|_V)$ and $\det(A|_V)$ for its trace and determinant, to emphasise the space that $A$ acts on. Traces and determinants of restrictions to invariant subspaces can be denoted in this way. As in (7.8.7) in \cite{BismutHypo}, set
\beq{eq def beta}
\beta := -\frac{1}{48} \tr(C^{\kk}|_{\kk})-\frac{1}{16} \tr(C^{\kk}|_{\kp}).
\eeq
Let $J_g$ be the function on $\kt$ defined in general by (5.5.5) in \cite{BismutHypo}; this simplifies  in the current setting as in lemma \ref{lem Jg Bismut} below. Let $F$ be the number operator on $\bigwedge \kp^*$. Then Bismut's trace formula, theorem 6.1.1 in \cite{BismutHypo}, has the following consequence.
\begin{theorem}[Bismut] \label{thm bismut}
For all $t>0$, 
\begin{multline}\label{eq bismut}
\cT_g(t/2) = \frac{e^{-\beta t}}{(2\pi t)^{(n+1)/2}}  
\int_{\kt} J_g(Y) \tr \left(  (-1)^F \left. \left(F-\frac{2n+1}{2} \right) e^{-i\ad(Y)} \Ad(g) \right|_{\bigwedge \kp^*} \right) e^{-\|Y\|^2/2t}\, dY.
\end{multline}
\end{theorem}
\begin{proof}
This is a special case of (7.9.3) in \cite{BismutHypo}, where  a missing factor $e^{-\|Y\|^2/2t}$ in the integrand seems to be a typo. Furthermore, we have used the facts that, because $g$ is regular, 
its centraliser in $\kk$ is $\kt$ and its centraliser in $\kp$ is $\ka$.  Also, we have omitted a factor $\vol \left(K/T \right)$ between Bismut's orbital integral trace and our $\Tr_g$, which may be normalised to $1$; compare also definition 4.2 and remark 4.3 in \cite{Shen18}.
\end{proof}

In the rest of this section, we compute the various components of \eqref{eq bismut} (some in general, some for $n=1$), and use the resulting expression for $\cT_g(t/2) $ to prove proposition \ref{prop H3}.

\subsection{The function $J_{g}$} \label{sec Jg}

\begin{lemma}\label{lem Jg Bismut}
For all $Y \in \kt$,
\beq{eq Jh}
J_g(Y) = \left( \frac{\det(1-e^{-i\ad(Y)} \Ad(g)|_{\kk/\kt})}{\det(1- \Ad(g)|_{\kg/\kh})\det(1-e^{-i\ad(Y)} \Ad(g)|_{\kp/\ka})} \right)^{1/2},
\eeq
where the choice of square root is as at the start of section 5.5 of \cite{BismutHypo}.
\end{lemma}
\begin{proof}
The expression for $J_g(Y)$ in theorem 5.5.1 in \cite{BismutHypo} reduces to  \eqref{eq Jh} in the current setting, because the centraliser of $g$ in $\kg$ is $\kh$, its centraliser in $\kk$ is $\kt$ and its centraliser in $\kp$ is $\ka$. Furthermore, $\ad(Y)$ acts as zero on $\ka$ and $\kt$ for $Y \in \kt$, so the $\hat A$-factors in (5.5.5) in \cite{BismutHypo} equal $1$.
\end{proof}

For $j=1, \ldots, n$, let $H_j \in \kt$ be the matrix with the $2\times 2$-block $\begin{pmatrix}0&-1\\1&0\end{pmatrix}$ as the $j$th $2\times 2$-block on the diagonal. Define $e_j \in i\kt^*$ by $\langle e_j, H_k\rangle = i\delta_{jk}$. Write $g = \exp(X)$, with $X = x_1H_1+\cdots +x_n H_n$. Fix $Y  = y_1 H_1 + \cdots + y_n H_n \in \kt$ for the rest of this subsection. We denote complexifications of real vector spaces by superscripts $\C$. 
\begin{lemma}\label{lem det kt}
We have
\begin{multline} \label{eq det kt}
\det(1- e^{-i\ad(Y)}\Ad(g)|_{\kk/\kt}) = \\
(-1)^n\prod_{1 \leq j<k \leq n} \left(  2\sinh \left(\frac{i(x_j+x_k) + y_j+y_k}{2} \right) \right)^2
\left(  2\sinh \left(\frac{i(x_j-x_k) + y_j-y_k}{2} \right) \right)^2\\
\cdot \prod_{j=1}^n \left(  2\sinh \left(\frac{ix_j+y_j}{2} \right) \right)^2.
\end{multline}
\end{lemma}
\begin{proof}
The set of roots of $(\kk^{\C}, \kt^{\C})$ is
\[
R(\kk^{\C}, \kt^{\C}) =
\{\pm e_j \pm e_k; 1\leq j<k \leq n\} \cup \{\pm e_j, 1\leq j \leq n\}.
\]
See example 2 on page 84 of \cite{Knapp96}.
Now $e^{-i \ad(Y)} \Ad(g) = e^{\ad(X)-i\ad(Y)}$, and the eigenvalues of $\ad(X)-i\ad(Y)$ on $\kk^{\C}/\kt^{\C}$ are
\begin{multline*}
\{\langle \alpha, X\rangle-i \langle \alpha, Y \rangle; \alpha \in R(\kk^{\C}, \kt^{\C})\} \\= 
\{\pm (ix_j+y_j) \pm (ix_k + y_k); 1\leq j<k \leq n\} \cup \{\pm (ix_j +y_j), 1\leq j \leq n\}.
\end{multline*}
So the left hand side of \eqref{eq det kt} equals
\begin{multline*}
{\det}_{\C}(1- e^{\ad(X)-i\ad(Y)}|_{\kk^{\C}/\kt^{\C}})  \\= 
\prod_{1 \leq j<k \leq n} 
\left(1-e^{i(x_j+x_k)+y_j+y_k} \right)\left(1-e^{-(i(x_j+x_k)+y_j+y_k)} \right)\\
\left(1-e^{i(x_j-x_k)+y_j-y_k} \right)\left(1-e^{-(i(x_j-x_k)+y_j-y_k)} \right)\\
\cdot \prod_{j=1}^n 
\left(1-e^{ix_j+y_j} \right)\left(1-e^{-(ix_j+y_j)} \right).
\end{multline*}
The claim now follows from the fact that for all $z \in \C$,
\[
(1-e^z)(1-e^{-z}) = -(2\sinh(z/2))^2.
\]
\end{proof}

\begin{lemma}\label{lem det pa}
We have
\beq{eq det pa}
\det(1- e^{-i\ad(Y)} \Ad(g)|_{\kp/\ka}) = 
(-1)^n
 \prod_{j=1}^n \left(  2\sinh \left(\frac{ix_j+y_j}{2} \right) \right)^2.
\eeq
\end{lemma}
\begin{proof}
For $j=1, \ldots, n$, let
\[
E_{\pm e_j} = \begin{pmatrix}
0 & \cdots & & & \cdots & 0 \\
\vdots & & & & & \vdots \\
 & & & & & \pm 1 \\
 & & & & & -i \\
\vdots  & & & & & \vdots \\
0 & \cdots & \pm 1 & -i & \cdots & 0
\end{pmatrix} \quad \in \kp^{\C},
\]
where the two  nonzero entries in the last column are in rows $2j-1$ and $2j$, and the two nonzero entries in the bottom row are in columns $2j-1$ and $2j$. These $2n$ matrices together form a basis of $\kp^{\C}/\ka^{\C}$, and for all $j$, $k$, 
\[
[H_j, E_{\pm e_k}] = \pm i \delta_{jk}E_{\pm e_k}.
\]
Hence the eigenvalues of $\ad(X)-i \ad(Y)$ on $\kp^{\C}/\ka^{\C}$ are
\[
\{\pm (ix_j+y_j); 1\leq j \leq n\},
\]
and the left hand side of \eqref{eq det pa} equals
\[
{\det}_{\C}(1- e^{-i\ad(Y)} \Ad(g)|_{\kp^{\C}/\ka^{\C}}) =\prod_{j=1}^n 
\left(1-e^{ix_j+y_j} \right)\left(1-e^{-(ix_j+y_j)} \right).
\]
This equals the right hand side of \eqref{eq det pa}, as in the proof of lemma \ref{lem det kt}.
\end{proof}

\begin{lemma} \label{lem Jg}
We have
\[
J_g(Y) = 
\prod_{1 \leq j<k \leq n} 
\frac{  \sinh \left(\frac{i(x_j+x_k) + y_j+y_k}{2} \right)}{\sinh \left(\frac{i(x_j+x_k)}{2} \right)}
\frac{  \sinh \left(\frac{i(x_j-x_k) + y_j-y_k}{2} \right)}{\sinh \left(\frac{i(x_j-x_k)}{2} \right)}
\cdot \prod_{j=1}^n \frac{1}{\left(  2\sinh \left({ix_j}/{2} \right) \right)^{2}}.
\]
\end{lemma}
\begin{proof}
It follows from lemmas \ref{lem det kt} and \ref{lem det pa}, with $Y=0$, and $T$-invariance of $\kk/\kt$ and $\kp/\ka$, that
\begin{multline*}
\det(1-\Ad(g)|_{\kg/\kh}) = \det(1-\Ad(g)|_{\kk/\kt}) \det(1-\Ad(g)|_{\kp/\ka})\\
=  \prod_{1 \leq j<k \leq n} \left(  2\sinh \left(\frac{i(x_j+x_k)}{2} \right) \right)^2
\left(  2\sinh \left(\frac{i(x_j-x_k)}{2} \right) \right)^2
\cdot \prod_{j=1}^n \left(  2\sinh \left({ix_j}/{2} \right) \right)^4.
\end{multline*}
By lemmas \ref{lem Jg Bismut}, \ref{lem det kt} and \ref{lem det pa}, this implies the claim.
\end{proof}

\begin{corollary} \label{cor Jg n=1}
If $n=1$, then 
\[
J_g(Y) = 
 \frac{-1}{  4\sin ({x_1}/{2})^{2}}
\]
is independent of $Y$.
\end{corollary}

\subsection{The function $\cT_g$}

In the case $n=1$, we compute $\beta$ and the supertrace in the integrand of \eqref{eq bismut}, and use corollary \ref{cor Jg n=1} to compute the right hand side of \eqref{eq bismut}. We start with a result that applies for general $n$.
\begin{lemma} \label{lem str zero}
For all $Y \in \kt$,
\[
\tr( (-1)^Fe^{-i \ad(Y)} \Ad(g)|_{\bigwedge \kp^*}) = 0.
\]
\end{lemma}
\begin{proof}
The left hand side equals
\[
\det(1-e^{-i \ad(Y)} \Ad(g)|_{\kp^*}),
\]
which equals zero because $e^{-i \ad(Y)} \Ad(g)$ acts as the identity on $\ka$.
\end{proof}

For the rest of this section, we assume that $n=1$. We write $x = x_1$, i.e. $g = \exp(xH_1)$.  
\begin{lemma} \label{lem beta}
If $n=1$, then  $\beta = 1/4$.
\end{lemma}
\begin{proof}
The Killing form on $\mathfrak{so}(3)$ is $B(X_1, X_2) = \tr(X_1 X_2)$. So
\[
X_1 = \frac{1}{\sqrt{2}} \begin{pmatrix}
0&-1&0\\
1&0&0\\
0&0&0
\end{pmatrix}, \quad
X_2 = \frac{1}{\sqrt{2}} \begin{pmatrix}
0&0&-1\\
0&0&0\\
1&0&0
\end{pmatrix} \quad \text{and} \quad
X_3 = \frac{1}{\sqrt{2}} \begin{pmatrix}
0&0&0\\
0&0&-1\\
0&1&0
\end{pmatrix}
\]
form an orthonormal basis of $\mathfrak{so}(3)$. By computing commutators, we find that, with respect to the basis $\{X_1, X_2, X_3\}$, the adjoint actions by $X_1$, $X_2$ and $X_3$ on $\kk$ have the matrices
\[
\begin{split}
\mat(\ad(X_1)|_{\kk}) &= X_3;\\
\mat(\ad(X_2)|_{\kk}) &= -X_2;\\
\mat(\ad(X_3)|_{\kk}) &= X_1.
\end{split}
\]
On the right hand sides, the matrices $X_j$ are viewed as acting on $\R^3 \cong \kso(3)$, not as elements of $\kk$. The matrix of the adjoint action by $C^{\kk}$ on $\kk$ is the sum of the squares of these matrices, which is minus the $3\times 3$ identity matrix. So
\beq{eq tr Ckk}
\tr(C^{\kk}|_{\kk}) = -3.
\eeq

Consider the  basis of $\kp$ consisting of
\[
Y_1 = \begin{pmatrix}
0&0&0 & 1\\
0&0&0&0\\
0&0&0&0\\
1&0&0&0
\end{pmatrix}, \quad
Y_2 =  \begin{pmatrix}
0&0&0 & 0\\
0&0&0&1\\
0&0&0&0\\
0&1&0&0
\end{pmatrix} \quad \text{and} \quad
Y_3 =  \begin{pmatrix}
0&0&0 & 0\\
0&0&0&0\\
0&0&0&1\\
0&0&1&0
\end{pmatrix}.
\]
Via the embedding $\kso(3) \hookrightarrow \kso(3,1)$ in the top left corner, the adjoint actions of $X_1$, $X_2$ and $X_3$ on $\kp$ have matrices $\mat(\ad(X_j|_{\kp})) = X_j$ with respect to the basis $\{Y_1, Y_2, Y_3\}$. Again, on the right, $X_j$ is viewed as a matrix acting on $\R^3 \cong \kp$. As before, the matrix of the adjoint action by $C^{\kk}$ on $\kp$ is the sum of the squares of these matrices, which is minus the $3\times 3$ identity matrix. Hence 
\beq{eq tr Ckp}
\tr(C^{\kk}|_{\kp}) = -3.
\eeq
The claim follows from the definition \eqref{eq def beta} of $\beta$, \eqref{eq tr Ckk} and \eqref{eq tr Ckp}.
\end{proof}

\begin{lemma}\label{lem str wedge p}
If $n=1$, then for all $Y = yH_1 \in \kt$, 
\beq{eq str n=1}
\tr\left( (-1)^FF e^{-i\ad(Y)} \Ad(g)|_{\bigwedge \kp^*} \right) = 4\sinh\left( \frac{ix+y}{2} \right)^2.
\eeq
\end{lemma}
\begin{proof}
As we saw in the proof of lemma \ref{lem det pa}, the eigenvalues of $e^{\ad(X)-i \ad(Y)}$ on $\kp^{\C}$ are $e^{\pm (ix+y)}$ and $1$. The eigenvalues on $(\kp^*)^{\C}$ are the same,
and $\bigwedge^2 \kp^* \cong \kp^*$ as representations of $T$. Finally, because elements of $T$ act on $\kp^*$ with determinant $1$, they act as the identity on $\bigwedge^3 \kp^*$. Knowing these eigenvalues, one computes that the left hand side of \eqref{eq str n=1} equals
\[
e^{ix+y} + e^{-(ix+y)}-2,
\]
which equals the right hand side.
\end{proof}

\begin{lemma}\label{lem Tgt n=1}
If $n=1$, then for all $t>0$, 
\[
\cT_g(t) = \frac{\cos(x) - e^{-t/2} }{4 \sqrt{2\pi t} \sin(x/2)^2}.
\]
\end{lemma}
\begin{proof}
Using corollary \ref{cor Jg n=1}  and lemmas \ref{lem str zero}, \ref{lem beta} and \ref{lem str wedge p} in theorem \ref{thm bismut}, we find that 
\[
\cT_g(t/2) = 
 - \frac{e^{- t/4}}{2\pi t} 
  \frac{1}{  \sin ({x}/{2})^{2}}
\int_{\R} \sinh\left( \frac{ix+y}{2} \right)^2 e^{-y^2/t}\, dy.
\]
Here we used that $\|H_1\|^2 = -\tr(H_1^2) = 2$, so $\|Y\|^2 = 2y$.

Using the equalities
\[
\sinh\left( \frac{ix+y}{2} \right)^2 = \frac{1}{4} (e^{ix+y} + e^{-(ix+y)})-\frac{1}{2}
\]
and
\[
\int_{\R} e^{-\frac{y^2}{t} \pm y}\, dy = \sqrt{\pi t} e^{t/4}, 
\]
one computes 
\[
\int_{\R} \sinh\left( \frac{ix+y}{2} \right)^2 e^{-y^2/t}\, dy = 
 \frac{\sqrt{\pi t}}{2}\left(e^{t/4}\cos(x) -1\right).
\]
\end{proof}

\subsection{Equivariant analytic torsion for $n=1$}

Lemma \ref{lem Tgt n=1} allows us to prove proposition \ref{prop H3}.

\begin{lemma} \label{lem Tg sigma hyp}
If $n=1$, then for all $\sigma \in \C$ with $\Real(\sigma)>0$, 
\beq{eq Tg sigma H3}
\left. 
\frac{d}{ds} \right|_{s=0} \frac{1}{\Gamma(s)} \int_0^{\infty} t^{s-1} e^{-\sigma t} \cT_g(t)\, dt = 
 \frac{\sqrt{\sigma+1/2} - \cos(x) \sqrt{\sigma} }{2 \sqrt{2} \sin(x/2)^2}.
\eeq
\end{lemma}
\begin{proof}
By lemma \ref{lem Tgt n=1} and a short computation, we have for all  $\sigma,s \in \C$ with $\Real(\sigma)>0$ and $\Real(s)>1/2$, 
\[
 \int_0^{\infty} t^{s-1} e^{-\sigma t} \cT_g(t)\, dt = 
 \frac{1 }{4 \sqrt{2\pi} \sin(x/2)^2}\Gamma(s-1/2) \bigl(\cos(x) \sigma^{1/2 - s} - (\sigma+1/2)^{1/2- s}\bigr).
\]
Extending this meromorphically in $s$, and using \eqref{eq der Gamma quotient} and
\[
\lim_{s \to 0} \frac{\Gamma(s-1/2)}{\Gamma(s)} = 0,
\]
we obtain \eqref{eq Tg sigma H3}.
\end{proof}

\begin{proof}[Proof of proposition \ref{prop H3}.]
By lemma \ref{lem Tgt n=1}, $\cT_g(t) = \cO(t^{-1/2})$, so the second condition in definition \ref{def g torsion} holds. 
It is implicit in theorem \ref{thm bismut} that the first condition holds, and we saw in the proof of lemma \ref{lem Tg sigma hyp} that the third condition holds (alternatively, the first and third conditions in  definition \ref{def g torsion} hold by example \ref{ex SM conn} and proposition \ref{prop conv small t}).
By lemma \ref{lem Tg sigma}, it follows that $-2\log(T_g(G/K))$ is the value at $\sigma = 0$ of the meromorphic extension of \eqref{eq Tg sigma H3}, then \eqref{eq Tg H3 g} follows.
\end{proof}


\appendix

\section{Large-time boundedness of heat kernels}\label{app heat decay large t}

In the proof of proposition \ref{prop Tsigma TN}, we used the fact that heat kernels satisfy an estimate of the form \ref{eq decay kappa gen}, with $F_1$ bounded as $t \to \infty$. This is proposition \ref{prop heat large t bdd}. This is well-known for the heat kernel of the scalar Laplacian, but we did not come across a proof for non-scalar Laplacians in the literature.  We give this proof here, by adapting arguments from the scalar case.

The first ingredient is a version of the maximum principle; for the scalar Laplacian this is theorem 12.1 in \cite{Grigoryan09}. 
\begin{lemma}\label{lem max princ}
Let $M$ be a Riemannian manifold, and $E_1, E_2 \to M$ Hermitian vector bundles. Let $D \colon \Gamma^{\infty}(E_1) \to \Gamma^{\infty}(E_2)$ be a first-order differential operator, whose principal symbol $\sigma_D$ satsfies
\[
\sigma_D(\xi)^*\sigma_D(\xi) = \|\xi\|^2
\]
for all $\xi \in T^*M$. Let $u \in \Gamma^{\infty}(M \times (0, \infty), E_1 \times (0, \infty))$, and suppose that $\frac{\partial}{\partial t}u + D^*D u = 0$. Let $f\colon M \times (0, \infty) \to \R$ be smooth, and such that $\frac{\partial f}{\partial t} + \frac{1}{2} \|df\|^2 \leq 0$. (Here $df$ is the derivative of $f$ in the $M$-component.) For $t>0$, write $u_t := u(\relbar, t) \in \Gamma^{\infty}(E_1)$ and $f_t := f(\relbar, t) \in C^{\infty}(M)$. 
Suppose that for all $t>0$, the sections $D u_t$, $D(e^{f_t}u_t)$,  $ e^{f_t/2}u_t$ and $\|df_t \| e^{f_t/2}u_t$ and lie in  $L^2(E_1)$. Then the function 
\[
J\colon t \mapsto \int_M \|u_t(m)\|^2 e^{f_t(m)}\, dm
\]
 is non-increasing.
\end{lemma}
\begin{proof}
The derivative of $J$ is
\beq{eq max 1}
J'(t) =
\frac{\partial}{\partial t} (u_t, e^{f_t}u_t)_{L^2(E_1)} = 
 -(D^*D u_t, e^{f_t}u_t)_{L^2(E_1)} + \left(u_t, \frac{\partial f_t}{\partial t} e^{f_t}u_t \right)_{L^2(E_1)} 
-( u_t, e^{f_t}D^*D u_t)_{L^2(E_1)}. 
\eeq
The first term on the right equals
\[
-(D u_t, e^{f_t}D u_t)_{L^2(E_1)} - (D u_t, \sigma_D(df_t)e^{f_t}u_t)_{L^2(E_1)}.
\]
Similarly, the third term on the right hand side of \eqref{eq max 1} equals
\[
-(e^{f_t} D u_t, D u_t)_{L^2(E_1)} - ( \sigma_D(df_t)e^{f_t}u_t, D u_t)_{L^2(E_1)}.
\]
The second term on the right hand side of \eqref{eq max 1} is at most equal to
\[
-\frac{1}{2}
 \left(u_t, \|df_t\|^2e^{f_t}u_t \right)_{L^2(E_1)} = -\frac{1}{2}
 \left(\sigma_D(df_t) u_t, e^{f_t} \sigma_D(df_t)u_t \right)_{L^2(E_1)} . 
\]
We find that the  right hand side of \eqref{eq max 1} is at most equal to
\begin{multline*}
-2(e^{f_t} D u_t, D u_t)_{L^2(E_1)} -2 \Real  ( \sigma_D(df_t)e^{f_t}u_t, D u_t)_{L^2(E_1)}-\frac{1}{2}
 \left(\sigma_D(df_t) u_t, e^{f_t} \sigma_D(df_t)u_t \right)_{L^2(E_1)} \\
  = -\frac{1}{2} \| e^{f_t/2} (\sigma_D(df_t) u_t + 2Du_t)\|_{L^2(E_1)}^2 \leq 0.
  \end{multline*}
\end{proof}
\begin{example}\label{ex max princ}
For all $a>0$ and $m_0 \in M$, the function $f(m,t) = a d(m_0, m)^2/t$ has the properties in lemma \ref{lem max princ}. \end{example}

We will use a version of proposition 5.1 in \cite{Grigoryan94}.
\begin{lemma}\label{lem Ea}
Let $(\kappa_t)_{t>0}$ be a family of smooth kernel operators on a Hermitian vector bundle on a Riemannian manifold $M$,  with the semigroup property. For $a,t>0$ and $m \in M$, write
\[
E_a(t, m) := \int_M \|\kappa_t(m,m')\|^2 e^{a d(m,m')^2/t}\, dm'.
\]
Then for all $a,t>0$ and all $m,m'$ for which $E_a(t/2, m)$ and $E_a(t/2, m')$ converge,
\beq{eq off diag est}
\|\kappa_t(m,m')\| \leq \sqrt{E_{a}(t/2, m) E_{a}(t/2, m')} e^{-ad(m,m')^2/2t}.
\eeq
\end{lemma}
\begin{proof}
By the triangle inequality, we have for all $m'' \in M$,
\[
d(m,m'')^2 + d(m'', m')^2 - \frac{1}{2}d(m,m')^2 \geq 0.
\]
So by the semigroup property,
\[
\begin{split}
\|\kappa_t(m,m')\| &\leq \int_M \|\kappa_{t/2}(m,m'')\|  \|\kappa_{t/2}(m'',m')\|\, dm''\\
& \leq  e^{-ad(m,m')^2/2t}  \int_M \|\kappa_{t/2}(m,m'')\|  e^{ad(m,m'')^2/t} \|\kappa_{t/2}(m'',m')\|  e^{ad(m'',m')^2/t}  \, dm''.
\end{split}
\]
By the Cauchy--Schwartz inequality, the right hand side is at most equal to the right hand side of \eqref{eq off diag est}.
\end{proof}

\begin{proof}[Proof of proposition \ref{prop heat large t bdd}.]
Let $\kappa_t$ be the heat kernel associated to $\Delta_E^p$.
As in example \ref{ex kappa t heat op}, proposition 4.2(1) in \cite{CGRS14} implies that for small enough $a>0$, the expression $E_a(t,m)$ converges for all $m \in M$. Furthermore, we have $E_a(t,m) = E_a(t,xm)$ for all $x \in G$, so  $E_a(t,m)$ is bounded in $m$ because the action is cocompact. Lemma \ref{lem max  princ} (with $D = \nabla^E + (\nabla^E)^*$), and example \ref{ex max princ} imply that $E_a(t,m)$ is also bounded in $t \geq 1$. (The conditions in  lemma \ref{lem max  princ} on square-integrability of the relevant sections hold for small enough $a$,  by Gaussian off-diagonal estimates for heat kernels for fixed $t$.) So the claim follows from lemma \ref{lem Ea}. 
\end{proof}

\bibliographystyle{plain}

\bibliography{mybib}

\end{document}